\documentclass{amsart}
\usepackage{graphicx} 
\usepackage[english]{babel}
\usepackage[utf8]{inputenc}

\usepackage{tabularx}

\usepackage{graphicx}
\usepackage{setspace}
\usepackage{subfigure}
\usepackage{wrapfig}
\usepackage[dvipsnames]{xcolor} 
\usepackage{appendix}

\usepackage{amsmath}
\usepackage{amsfonts}
\usepackage{amssymb}
\usepackage{amsthm}
\usepackage{dlfltxbcodetips}
\usepackage{mathtools}
\usepackage{stix} 

\numberwithin{equation}{section}

\newcommand*{\N}{\mathbb{N}}
\newcommand*{\Z}{\mathbb{Z}}
\newcommand*{\Q}{\mathbb{Q}}
\newcommand*{\R}{\mathbb{R}}

\DeclareMathOperator{\bs}{\backslash}

\DeclareMathOperator{\Krav}{Kr}
\DeclareMathOperator{\Cr}{Cr}
\DeclareMathOperator{\divi}{div}

\usepackage{thmtools}

\declaretheorem[
	name=Theorem,
	numberwithin=section
	]{thm}
\declaretheorem[
	name=Lemma,
	sibling=thm,
	]{lem}
\declaretheorem[
	name=Proposition,
	sibling=thm,
	]{prop}
\declaretheorem[
	name=Corollary,
	sibling=thm,
	]{cor}
 \declaretheorem[
	name=Conjecture,
	sibling=thm,
	]{conj}

\declaretheorem[
	name=Definition,
	style=definition,
	numbered=no,
	]{defin}

\declaretheorem[
	name=Notation,
	style=definition,
	numbered=no,
	]{nota}

\declaretheorem[
	name=Remark,
	style=remark,
	numbered=no
	]{rem}	
\declaretheorem[
	name=Example,
	style=remark,
	numbered=no
	]{exam}

\title{The 1/4-phenomenon of placement probabilities of tilings in the Aztec diamond}

\author{Marcus Schönfelder}
\thanks{{marcus.schoenfelder@univie.ac.at}. The author was supported by the Austrian Science Foundation FWF, grant 10.55776/F1002, in the framework of the Special Research Programm ``Discrete Random Structures: Enumeration and Scaling Limits".}

\date{May 2026}

\begin{document}

\begin{abstract}
    We consider domino tilings of the Aztec diamond. Using the \textit{Domino Shuffling} algorithm introduced by Elkies, Kuperberg, Larsen, and Propp (1992), we are able to generate domino tilings uniformly at random. In this paper, we investigate the probability of finding a domino at a specific position in such a random tiling. We prove that this placement probability is always equal to 
$1/4$ plus a rational function, whose shape depends on the location of the domino, multiplied by a position-independent factor that involves only the size of the diamond. This result leads to significantly more compact explicit counting formulas compared to previous findings. As a direct application, we derive explicit counting formulas for the domino tilings of Aztec diamonds with 
$2\times 2$-square holes at arbitrary positions.   
\noindent  \end{abstract}
\maketitle

\section{Introduction}
\subsection{Background and motivation}

The dimer model has risen to a central object in statistical physics. From the point of view of mathematics, it concerns itself with the probabilistic behaviour of random perfect matchings on graphs. A perfect matching of a simple planar graph $G=(V,E)$ is a subset of edges $\mu\subseteq E$ such that any vertex $v\in V$ is contained in exactly one edge of $\mu$ (and therefore matched with the other vertex in this edge). 
If the graphs are subgraphs of the square- or hexagonal lattice, perfect matchings, in the context of statistical physics also called \textit{dimer configurations}, are in bijection with tilings of regions related to the graph by duality. The most prominent examples are rhombus- and domino tilings. For the second, one wants to cover a region consisting of a union of unit squares with domino shaped tiles, i.e., $1\times 2$ rectangles. An example of this idea is given in Figure~\ref{domi1}. The connection with perfect matchings is visualised in Figure~\ref{domi2}.

\begin{figure}
\centering
\includegraphics[width=0.4\textwidth]{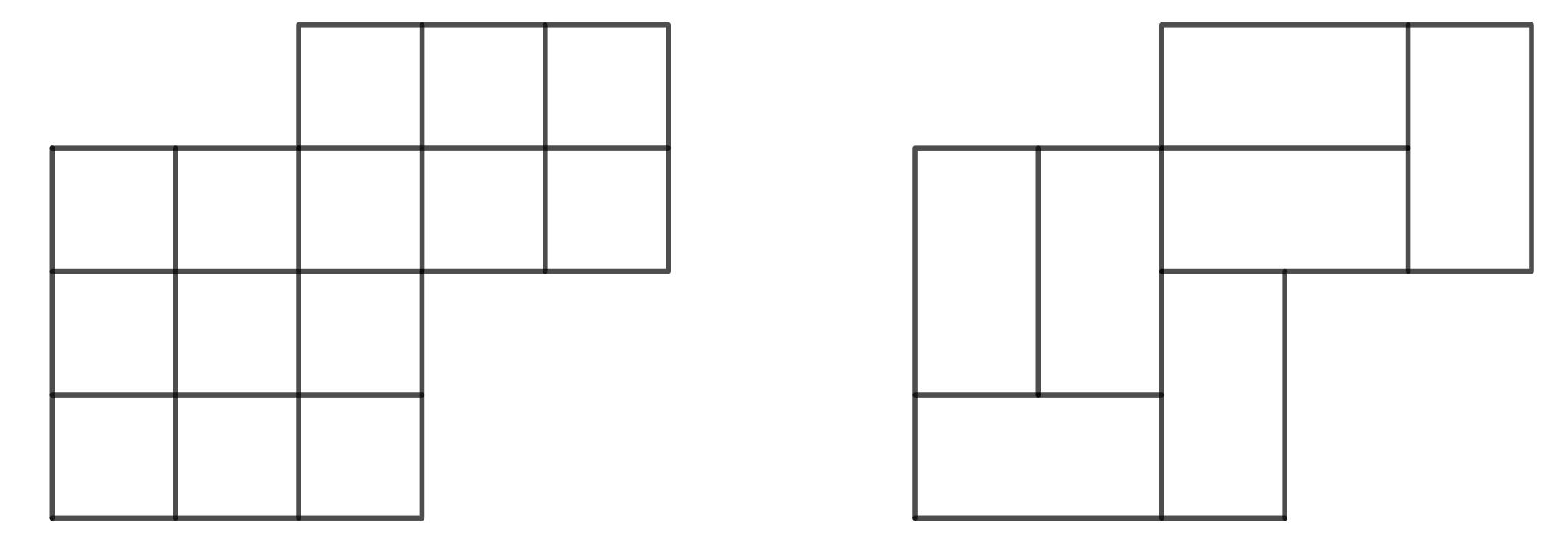}
\caption{A domino region left and one of its possible domino tilings on the right.}
\label{domi1}
\end{figure}

\begin{figure}
\centering
\includegraphics[width=0.2\textwidth]{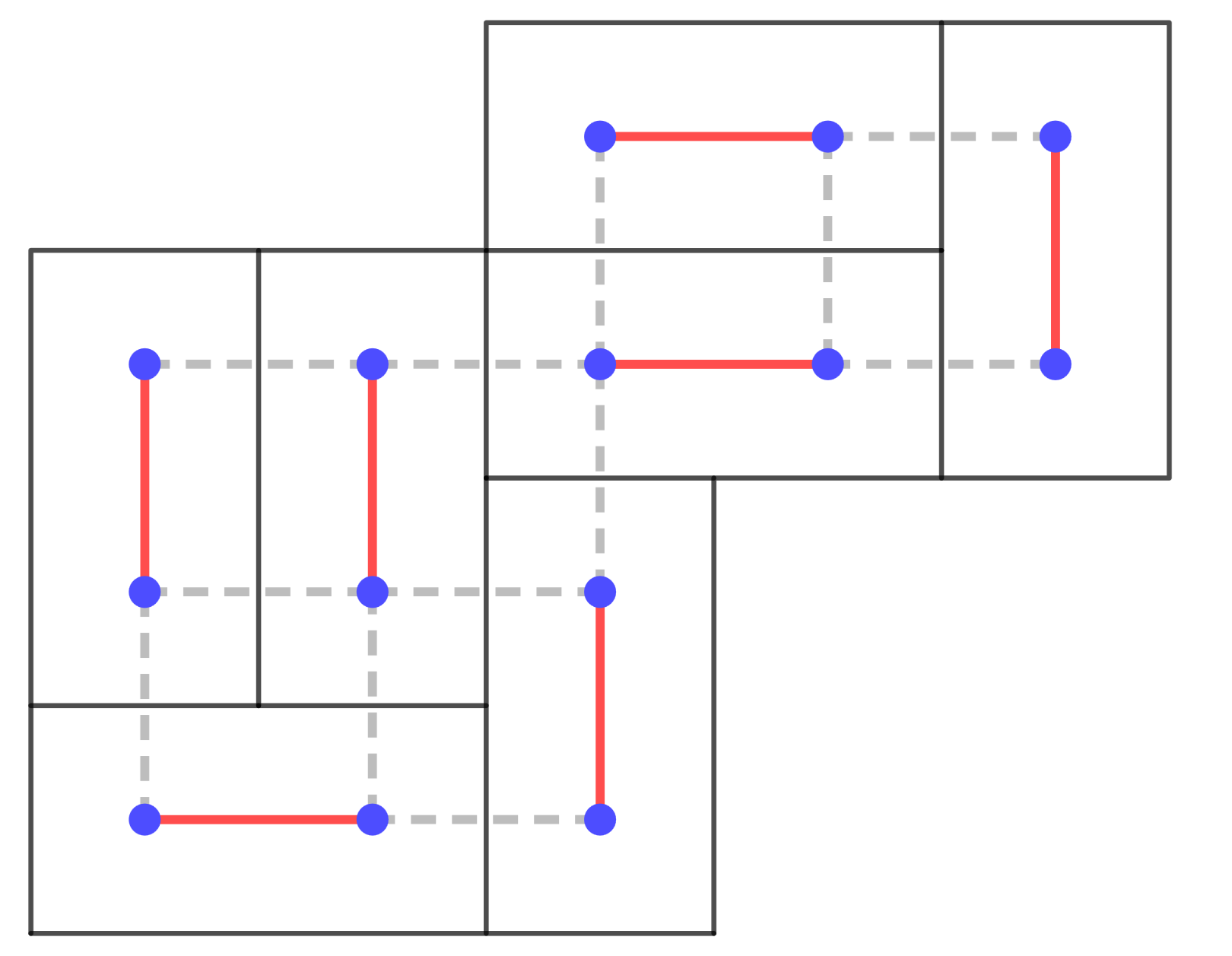}
\caption{The bijection between domino tilings and perfect matchings of the dual graph.}
\label{domi2}
\end{figure}

Advances in this theory are achieved hand in hand with the development of new enumerative techniques. First breakthroughs were accomplished by Temperley and Fisher~\cite{TemperleyFisher} and Kasteleyn~\cite{KSquare}, who simultaneously derived the enumeration formula counting domino tilings of the $n\times m$-rectangle. In the case of rhombus tilings it is MacMahon's Box Formula~\cite{MacMahon} which counts tilings of semiregular hexagons and connects them with plane partitions.
Later on, Kasteleyn~\cite{Kasteleyn} developed his method to express the number of perfect matchings of any planar graph as a Pfaffian or determinant, making the dimer model a \textit{solvable} statistical model. Meanwhile, people started to study the correlation functions on this probabilistic model. An early example would be the work of Fisher and Stephenson~\cite{FisherStephenson}. Later, Kenyon~\cite{Kenyon} combined the computation of correlations with Kasteleyn's theory. In their simplest form, so-called one-point correlation functions ask for the \textit{placement probability} telling the chance of observing a certain tile at a particular position in a random tiling. Equivalently, they encode for an edge the probability that it is contained in a random perfect matching. 

For the semiregular hexagon, closed formulas for the placement probabilities were developed by Fischer~\cite{FischerHexaLoch}, Gilmore~\cite{gilmoreInv} and in even more generality Petrov~\cite{Petrov}. However, the complexity of these formulas makes them very cumbersome for many applications. A conjecture by Krattenthaler~\cite{1/3}, which is related to our main result, promises vast possible simplifications of the structure of those formulas.

We study placement probabilities in the Aztec diamond. Here, explicit formulas were found by Helfgott~\cite{helfgott} and Cohn, Elkies and Propp~\cite{localstats}. The probabilities were expressed in terms of products of Kravchuk polynomials, whose complexity leads to similar problems as in the case of the semiregular hexagon. In this article we prove a strongly simplifying structural result on the shape of the formulas for those placement probabilities. 

\subsection{Set-up and the main result}
The Aztec diamond of size $n\in \mathbb{N}$ is a region in the plane consisting of unit squares aligned as depicted in Figure~\ref{Aztecdiamond}. Throughout this article we assume our diamonds to be coloured black and white in a chessboard pattern such that the left one of the top-most unit squares is black. It was shown by Elkies, Kuperberg, Larsen and Propp~\cite{Azdiam1,Azdiam2} that the number of domino tilings of the Aztec diamond of size $n$ is equal to
\[2^{(n+1)n/2}.\]

\begin{figure}
\centering
\includegraphics[width=0.6\textwidth]{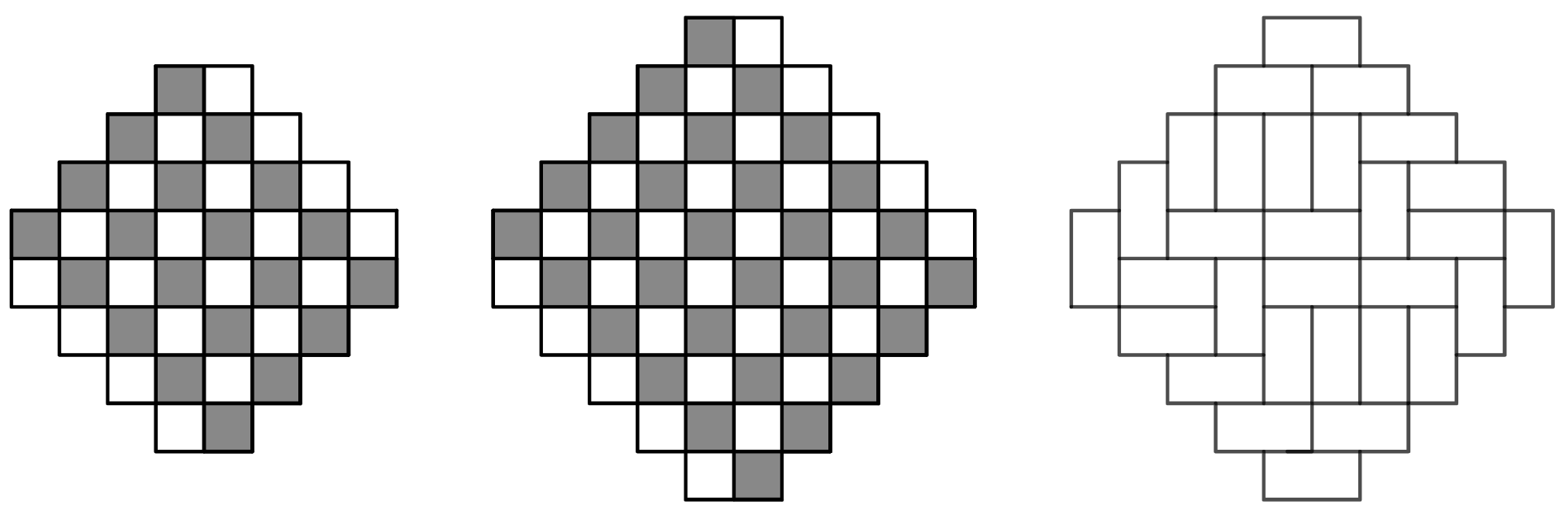}
\caption{The Aztec diamonds of sizes 4 (left) and 5 (centre) as well as a domino tiling of the Aztec diamond (right).}
\label{Aztecdiamond}
\end{figure}

Moreover, also in~\cite{Azdiam2} the researchers developed an algorithm called \textit{Domino Shuffling} which enables us to sample a domino tiling of the Aztec diamond uniformly at random. The implications of this algorithm make it possible to study single placement probabilities. First of all, introduce a coordinate system to our Aztec diamond by assuming that its centre is located at $(0,0)$ and all squares have side length one. We consider the quantity $\mathbb{P}(l,m;n)$ defined to be the probability that in a random tiling of the Aztec diamond of dimension $n$ we observe a horizontal domino tile whose left square is coloured black and whose centre of its lower side is located at $(l,m)$. The idea is visualised in Figure~\ref{para}.

\begin{figure}
\centering
\includegraphics[width=0.3\textwidth]{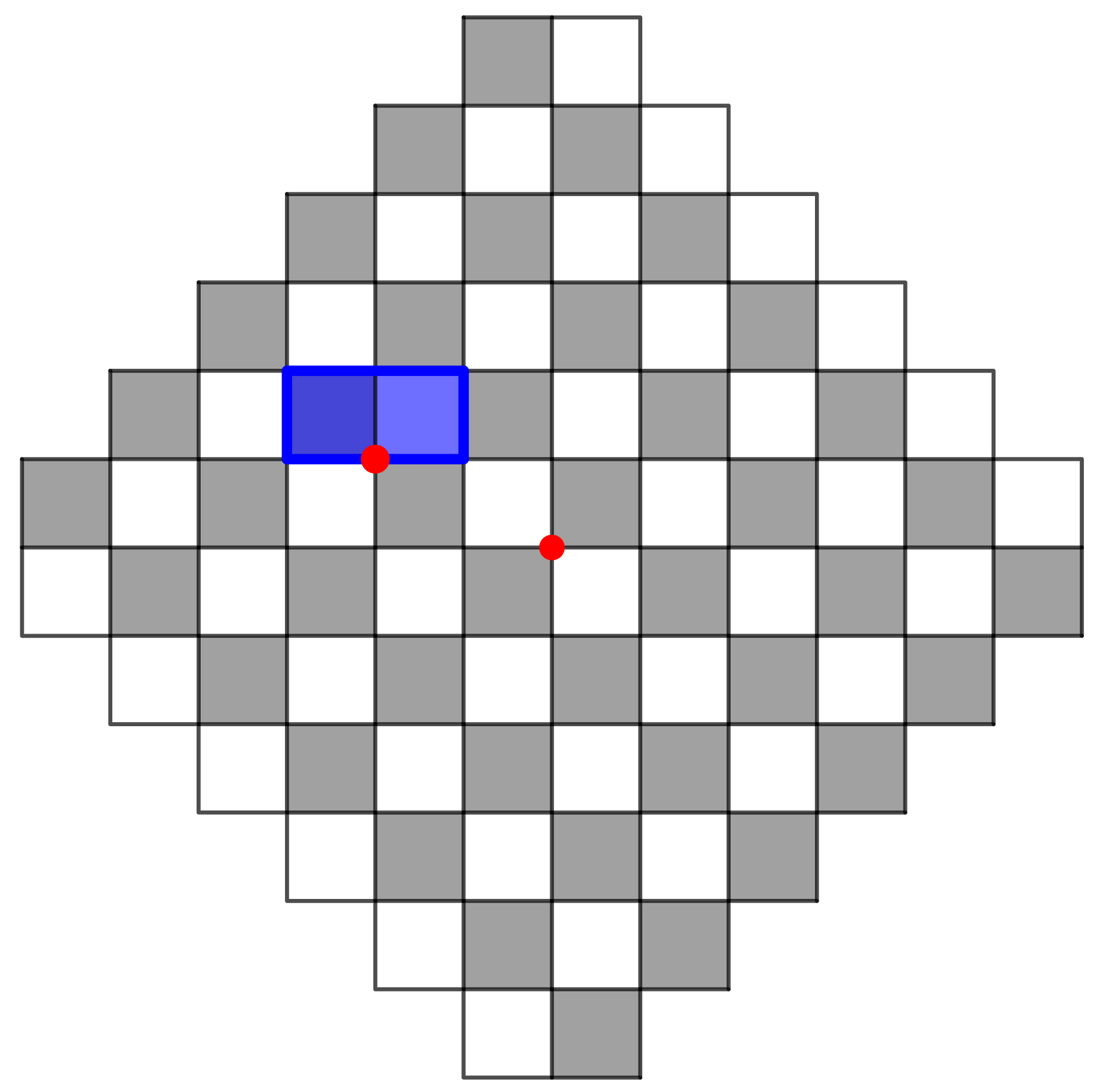}
\caption{The Aztec diamonds of sizes 6. In blue we marked the space for a horizontal domino tile at position $(-2,1)$. The probability, to observe a horizontal domino at this location in a random tiling is given by $\mathbb{P}(-2,1;6)$.}
\label{para}
\end{figure}

The reason why we have to distinguish between horizontal tiles whose left square is black or whose left square is white originates from the Domino Shuffling algorithm, where those two cases lead to different actions in the procedure. Naturally, once we know $\mathbb{P}(l,m;n)$ for all $l,m\in \mathbb{Z}$ and $n\in \mathbb{N}$ we know all possible placement probabilities by just rotating the diamond. Clearly, $\mathbb{P}(l,m;n)=0$ if $l+m\equiv n \mod 2$, since then the left square at position $(l,m)$ is simply white according to the rules of our colouring.

In~\cite{localstats}, Cohn, Elkies and Propp worked out the asymptotic behaviour of $\mathbb{P}(x,y;n)$ for $n\to \infty$ where $(x,y)=(l/n,m/n)$ encodes the relative position inside a \textit{normalised} Aztec diamond. We on the other hand are interested in the exact value of those probabilities for a fixed position in the diamond. We are able to prove the following statement about the shape of an explicit formula.

\begin{thm}[\sc{The 1/4-phenomenon}]\label{1/4}
    Consider the Aztec diamond of size $n\in \mathbb{N}$ with $\alpha\in\{0,1,2,3\}$ and $p\in \mathbb{N}$ such that $n=4p+\alpha$. Dye the Aztec diamond in a chessboard pattern black and white in a way such that the left of the topmost vertices is coloured black. Then, after sampling a tiling uniformly at random, the probability to find a horizontal domino tile at position $(l,m)$ whose left unit square is black equals 
    \[\mathbb{P}(l,m;4p+\alpha)=\begin{cases}
        \frac{1}{4}+2^{-4p-\alpha}\binom{2p-1}{p}^2f_{l,m,\alpha}(p), &\text{ if }\ l+m\equiv \alpha +1 \mod 2,\\
        0, &\text{ otherwise,}
    \end{cases} \]
    where $f_{l,m,\alpha}(p)$ is a rational function in $p$.
\end{thm}

\begin{rem}
    By the recursive arguments of the proof of Theorem~\ref{1/4} it is also possible to construct the rational functions $f_{l,m,\alpha}(p)$ explicitly for a given location $(l,m)\in \mathbb{Z}^2$. However, depending on the distance of $(l,m)$ to the origin this procedure may involve many computational steps. Thus, it may easier to find $f_{l,m,\alpha}(p)$ via polynomial interpolation. We state necessary bounds on the degrees in Theorem~\ref{shortadv} and in more detail also in Theorem~\ref{adv1/4}.
\end{rem}

The idea behind the formula in Theorem~\ref{1/4} is to expect an almost symmetric behaviour of the single domino placements. Given a unit square in the Aztec diamond it has to be covered by a domino tile whose second part is directed either north, east, south or west. Hence, there are exactly four possibilities to place a domino tile on said unit square. Now, Theorem~\ref{1/4} tells us that the probability to find a tile in a certain position is in fact $1/4$ plus \textit{something nice}, namely a universal factor times a rational function. Moreover, the theorem enables us to derive beautiful enumeration formulas for Aztec diamonds with dimer holes (i.e., domino shaped) just by finding the according rational function $f_{l,m,\alpha}(p)$. We give a couple of examples below.

\begin{exam}
    Here we list a couple of instances for the rational function $f_{l,m,\alpha}(p)$.
    \begin{align}
        f_{0,0,1}(p)=f_{1,1,1}(p)=f_{2,0,1}(p)&=2, \notag \\
        f_{1,-1,1}(p)=f_{0,-2,1}(p)&=-2, \notag \\
        f_{0,0,3}(p)=f_{1,-1,3}(p)=f_{0,-1,2}(p)=f_{1,0,2}&=0,\notag \\
        f_{0,-2,3}(p)&=-\frac{4(1+2p)^2}{(1+p)^2}, \notag \\
        f_{3,1,3}(p)&=\frac{8(1+2p)}{1+p},\notag\\
        f_{0,-4,1}(p)&=-\frac{2(5+6p+3p^2)}{(1+p)^2},\notag \\
        f_{0,-1,0}(p)&=-1, \notag\\
        f_{2,3,2}(p)&=\frac{4(-1+6p+8p^2)}{(1+p)(-1+2p)},\notag \\
        f_{4,-3,0}(p)&=\frac{1-3p-6p^2}{(1+p)(-1+2p)}. \notag
    \end{align}
    All these rational functions lead to nice expressions for placement probabilities or numbers of tilings of Aztec diamonds with dimer defects (i.e., tilings with a tile fixed at a certain position). However, the complexity of the rational function increases with the distance to the origin. For example, for the values $(l,m)=(7,6)$ and $\alpha=0$ we have
    \[f_{7,6,0}(p)=\frac{-90 + 441 p + 756 p^2 - 497 p^3 - 462 p^4 + 84 p^5 + 56 p^6}{(1 + p) (2 + p) (3 + p) (-5 + 2 p) (-3 + 2 p) (-1 + 2 p)}.\]
    All these results are proven in the sense that we translated the single steps of the proof of the main result into a recursive algorithm that computes the rational polynomials automatically.
\end{exam}

For large values of $l$ or $m$ the recursive algorithm for creating the rational functions $f_{l,m,\alpha}(p)$ becomes inefficient. One might therefore try to find $f_{l,m,\alpha}(p)$ via (polynomial) interpolation. The following result offers the necessary bounds on the polynomial degrees for this.  

\begin{thm}\label{shortadv}
    Let $m,l\in\Z$ and $\alpha\in\{0,1,2,3\}$ such that $l+m\equiv \alpha+1\mod 2$. Then there exist polynomials $Z_{l,m,\alpha}(p)$ and $D_{l,m,\alpha}(p)$ in $p$ with
    \[\mathbb{P}(l,m;4p+\alpha)=\frac{1}{4}+2^{-4p-\alpha}\binom{2p-1}{p}^2\frac{Z_{l,m,\alpha}(p)}{D_{l,m,\alpha}(p)}\]
    where $\deg(Z_{l,m,\alpha})\leq \deg(D_{l,m,\alpha})\leq \max(|l|,|m|)+\big\lfloor\min(|l|+3,|m|)/2\big\rfloor +1$ and $D_{l,m,\alpha}(p)$ is explicitly known.
\end{thm}

Theorem~\ref{shortadv} is presented in a more precise version as Theorem~\ref{adv1/4}. Moreover, we should mention that experiments suggest an even better bound for the degrees of the numerator and denominator polynomials of $f_{l,m,\alpha}(p)$. Namely, we expect the following to hold.
\begin{conj}
    Let $Z(l,m,\alpha;p)$ and $D(l,m,\alpha;p)$ be the minimal polynomials such that
    \[f_{l,m,\alpha}(p)=\frac{Z(l,m,\alpha;p)}{D(l,m,\alpha;p)}.\]
    Then $\deg(Z(l,m,\alpha;p))\leq\deg(D(l,m,\alpha;p))\leq \max(|l|,|m|).$
\end{conj}

In Corollary~\ref{divVA} and Corollary~\ref{advm=0} we verify the conjecture for positions $(l,0)$, $(0,m)$ and $(1,m)$ for arbitrary $l,m\in \Z$. However, it still remains open for more general locations.
\newline

This article is organised as follows. In Section~\ref{sectk} we gather a couple of tools and known results about placement probabilities in the Aztec diamond. We study the \textit{creation rate} $\Cr(l,m;n)$, a quantity that arises from the Domino Shuffling algorithm. As we will see, the creation rate describes, up to some shift of the size of the diamond, the difference between the placement probabilities of vertically neighbouring positions. In particular, if Theorem~\ref{1/4} holds for $\mathbb{P}(l,m;n)$ and $\mathbb{P}(l,m-1;n-1)$ then also the creation rate should grow in a similar manner. I.e., we would expect
\[\Cr(l,m;n) =2^{-4p-\alpha+1}\binom{2p-1}{p}^2g_{l,m,\alpha}(p)\]
for some rational function $g_{l,m,\alpha}(p)$. 

The idea is therefore, to prove the so-to-speak homogeneous version of Theorem \ref{1/4} for all creation rates $\Cr(l,m;n)$ to reduce the problem for $\mathbb{P}(l,m;n)$ to certain positions $(l',m')$ closer to the origin. Moreover, for the creation rate an explicit formula in terms of Kravchuk polynomials is known. 

In Section~\ref{Kravchukgrogo} we show a theorem telling that already Kravchuk polynomials grow in a way related to the expression in Theorem~\ref{1/4}. From this we derive the statement for the creation rates.

Finally, in Section~\ref{Mainresult} we present the proof of our main result. The three major steps in the proof are organised in seperate subsections. First we show, that it is enough to prove Theorem~\ref{1/4} for positions along the horizontal axis. Afterwards, we further reduce the problem to the case checking whether the statement of the theorem holds at the origin. Eventually, we compute the exact placement probabilities for $\mathbb{P}(0,0;n)$ which finishes the proof.

In Section~\ref{RemA} we make a remark about the asymptotic behaviour of the placement probabilities. These results have actually already been worked out by Cohn, Elkies and Propp~\cite{localstats}.

In Section~\ref{dimdef}, we apply Theorem~\ref{1/4} to obtain enumeration formulas counting domino tilings of Aztec diamonds with a $2\times2$-square hole at an arbitrary position. In particular, we compute the number of tilings of Aztec diamonds of size $4p$ and $4p+1$ with a central $2\times2$ square hole poked into them. This result serves as a complement to a theorem by Mihai Ciucu~\cite{Ciucu2} counting the tilings of holey Aztec diamonds of size $4p+2$ and $4p+3$.

In the appendix we have a closer look at the rational functions $f_{l,m,\alpha}(p)$. While it seems to be impossible to present a simple expression for $f_{l,m,\alpha}(p)$ at general positions, by a more thorough analysis of our recursive arguments we are indeed able to give explicit formulas for its denominator and also give linear degree bounds for the numerator polynomial. In particular, we prove Theorem~\ref{shortadv} as an improvement to Theorem~\ref{1/4}.
The proof follows a similar line of arguments as for our main result. However, this and related statements on degree bounds turn out to be far more technical without yielding further mathematical insights on the matter. That is the main reason these last results are covered in the appendix of this article.

\section{Relevant relations of placement probabilities} \label{sectk}

In this section we gather the central tools and known results we will need to prove Theorem \ref{1/4}. Our work is mainly built upon the article by Cohn, Elkies and Propp~\cite{localstats}. There the authors compute the asymptotic value of the placement probability for any relative position. In our case, we are interested in the \textit{exact} value of these probabilities at a \textit{fixed} position. However, we will make use of their observations concerning the \textit{net creation rates}, a quantity emerging again from the Domino Shuffling algorithm.

\begin{defin}[\cite{localstats}]
    The creation rate is defined as 
    \[\Cr(l,m;n):= 2(\mathbb{P}(l,m;n)-\mathbb{P}(l,m-1;n-1)).\]
\end{defin}
 What comes in handy is that we have an explicit formula for the creation rates in terms of Kravchuk polynomials. These are a family of discrete orthogonal polynomials (see \cite{orthopoly}) defined as follows.

\begin{defin}[Kravchuk polynomial]
The Kravchuk polynomial $\Krav(a,b;n)$ is given as the coefficient of $z^a$ in $(1+z)^{n-b}(1-z)^b.$
Thus we have explicitly,
\[\Krav(a,b;n)=\sum_{j=0}^a(-1)^j\binom{b}{j}\binom{n-b}{a-j}.\]
\end{defin}

The following formula is stated in \cite[Prop. 2]{localstats} but also follows from the main results in the thesis of Harald Helfgott~\cite{helfgott}.

\begin{lem} \label{Creaform} 
    Let $n>0$. Suppose $l$ and $m$ are integers with $l+m\equiv n$~\em{mod}~$2$ \em and\/ $|l|+|m|\leq n$. Set $a:= (l+m+n)/2$ and\/ $b:= (l-m+n)/2$, then
    \[\Cr(l,m;n+1)=2^{-n}\Krav(a,b;n) \Krav(b,a;n).\]
    For all other integers $l$ and $m$ we have $\Cr(l,m;n+1)=0$.
\end{lem}
This already yields the hint, that we actually need to prove a statement about the growth of Kravchuk polynomials. Hence, the idea is to prove that Kravchuk polynomials grow in a way connected to Theorem \ref{1/4} and deduce from that the correct growth rate for the creation rates $\Cr(l,m;n)$. To do so, we will need the following two relations of Kravchuk polynomials. They are taken from \cite{ErrorCorrect}, however for a general introduction to the theory of hypergeometric orthogonal polynomials see \cite{orthopoly}. First we have a symmetry relation.

\begin{lem}[{\cite[Thm. 17]{ErrorCorrect}}]\label{kravsum}
    \[\Krav(b,a;n)=\frac{a!(n-a)!}{b!(n-b)!}\Krav(a,b;n).\]
\end{lem}
And secondly, we have the typical three term recurrence.
\begin{lem}[{\cite[Thm. 19]{ErrorCorrect}}]\label{Kravrec}
   The Kravchuk polynomials fulfil the following recurrence,
    \[(a+1)\Krav(a+1,b,n)=(n-2b)\Krav(a,b,n)-(n-a+1)\Krav(a-1,b,n)\]
    and thus
    \[\Krav(a+1,b;n)=\frac{n-2b}{a+1}\Krav(a,b;n)-\frac{n-a+1}{a+1}\Krav(a-1,b;n))\]
    or the other way around
    \[\Krav(a-1,b;n)=\frac{-a-1}{n-a+1}\Krav(a+1,b;n)+\frac{n-2b}{n-a+1}\Krav(a,b;n).\]
    
\end{lem}



\section{The Growth Theorem for Kravchuk polynomials} \label{Kravchukgrogo}
In this section we show that the Kravchuk polynomial expressions grow in a suitable way with respect to the formula stated in Theorem \ref{1/4}. More precisely, we prove the following result.
\begin{thm}[\sc{The Growth Theorem for Kravchuk polynomials}]\label{growing1}
    Let $a,b\in \N$ and $\alpha\in\{0,1,2,3\}$. Then we have
    \[\Krav(a+2p,b+2p; 4p+\alpha-1)=(-1)^{p}\binom{2p-1}{p}g_{a,b,\alpha}(p),\]
    where $g_{a,b,\alpha}(p)$ is a rational function in $p$.
\end{thm}

We show the statement above by reducing it to a finite amount of basic cases. To do so, we introduce the following shorthand-terminology: we say ``\textit{Theorem \ref{growing1} holds for $(a,b)\in \N^2$}\,'' if 
\[\Krav(a+2p,b+2p; 4p+\alpha-1)=(-1)^{p}\binom{2p-1}{p}g_{a,b,\alpha}(p)\]
for all $\alpha\in\{0,1,2,3\}.$ Then the following proposition tells us how to reduce the problem.

\begin{prop}\label{Kravreduce}  
The following three statements are true.

\begin{itemize}
    \item[(i)] If the assertion in Theorem \ref{growing1} is true for $(a,b)\in \N^2$, then it is also true for $(b,a)$.
    \item[(ii)] If the assertion in Theorem \ref{growing1} is true for $(0,b)\in \N^2$ and $(1,b)\in \N^2$, then it is also true for all $(n,b)$ with $n\in \N$ arbitrary.
    \item[(iii)] If the assertion in Theorem \ref{growing1} is true for $(a,b)=(0,0),(1,0)$ and $(1,1)$, then Theorem \ref{growing1} holds for all $a,b\in \N$ and all values for $\alpha$.
\end{itemize}
\end{prop}

\begin{proof}
        (i): Assume $\Krav(a+2p,b+2p; 4p+\alpha-1)=(-1)^{p}\binom{2p-1}{p}g(p)$. Then by using Lemma \ref{kravsum} we obtain for $\Krav(b+2p,a+2p,4p+\alpha -1)$ that
        \begin{align}
            \Krav(b+2p,a+2p,4p+\alpha -1)&= \notag\\ &\dfrac{(a+2p)!(4p+\alpha-a-2p)!}{(b+2p)!(4p+\alpha-b-2p)!}\Krav(a+2p,b+2p; 4p+\alpha-1) \notag \\
            &=\frac{(a+2p)!(2p+\alpha-a)!}{(b+2p)!(2p+\alpha-b)!}(-1)^{p}\binom{2p-1}{p}g(p) \notag\\
            &=(-1)^p \binom{2p-1}{p}\Bigg(\frac{(a+2p)!(2p+\alpha-a)!}{(b+2p)!(2p+\alpha-b)!}g(p)\Bigg), \notag 
        \end{align}
        where we notice that $\hat g(p):=\Big(\frac{(a+2p)!(2p+\alpha-a)!}{(b+2p)!(2p+\alpha-b)!}g(p)\Big)$ is again a rational function in $p$. \\
        (ii): Assume we have $\Krav(2p,b+2p; 4p+\alpha-1)=(-1)^{p}\binom{2p-1}{p}g_{0,b}(p)$ and \\ $\Krav(1+2p,b+2p; 4p+\alpha-1)=(-1)^{p}\binom{2p-1}{p}g_{1,b}(p)$. We use this as the start of an induction. Thus it remains to show, that if Theorem \ref{growing1} is true for $(n-2,b)$ and $(n-1,b)$ then it is also true for $(n,b)$. To see this we use Lemma \ref{Kravrec} which yields
        \begin{align}
           \Krav&(n+2p,b+2p,4p+\alpha-1)=\notag \\
           &\frac{4p+\alpha-1-2(b+2p)}{n+2p}\Krav(n-1+2p,b+2p,4p+\alpha-1)\notag \\
           &-\frac{4p+\alpha-1-(n-1+2p)+1}{n+2p}\Krav(n-2+2p,b+2p,4p+\alpha,-1)\notag \\
           &=\frac{\alpha-2b-1}{n+2p}\Big((-1)^p\binom{2p-1}{p}g_{n-1,b}(p)\Big)-\frac{2p+\alpha-n+1}{n+2p}\Big((-1)^p\binom{2p-1}{p}g_{n-2,b}(p) \notag \\
           &=(-1)^p\binom{2p-1}{p}\underbrace{\Bigg(\frac{\alpha-2b-1}{n+2p}g_{n-1,b}(p)-\frac{2p+\alpha-n+1}{n+2p}g_{n-2,b}(b)\Bigg)}_{g_{n,b}(p)}. \notag
        \end{align}
        Here we see again, that $g_{n,b}(p)$ is a rational function in $p$ provided that $g_{n-1,b}(p)$ and $g_{n-2,b}(p)$ are rational functions.
        
        (iii): Finally, assume that Theorem \ref{growing1} holds for $(a,b)=(0,0),(1,0)$ and $(1,1)$. If the theorem therefore holds for $(0,0)$ and $(1,0)$, (ii) implies that the theorem also holds for $(n,0)$ for all $n\in \N$. Furthermore, by (i), if the assertion is true for $(1,0)$ it is also true for $(0,1)$. Therefore, since we assume truth for the cases $(0,1)$ and $(1,1)$ we also have the theorem for the case $(n,1)$ with $n\in \N$ again arbitrary. Now, choose arbitrary $k,l\in \N$. By the arguments before, we have that the theorem holds for $(l,0)$ and $(l,1)$. Via the symmetry assertion in (i) we therefore also have the cases $(0,l)$ and $(1,l)$. However, again by (ii) this implies the assertion also for $(k,l)$. Since $k$ and $l$ where arbitrary, this proves the proposition.
\end{proof}

Proposition~\ref{Kravreduce} tells us that Theorem~\ref{growing1} is true, once we show it for the twelve cases of $(a,b)$ being $(0,0),(1,0),(1,1)$ and $\alpha$ varying over the values $0,1,2,3$. We summarize this last part of the proof in the following proposition.

\begin{prop}
    Let $a,b\in\{0,1\}$ but $(a,b)\neq (0,1)$ and $\alpha\in\{0,1,2,3\}$. Then we have
    \[\Krav(a+2p,b+2p,4p+\alpha-1)=(-1)^p\binom{2p-1}{p}g_{a,b,\alpha}(p),\]
    where the rational function $g_{a,b,\alpha}(p)$ can be read off Table \ref{tab:1}.
    \begin{table}[]
        \centering
        \begin{tabular}{|m{1.5cm}|m{1.3cm}|m{1cm}|m{1cm}|m{1cm}|}
        \hline
            $\alpha=$      &$0$\vphantom{$\displaystyle{\sum}$}  &$1$  &$2$  &$3$  \\
            \hline
            $g_{0,0,\alpha}(p)$ &$1$ &$2$&$2$&$\dfrac{2}{p+1}$\vphantom{$\displaystyle{\sum_n^n}$} \\
            \hline
            $g_{1,0,\alpha}(p)$ &$-1$\vphantom{$\displaystyle{\sum}$} &$0$  &$2$  &$4$ \\
            \hline
            $g_{1,1,\alpha}(p)$ &$\dfrac{3-2p}{-1+2p}$\vphantom{$\displaystyle{\sum_n^n}$}   &$-2$  &$-2$    &0\\ 
            \hline
        \end{tabular}
        \vskip10pt
        \caption{The rational function $g_{a,b,\alpha}(p).$}
        \label{tab:1}
    \end{table}
\end{prop}

\begin{proof}
    To prove the proposition we would need to verify the statement with the twelve rational functions provided in Table \ref{tab:1}, i.e., we would need to prove twelve different identities for alternating sums of products of binomial coefficients. Since the proof of all of these identities uses the same trick, we will only present it for the cases $(a,b,\alpha)=(0,0,0)$ and $(a,b,\alpha)=(1,1,0)$. The other cases work analogously.\\
    \textbf{Case $(a,b,\alpha)=(0,0,0)$:} We need to show
    \[\Krav(2p,2p,4p-1):=\sum_{j=0}^{2p}(-1)^j\binom{2p}{j}\binom{2p-1}{2p-j}=(-1)^j\binom{2p-1}{p}.\]
    The first equality is just the formula for Kravchuk polynomials. To see the second equality, look at the generating function $(1-z)^{2p}(1+z)^{2p-1}$. We have
    \begin{align}
        (1-z)^{2p}(1+z)^{2p-1}&=(1-z)(1-z^2)^{2p-1}=(1+z)\sum_{j=0}^{2p-1}(-1)^j\binom{2p-1}{j}z^{2j} \notag \\
        &=\sum_{j=0}^{2p-1}(-1)^j\binom{2p-1}{j}z^{2j}-\sum_{j=0}^{2p-1}(-1)^j\binom{2p-1}{j}z^{2j+1}. \label{Eq1}
    \end{align}
    However, on the other hand we also have
    \begin{align}
        (1-z)^{2p}(1+z)^{2p-1}&=\Bigg(\sum_{j=0}^{2p}(-1)^j\binom{2p}{j}z^j\Bigg)\Bigg(\sum_{k=0}^{2p-1}\binom{2p-1}{k}z^k\Bigg) \notag \\
        &=\sum_{k=0}^{2p-1}\sum_{j=0}^{2p}(-1)^j\binom{2p}{j}\binom{2p-1}{k}z^{j+k}. \label{Eq2}
    \end{align}
    Both expressions, (\ref{Eq1}) and (\ref{Eq2}), are equal. Now we compare the coefficient of $z^{2p}$ in both expressions.
    \begin{align}
        \langle z^{2p}\rangle \sum_{j=0}^{2p-1}(-1)^j\binom{2p-1}{j}z^{2j}-\sum_{j=0}^{2p-1}(-1)^j\binom{2p-1}{j}z^{2j+1}=(-1)^p\binom{2p-1}{p}-0 \notag
    \end{align}
    while
    \begin{align}
         \langle z^{2p}\rangle& \sum_{k=0}^{2p-1}\sum_{j=0}^{2p}(-1)^j\binom{2p}{j}\binom{2p-1}{k}z^{j+k}
         \underset{k+j=2p}{=} \sum_{j=0}^{2p}(-1)^j\binom{2p}{j}\binom{2p-1}{2p-j}.\notag
    \end{align}
    Hence the two rightmost expressions in the computations above must be equal.\\
    \textbf{Case $(a,b,\alpha)=(1,1,0)$:} We are playing the same game again. We need to show
    \[\Krav(1+2p,1+2p,4p-1):=\sum_{j=0}^{2p+1}(-1)^j\binom{2p+1}{j} \binom{2p-2}{2p+1-j}=(-1)^p\binom{2p-1}{p}\frac{3-2p}{-1+2p}.\]
    We compare the coefficient of $z^{2p+1}$ of the generating function $(1-z)^{2p+1}(1+z)^{2p-2}$:
    \begin{align}
        &\langle z^{2p+1}\rangle (1-z)^{2p+1}(1+z)^{2p-2}\notag \\
        =&\langle z^{2p+1}\rangle\Bigg(\sum_{j=0}^{2p+1}(-1)^j\binom{2p+1}{j}z^j\Bigg)\Bigg(\sum_{k=0}^{2p-2}\binom{2p-2}{k}z^k\Bigg) \notag \\
        =&\langle z^{2p+1}\rangle\sum_{k=0}^{2p-2}\sum_{j=0}^{2p+1}(-1)^j\binom{2p+1}{j}\binom{2p-2}{k}z^{j+k} \notag\\
        =&\sum_{j=0}^{2p+1}(-1)^j\binom{2p+1}{j}\binom{2p-2}{2p+1-j}. \notag
    \end{align}
    On the other hand, we have
    \begin{align}
        &\langle z^{2p+1}\rangle(1-z)^{2p+1}(1+z)^{2p-2}\notag \\
        &=\langle z^{2p+1}\rangle (1-z)^{3}(1-z^2)^{2p-2} \notag \\
        &=\langle z^{2p+1}\rangle(1-3z+3z^2-z^3)\sum_{j=0}^{2p-2}(-1)^j\binom{2p-2}{j}z^{2j} \notag\\
        &=\langle z^{2p+1}\rangle\Bigg(\sum_{j=0}^{2p-2}(-1)^j\binom{2p-2}{j}z^{2j}-3\sum_{j=0}^{2p-2}(-1)^j\binom{2p-2}{j}z^{2j+1}\notag\\&+3\sum_{j=0}^{2p-2}(-1)^j\binom{2p-2}{j}z^{2j+2}-\sum_{j=0}^{2p-2}(-1)^j\binom{2p-2}{j}z^{2j+3}\Bigg) \notag \\
        &= 0-3(-1)^p\binom{2p-2}{p}+0-(-1)^{p-1}\binom{2p-2}{p-1} \notag \\
        &= (-1)^p\Big(-3\frac{(2p-2)!}{p!(p-2)!}+\frac{(2p-2)!}{(p-1)!(p-1)!} \Big)\notag \\
        &= (-1)^p\Big(-3\frac{(2p-1)!}{p!(p-1)!}\frac{p-1}{2p-1}+\frac{(2p-1)!}{p!(p-1)!}\frac{p}{2p-1}\Big) \notag \\
        &=(-1)^p\binom{2p-1}{p}\Big(\frac{3-3p+p}{2p-1}\Big)=(-1)^p\binom{2p-1}{p}\Big(\frac{3-2p}{-1+2p}\Big)\notag
    \end{align}
    just as desired. The other ten cases follow the same path of computations. This proves the proposition and therefore also Theorem \ref{growing1}.
\end{proof}

\begin{rem}
    By the recursion for Kravchuk polynomials and the arguments of Proposition \ref{Kravreduce} (i) and (ii), the statement of Theorem \ref{growing1} also extends to negative values for $a$ or $b$ as long as we start with a $p\in\N$ such that $a+2p\geq 0$ and $b+2p\geq 0$.
\end{rem}

In particular, we need the following identity for the case $a=0$ and $b=-1$ later on.
\begin{lem}\label{b-1}
    We have that 
    \[\Krav(2p,2p-1,4p-1)=(-1)^p\binom{2p-1}{p}.\]
\end{lem}

\begin{proof}
    Just as before we have
    \begin{align}
        \Krav(2p,2p-1,4p-1)&= \langle z^{2p}\rangle (1-z)^{2p-1}(1+z)^{2p} \notag\\
        &=\langle z^{2p}\rangle (1+z) (1-z^2)^{2p-1}\notag\\
        &=\langle z^{2p}\rangle (1+z) \sum_{j=0}^{2p-1}(-1)^j\binom{2p-1}{j}z^{2j} \notag\\
        &=\langle z^{2p}\rangle \sum_{j=0}^{2p-1}(-1)^j\binom{2p-1}{j}z^{2j}+\sum_{j=0}^{2p-1}(-1)^j\binom{2p-1}{j}z^{2j+1} \notag\\
        &=(-1)^p\binom{2p-1}{p}+0. \notag
    \end{align}
\end{proof}

As a simple corollary of the Growth Theorem for Kravchuk polynomials we obtain the suitable growth for the creation rates.
\begin{cor} \label{Crgrowth}
    Let $n=4p+\alpha\geq 1$ and suppose $l$ and $m$ are integers with $l+m\equiv n-1$~\em{mod}~$2$ and $|l|+|m|\leq n-1$. Then
    \[\Cr(l,m;4p+\alpha)=2^{-4p-\alpha+1}\binom{2p-1}{p}^2f_{l,m,\alpha}(p)\]
    where $f_{l,m,\alpha}(p) $ is a rational function in $p$.
\end{cor}

\begin{proof}
    By Lemma \ref{Creaform} and Lemma \ref{kravsum} we have
    \begin{multline}
         \Cr(l,m;4p+\alpha)\\
         =2^{-4p-\alpha+1}\Krav\Big(\frac{l+m+4p+\alpha-1}{2},\frac{l-m+4p+\alpha-1}{2},4p+\alpha-1\Big)^2g(p) \notag
    \end{multline}
    with $g(p)$ the rational function obtained from the symmetry-factor in Lemma \ref{kravsum}. Now setting $a=\frac{l+m+\alpha-1}{2}$ and $b=\frac{l-m+\alpha-1}{2}$ we continue and see
    \[\Cr(l,m;4p+\alpha)=2^{-4p-\alpha+1}\Krav(a+2p,b+2p,4p+\alpha-1)^2g(p).\]
    If we now use the Growth Theorem for Kravchuk polynomials the assertion follows.
\end{proof}

\section{Proof of the main result} \label{Mainresult}
In this section we show Theorem~\ref{1/4}. We need to deduce that the statement is true for every position $(l,m)\in \Z^2$ for the horizontal domino tile whose left square is black. More precisely, for each position $(l,m)$ and each $\alpha\in\{0,1,2,3\}$ we want to show that the probability of observing such a tile at the position $(l,m)$ in the Aztec diamond of size $4p+\alpha$ is either zero or equal to
\[\mathbb{P}(l,m;4p+\alpha)=\frac{1}{4}+2^{-4p-\alpha}\binom{2p-1}{p}^2f_{l,m,\alpha}(p).\]
Our proof will proceed as follows: at first we show in Subsection~\ref{horax}, that Theorem~\ref{1/4} is true once it is shown for positions along the horizontal axis, i.e., we reduce the problem to identities for $\mathbb{P}(l,0;4p+\alpha)$ where $m=0$. In a second step, in Subsection~\ref{origin}, we use a recursion argument to see that its actually enough to show the certain shape of the formula for $\mathbb{P}(0,0;4p+1)$ and $\mathbb{P}(0,0;4p+3)$. Finally, we show in Subsection~\ref{final} the assertion of the theorem for these two cases. All other positions and cases modulo four can be obtained from them.

\subsection{Reduction to the horizontal axis}  \label{horax}
First of all, we use the symmetry of the Aztec diamond to observe that
\[\mathbb{P}(-l,m;4p+\alpha)=\mathbb{P}(l,m;4p+\alpha)\]
for all values $l,m\in\Z^2$ and $p\in\N,\alpha\in\{0,1,2,3\}$. Hence, we are allowed to assume $l\geq 0$. Moreover, the following statement holds.
\begin{prop} \label{horizon}
    If we have for all residues $\alpha\in\{0,1,2,3\}$ that either
    \[\mathbb{P}(l,0;4p+\alpha)=\frac{1}{4}+2^{-4p-\alpha}\binom{2p-1}{p}^2f(p)\]
    with $f(p)$ a rational function in $p$, whose shape depends only on $l\in\N$ and $\alpha$ \textbf{or} the probability equals $0$, then also for all $m\in\Z$ and $\hat \alpha\in\{0,1,2,3\}$ we find a rational function $\hat f(p)$ such that
    \[\mathbb{P}(l,m;4p+\hat\alpha)=\frac{1}{4}+2^{-4p-\alpha}\binom{2p-1}{p}^2\hat f(p)\]
    or $\mathbb{P}(l,m;4p+\hat\alpha)=0$.
\end{prop}

\begin{proof}
    Recalling the definition of the creation rate we have
    \[\Cr(l,m;n)= 2\big(\mathbb{P}(l,m;n)-\mathbb{P}(l,m-1;n-1)\big).\]
    In particular, for $m>0$ we have after substituting $n=4p+\alpha$:
    \[\mathbb{P}(l,m;4p+\alpha)=\mathbb{P}(l,m-1;4p+\alpha-1)+\frac{\Cr(l,m;4p+\alpha)}{2}.\]
    Now applying Corollary \ref{Crgrowth} we obtain
    \[\mathbb{P}(l,m;4p+\alpha)=\mathbb{P}(l,m-1;4p+\alpha-1)+2^{-4p-\alpha}\binom{2p-1}{p}^2f_{l,m,\alpha}(p),\]
    from which we derive that $\mathbb{P}(l,m;4p+\alpha)$ is of the form
    \[\frac{1}{4}+2^{-4p-\alpha}\binom{2p-1}{p}^2\hat{f}_{m,l,\alpha}(p)\]
    for some rational function $\hat{f}_{m,l,\alpha}(p)$ in $p$, if already $\mathbb{P}(l,m-1;4p+\alpha-1)$ fulfils this property. Note that we use here the fact that for any fixed $a,b,c\in \Z$ we have
    \[2^{-4p-\alpha +a}\binom{2p-1+b}{p+c}=2^{-4p-\alpha}\binom{2p-1}{p}g(p)\]
    for some rational function $g(p)$ in $p$. Moreover, if $h(p)$ and $g(p)$ are rational functions, then also $h(g(p))$ is rational in $p$. However, thus inductively the assertion is true for $\mathbb{P}(l,m;4p+\alpha)$ if it is true for $\mathbb{P}(l,0;4p+\alpha)$.

    For the case $m<0$ we see that the definition of the creation rate also yields
    \[\mathbb{P}(l,m;4p+\alpha)=\mathbb{P}(l,m+1;4p+\alpha+1)-\frac{\Cr(l,m+1;4p+\alpha+1)}{2}.\]
    The rest of the argument works analogously to the case before.
\end{proof}

The proof of Proposition \ref{horizon} tells us, that once Theorem~\ref{1/4} is proven for all residues $\alpha$ at position $(l,m)$, the statement also holds for all $\alpha\in\{0,1,2,3\}$ at position $(l,m+k)$ where $k\in \Z$. In particular it is enough to show Theorem \ref{1/4} for all positions $(l,0)$, $l\geq 0$ along the horizontal axis and all residues $\alpha\in\{0,1,2,3\}$. In a next step we show that it is actually enough to prove the theorem for $\mathbb{P}(0,0;4p+\alpha)$ with $\alpha\in\{1,3\}$.

\subsection{Reduction to the origin} \label{origin}

We want to show the following proposition.
\begin{prop}\label{horireduce}
    If we have  
    \[\mathbb{P}(0,0;4p+\alpha)=\frac{1}{4}+2^{-4p-\alpha}\binom{2p-1}{p}^2f_\alpha(p)\]
    for a rational function $f_\alpha(p)$ in $p$ if $\alpha\in\{1,3\}$ and $\mathbb{P}(0,0;4p+\alpha)=0$ if $\alpha\in\{0,2\}$ then also 
    \[\mathbb{P}(l,0;4p+\alpha)=\frac{1}{4}+2^{-4p-\alpha}\binom{2p-1}{p}^2f_{l,\alpha}(p)\]
    with $f_{l,\alpha}(p)$ a rational function in $p$ or $\mathbb{P}(l,0;4p+\alpha)=0$.
\end{prop}

\begin{proof}
    The decision whether $\mathbb{P}(l,0;4p+\alpha)=0$ depends only on the parity of $l$ and $\alpha$. Thus we only need to check whether such a formula in the presented shape holds when $\mathbb{P}(l,0;4p+\alpha)\neq 0$.
    We prove this via induction on $l\geq 0$. Hence, let us assume the assertion is true for all $\mathbb{P}(k,0;4p+\beta)$ with $k<l$ and $\beta\in\{0,1,2,3\}$. By Proposition \ref{horizon} we can therefore assume that the shape of the formula holds for all $\mathbb{P}(k,m;4p+\beta)$ with $m\in \Z$. 

    Now we use for the first time that we are actually working with probabilities. We are sampling domino tilings on the Aztec diamond uniformly at random. Every square in the Aztec diamond will be covered by a single domino, pairing said square with a neighbouring square of the other colour. This neighbour is located either to the \textit{west, east, south} or \textit{north} of the original square. Compare this with Figure~\ref{orient}. We denote by
    \[\mathbb{P}(x,y;\textit{ north})\]
    the probability that in the Aztec diamond of size $4p+\alpha$ the square with its bottom-right corner located at $(x,y)\in\Z^2$ is covered by an north-looking domino in our random tiling. Similarly, we denote by $\mathbb{P}(x,y; \textit{ west})$, $\mathbb{P}(x,y; \textit{ east})$ and $\mathbb{P}(x,y; \textit{ south})$ the probabilities for the other three cases.
    \begin{figure}
    \centering
    \includegraphics[width=0.3\textwidth]{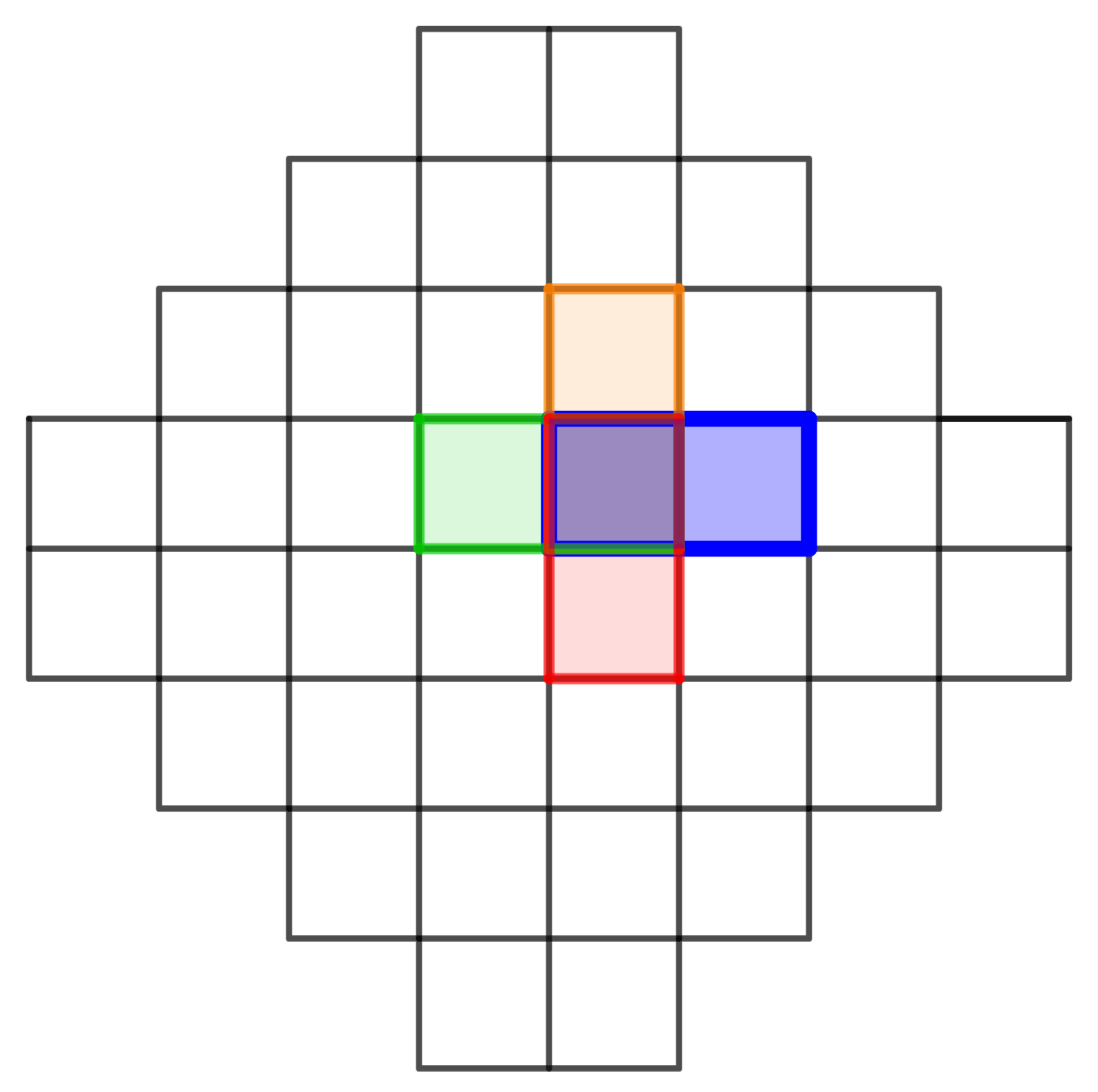}
    \caption{The four possibilities to place a domino tile on the upper-right central unit square. From the point of view of the covered unit square, the covering tile has to be oriented either \textit{east}, \textit{north}, \textit{west} or \textit{south}. By symmetries and our conventions of colouring we have $\mathbb{P}(1,0;4p+\alpha)=\mathbb{P}(\textit{east})=\mathbb{P}(\textit{north})$ and $\mathbb{P}(0,-1;4p+\alpha)=\mathbb{P}(\textit{west})=\mathbb{P}(\textit{south})$.}
    \label{orient}
    \end{figure}
    
    Thus we have
    \[1=\mathbb{P}(x,y; \textit{ west})+\mathbb{P}(x,y; \textit{ north})+\mathbb{P}(x,y; \textit{ east})+\mathbb{P}(x,y; \textit{ south})\]
    for every position $(x,y)\in\Z^2$. If $(l,m)$ indicates a black square in the Aztec diamond of size $4p+\alpha$, then we obtain, using our notation,
    \[\mathbb{P}(l,m;4p+\alpha)=\mathbb{P}(l,m; \textit{ east}).\]
    Therefore, we have 
    \[\mathbb{P}(l,0;4p+\alpha)= 1-\mathbb{P}(l,0; \textit{ west})-\mathbb{P}(l,0; \textit{ north})-\mathbb{P}(l,0; \textit{ south}).\]
    If we rewrite the probabilities on the right hand side into our original parametrisation by rotating the Aztec diamond in the right way and using its symmetry with respect to the vertical axis, we see that
    \begin{align}
        \mathbb{P}(l,0; \textit{ west})&=\mathbb{P}(l-1,-1;4p+\alpha) \notag \\
        \mathbb{P}(l,0; \textit{ north})&=\mathbb{P}(1,l-1;4p+\alpha) \notag \\
        \mathbb{P}(l,0; \textit{ south})&=\mathbb{P}(0,-l; 4p+\alpha). \notag
    \end{align}
    First of all, notice that none of these probabilities equals $0$. Moreover we have 
    \[\mathbb{P}(0,-l; 4p+\alpha)=\frac{1}{4}+2^{-4p-\alpha}\binom{2p-1}{p}^2f_{\textit{south},\alpha}(p)\]
    since such a formula is assumed for $\mathbb{P}(0,0;4p+\alpha)$ and we have Proposition \ref{horizon}. Furthermore, for the same reason, $\mathbb{P}(1,l-1;4p+\alpha)$ is of such a shape if $\mathbb{P}(1,0;4p+\alpha)$ is. However, using the same trick as before and the symmetries of the Aztec diamond we have
    \begin{align}
        1&= 2 \mathbb{P}(1,0;4p+\alpha)+2 \mathbb{P}(0,-1;4p+\alpha) \label{1reflex} \\
        \Leftrightarrow \mathbb{P}(1,0;4p+\alpha)&=\frac{1}{2}\Big(1-2\mathbb{P}(0,-1;4p+\alpha)\Big) \notag \\
        &= \frac{1}{2}\Big(1-2\Big(\frac{1}{4}+2^{-4p-\alpha}\binom{2p-1}{p}^2f_{0,-1,\alpha}(p)\Big)\Big) \notag \\
        &=\frac{1}{4}+2^{-4p-\alpha}\binom{2p-1}{p}^2f_{0,-1,\alpha}(p). \notag
    \end{align}
    Therefore we conclude
    \[\mathbb{P}(1,l-1; 4p+\alpha)=\frac{1}{4}+2^{-4p-\alpha}\binom{2p-1}{p}^2f_{\textit{north},\alpha}(p).\]
    Finally, $\mathbb{P}(l-1,-1;4p+\alpha)$ is of the desired form if $\mathbb{P}(l-1,0;4p+\alpha)$ is. However, at this point we might use induction to obtain
    \[\mathbb{P}(l-1,-1;4p+\alpha)=\frac{1}{4}+2^{-4p-\alpha}\binom{2p-1}{p}^2f_{\textit{west},\alpha}(p).\]
    Finally, we insert all of this in our expression for $\mathbb{P}(l,0;4p+\alpha)$ to obtain
    \begin{align}
        \mathbb{P}(l,0;4p+\alpha)&=1-\Big(\frac{3}{4}+2^{-4p-\alpha}\binom{2p-1}{p}^2\Big(f_{\textit{west},\alpha}(p)+f_{\textit{north},\alpha}(p)+f_{\textit{south},\alpha}(p) \Big)\Big) \notag\\
        &=\frac{1}{4}+2^{-4p-\alpha}\binom{2p-1}{p}^2f_{l,\alpha}(p), \notag
    \end{align}
    with $f_{l,\alpha}(p)=-f_{\textit{west},\alpha}(p)-f_{\textit{north},\alpha}(p)-f_{\textit{south},\alpha}(p)$ still being a rational function in $p$.
\end{proof}

\subsection{The case of the origin} \label{final}
During the last two subsections we were able to reduce the proof of Theorem \ref{1/4} to checking whether it holds at the origin. First of all, due to parity issues, we have
\[\mathbb{P}(0,0;4p+\alpha)=0\]
if $\alpha=0$ or $2$. This is clear, since in these cases the left square of a horizontal domino tile located at $(0,0)$ would be white. Hence, we only need to find rational functions $f_1(p)$ and $f_3(p)$ such that 
\[\mathbb{P}(0,0;4p+\alpha)=\frac{1}{4}+2^{-4p-\alpha}\binom{2p-1}{p}^2f_\alpha(p)\]
with $\alpha\in\{1,3\}$. Now, simulations have suggested the following.
\begin{align}
    \mathbb{P}(0,0;4p+1)&=\frac{1}{4}+2^{-4p}\binom{2p-1}{p}^2, \notag\\
    \mathbb{P}(0,0;4p+3)&=\frac{1}{4}. \notag
\end{align}
I.e., we have $f_1(p)=2$ and $f_3(p)=0$.

We start with the $n=4p+3$ case, since this seems to be the simpler one. To show that $\mathbb{P}(0,0;4p+3)$ is constantly $1/4$ we will need a result from general perfect matching enumeration namely the so-called \textit{Kuo condensation}.
\begin{defin}
    Let $G=(V,E)$ be a bipartite planar graph with set of vertices $V$ and set of edges $E$. Assume $G$ has as many black as white vertices. A \textit{perfect matching} of $G$ is a subset of edges $\mu\subseteq E$ such that every vertex is contained in exactly one edge in $\mu$.
\end{defin}

\begin{rem}
    It is well known that the number of domino tilings of a domino region is in 1-to-1 correspondence with the perfect matchings of the inner dual graph (where we remove the vertex associated to the unbound face). 
\end{rem}

Every domino tile in a random tiling corresponds to an edge in the perfect matching of the dual graph. Therefore, the probability of observing a horizontal domino equals the probability to find the according edge in a random perfect matching. Let $e=\{a,b\}$ be such an edge, where $a$ and $b$ are vertices. Then the number of perfect matchings, which contain the edge $e$ equals the number of perfect matchings of the subgraph $G\bs\{a,b\}$. Kuo condensation relates the perfect matching numbers of such subgraphs.

\begin{thm}[{\cite[Prop.~1.1]{Kuo}}] \label{Kuo}
    Let $G=(V,E)$ be a planar graph and let $a,b,c,d\in V$ be four vertices appearing in cyclic order in a face of $G$. Denote by $M(H)$ the number of perfect matchings of a graph $H$. Then we have
    \begin{align}
        M(G)M(G\bs\{a,b,c,d\})=&M(G\bs\{a,b\})M(G\bs\{c,d\})-M(G\bs\{a,c\})M(G\bs\{b,d\})\notag \\
        &+M(G\bs\{a,d\})M(G\bs\{b,c\}),\notag
    \end{align}
    where $G\bs\{v,v'\}$ means the graph obtained from $G$ by deleting the vertices $v,v'\in V$.
\end{thm}

In the context of a domino region like the Aztec diamond, the deletion of a vertex refers to the removal of a square. Theorem \ref{Kuo} will be be the essential ingredient to validate the conjectured formulas for $\mathbb{P}(0,0;4p+1)$ and $\mathbb{P}(0,0;4p+3)$.
 In particular, going back to the setting of Aztec diamonds, we are almost ready to prove the probability in the $4p+3$-case. All we need is a related result of Mihai Ciucu, counting domino tilings of the Aztec diamond with a central $2\times 2$-square hole. 

\begin{thm}[{\cite[Thm.~3.1]{Ciucu2}}] \label{reflect}
    The \textit{holey} Aztec diamonds of order $4p+2$ and $4p+3$ have $2^{8p^2+10p}$ and $2^{8p^2+14p+3}$ domino tilings, respectively. 
\end{thm}

Ciucu proves his result applying his \textit{Matching Factorization Theorem} using one of the several symmetries of the Aztec diamond. We will only need the second of his results to prove the following proposition.

\begin{prop}[The $4p+3$-case] \label{Case3}
    We have for all $p\in \N$:
    \[\mathbb{P}(0,0;4p+3)=\frac{1}{4}.\]
\end{prop}

\begin{proof}
    Set $X:=\mathbb{P}(0,0;4p+3)$. Now consider the number of domino tilings of the holey Aztec diamond, where the four central squares were removed. Label those central squares by $a,b,c,d$ in clockwise order. We apply Theorem~\ref{Kuo} to the Aztec diamond $\mathcal{A}$ of size $4p+3$ and the four central squares $a,b,c,d$. As the Aztec diamond corresponds to a bipartite graph, the Kuo condensation formula reduces to
    \begin{align}
        M(\mathcal{A})M(\mathcal{A}\bs\{a,b,c,d\})=M(\mathcal{A}\bs\{a,b\})M(\mathcal{A}\bs\{c,d\})+M(\mathcal{A}\bs\{a,d\})M(\mathcal{A}\bs\{b,c\}), \notag
    \end{align}
    since, according to our labelling both $a$ and $c$ are squares of the same colour and therefore $M(\mathcal{A}\bs\{a,c\})=0$.
    
    The (full) Aztec diamond of size $4p+3$ has $M(\mathcal{A})=2^{(4p+4)(4p+3)/2}$ tilings. By Theorem~\ref{reflect} the holey Aztec diamond of size $4p+3$ has $M(\mathcal{A}\bs\{a,b,c,d\})=2^{8p^2+14p+3}$ tilings. Now by rotational symmetry, the regions $\mathcal{A}\bs\{a,b\}, \mathcal{A}\bs\{c,d\}$ etc. all correspond to the Aztec diamond where a horizontal tile at position $(0,0)$ was fixed. Thus we have
    \[M(\mathcal{A}\bs\{a,b\})=M(\mathcal{A}\bs\{c,d\})=M(\mathcal{A}\bs\{a,d\})=M(\mathcal{A}\bs\{b,c\})=2^{(4p+4)(4p+3)/2}\cdot X.\]
    Inserting all these numbers into the Kuo condensation equation above, we compute
    \begin{align}
        2^{(4p+4)(4p+3)/2}\cdot2^{8p^2+14p+3}&= 2\cdot(2^{(4p+4)(4p+3)/2}\cdot X)^2 \notag\\
        2^{16p^2+28p+9}&=2^{16p^2+28p+13}\cdot X^2 \notag\\
        2^{-4}&=X^2, \notag
    \end{align}
    and thus $\mathbb{P}(0,0;4p+3)=X=1/4$.
\end{proof}

It remains to show the formula for $\mathbb{P}(0,0;4p+1)$. To do so, we will again apply Kuo condensation, Theorem~\ref{Kuo}, but now we \textit{condensate from the outside}. What exactly we mean by that will become clear in the proof of the following proposition, which is moreover the final result needed to prove Theorem \ref{1/4}.

\begin{prop}[The $4p+1$-case)] \label{Case1}
    We have for all $p\in \N$:
    \[\mathbb{P}(0,0;4p+1)=\frac{1}{4}+2^{-4p}\binom{2p-1}{p}^2.\]
\end{prop}

\begin{proof}
    Let $\mathcal{A}^+$ be the Aztec diamond of size $4p+1$ with $p\in \N$, with a horizontal domino shaped hole at position $(0,0)$ and with squares denoted by $a,b,c,d$ added in second position in clockwise order at the corners as its shown in Figure \ref{A+}.
    \begin{figure}[ht] 
		\centering
		\includegraphics[width=0.55\textwidth]{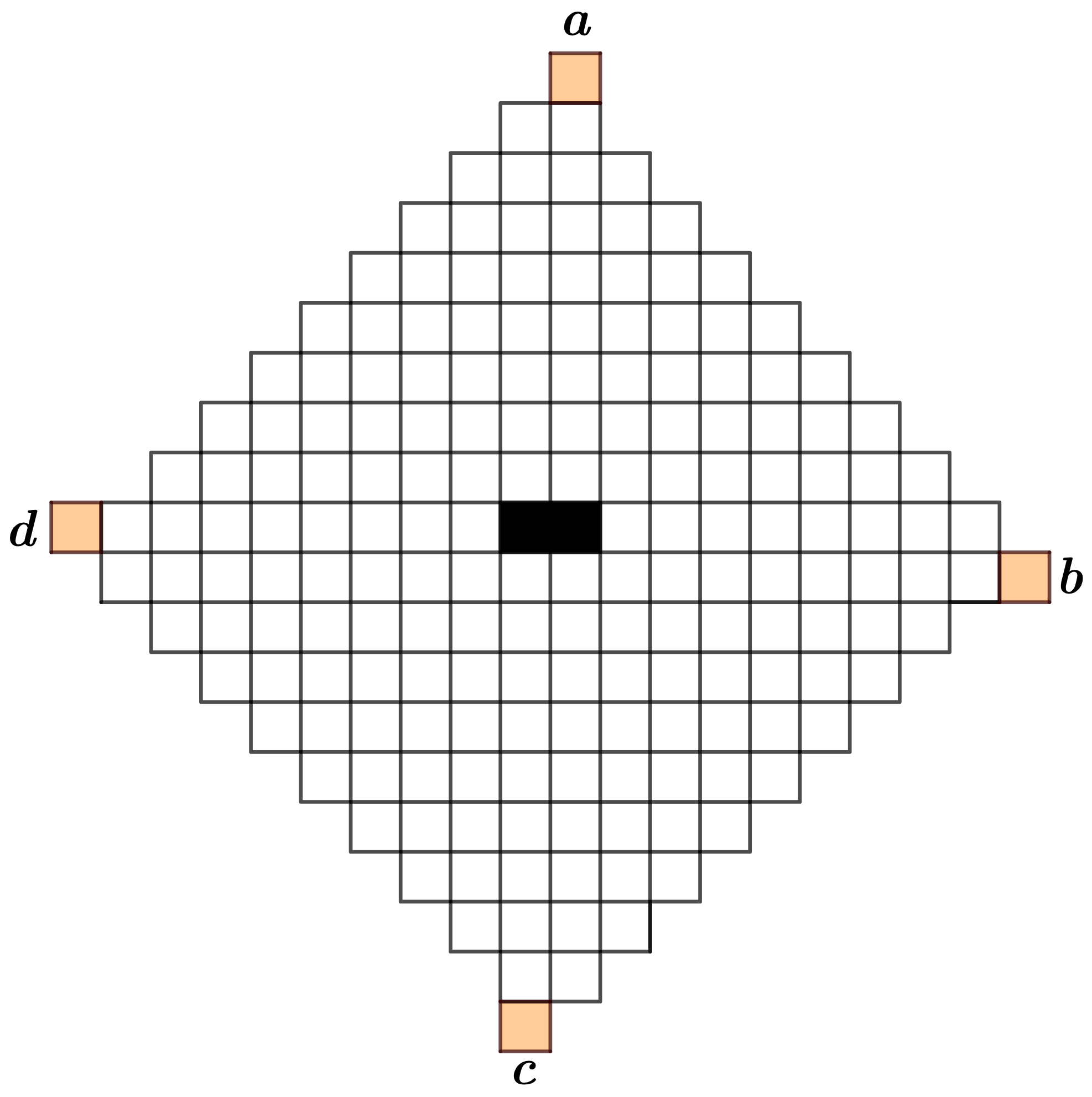}
		\caption{The domino domain $\mathcal{A}^+$: an Aztec diamond of size $4p+1$ (here $p=2$) with a horizontal domino shaped hole at position $(0,0)$ and additional squares $a,b,c,d$ in particular positions of the corners. (Always added at second position at each tip, when read clockwise.}
		\label{A+}
	\end{figure}
    
    Applying Kuo condensation with respect to the additional squares we get
     \begin{align}
        M(\mathcal{A}^+)M(\mathcal{A}^+\bs\{a,b,c,d\})=&M(\mathcal{A}^+\bs\{a,b\})M(\mathcal{A}^+\bs\{c,d\})\notag \\&-M(\mathcal{A}^+\bs\{a,c\})M(\mathcal{A}^+\bs\{b,d\})\notag \\
        &+M(\mathcal{A}^+\bs\{a,d\})M(\mathcal{A}^+\bs\{b,c\}).\notag
    \end{align}
    Since $\mathcal{A}^+$ is bipartite and the squares $a$ and $c$ are from the same colour, we have $M(\mathcal{A}^+\bs\{a,c\})=0$ and thus arrive at the following equation:
   \begin{align}
       0=&- M(\mathcal{A}^+)M(\mathcal{A}^+\bs\{a,b,c,d\}) &(I)\notag\\
       &+M(\mathcal{A}^+\bs\{a,b\})M(\mathcal{A}^+\bs\{c,d\}) &(II) \notag \\
       &+M(\mathcal{A}^+\bs\{a,d\})M(\mathcal{A}^+\bs\{b,c\}).\notag &(III) 
   \end{align}
   
   Our goal is now to replace all this perfect matching terms by explicit formulas in terms of placement probabilities. 

    \begin{itemize}
        \item  For $M(\mathcal{A}^+)$ we have by the Aztec Diamond Theorem and the observations described in Figure \ref{A+abcd}:
   \begin{align}
       M(\mathcal{A}^+)=2^{2p(4p-1)}\mathbb{P}(0,0;4(p-1)+3)=2^{8p^2-2p-2} \label{Erst}
   \end{align}
   since $\mathbb{P}(0,0;4(p-1)+3)=1/4$ by Proposition \ref{Case3}.
  
   \item The value of $M(\mathcal{A}^+\bs\{a,b,c,d\})$ is simply given by 
   \[2^{(2p+1)(4p+1)}\mathbb{P}(0,0;4p+1)\]
   since $\mathcal{A}^+\bs\{a,b,c,d\}$ just is an Aztec diamond of size $4p+1$ with a horizontal domino shaped hole at position $(0,0)$. (Compare with Figure \ref{A+}.) Since $\mathbb{P}(0,0;4p+1)$ is the entity we are looking for, we set $X:=\mathbb{P}(0,0;4p+1)$ to simplify the upcoming equations. Thus we have
   \begin{align}
       M(\mathcal{A}^+\bs\{a,b,c,d\})=2^{8p^2+6p+1}X.
   \end{align}

   \item For $M(\mathcal{A}^+\bs\{a,b\})$ we use the observations of Figure \ref{A+ab} to compute
   \begin{align}
       M(\mathcal{A}^+\bs\{a,b\})&= 2^{2p(4p+1)}\mathbb{P}(0,-1;4p) \notag \\
        &=2^{8p^2+2p}\Bigg(\mathbb{P}(0,0;4p+1)-\dfrac{\Cr(0,0;4p+1)}{2}\Bigg) \notag \\
        &=2^{8p^2+2p}\Big(X-2^{-4p-1}\Krav(2p,2p,4p)^2\Big) \notag
    \end{align}
    where we first used the relation given in the definition of the creation rates and afterwards applied Lemma~\ref{Creaform} to express this in terms of Kravchuk polynomials. However, by Theorem~\ref{growing1} the computation above continues to
    \begin{align}   
      M(\mathcal{A}^+\bs\{a,b\})&=2^{8p^2+2p}X-2^{8p^2-2p-1}\binom{2p-1}{p}^2 g_{0,0,1}(p) \notag \\
         &=2^{8p^2+2p}X-2^{8p^2-2p+1}\binom{2p-1}{p}^2,
    \end{align}
    where $g_{0,0,1}(p)$ was read off Table\ref{tab:1}.

    \item Similarly, for $M(\mathcal{A}^+\bs\{c,d\})$ we have by Figure \ref{A+cd}:
    \begin{align}
        M(\mathcal{A}^+\bs\{c,d\})&= 2^{2p(4p+1)}\mathbb{P}(0,1;4p) \notag \\
        &=2^{8p^2+2p}\Bigg(\mathbb{P}(0,0;4p-1)+\dfrac{\Cr(0,1;4p)}{2}\Bigg) \notag \\
        &=2^{8p^2+2p}\Bigg(\dfrac{1}{4}+\dfrac{\Cr(0,1;4p)}{2}\Bigg) \notag
    \end{align}
    where we used that $\mathbb{P}(0,0;4p-1)=1/4$ by Proposition~\ref{Case3}. Now, again by applying Lemma~\ref{Creaform} to replace the creation rates by Kravchuk polynomials and afterwards making use of the symmetry relation given in Lemma~\ref{kravsum} the computation continues with
    \begin{align}
        &=2^{8p^2+2p}\Bigg(\dfrac{1}{4}+2^{-4p}\Krav(2p,-1+2p,4p-1)\Krav(-1+2p,2p,4p-1)\Bigg) \notag \\
         &=2^{8p^2+2p}\Bigg(\dfrac{1}{4}+2^{-4p}\Krav(2p,2p-1,4p-1)^2\frac{(2p)!(4p-1-2p)!}{(2p-1)!(4p-1-2p+1)!}\Bigg) \notag\\
         &=2^{8p^2+2p}\Bigg(\dfrac{1}{4}+2^{-4p}\Krav(2p,2p-1,4p-1)^2\Bigg) \notag \\
        &=2^{8p^2+2p}\Bigg(\dfrac{1}{4}+2^{-4p}\binom{2p-1}{p}^2\Bigg) \notag \\
         &=2^{8p^2+2p-2}+2^{8p^2-2p}\binom{2p-1}{p}^2,
     \end{align}
     where in the second to last step we applied the identity of Lemma~\ref{b-1}.

     \item For $M(\mathcal{A}^+\bs\{a,d\})$ we have by Figure \ref{A+ad}:
     \[M(\mathcal{A}^+\bs\{a,d\})=2^{8p^2+2p}\mathbb{P}(1,0;4p).\]
        However, by Equation \eqref{1reflex} this becomes
        \begin{align}
            M(\mathcal{A}^+\bs\{a,d\})&=2^{8p^2+2p}\Bigg(\frac{1}{2}-\mathbb{P}(0,-1;4p)\Bigg)\notag \\
            &=2^{8p^2+2p}\Bigg(\frac{1}{2}-\Bigg(\mathbb{P}(0,0;4p+1)-\dfrac{\Cr(0,0;4p+1)}{2} \Bigg)\Bigg) \notag \\
            &=2^{8p^2+2p-1}-2^{8p^2+2p}X+2^{8p^2-2p+1}\binom{2p-1}{p}^2.
        \end{align}

    \item Finally, by symmetry of the Aztec diamond, we obtain from the observations of Figure \ref{A+bc} that
    \begin{align}
        M(\mathcal{A}^+\bs\{b,c\})&=M(\mathcal{A}^+\bs\{a,d\})\notag \\
        &=2^{8p^2+2p-1}-2^{8p^2+2p}X+2^{8p^2-2p+1}\binom{2p-1}{p}^2. \label{Letzt}
    \end{align}
   \end{itemize}

   \begin{figure}[ht] 
		\centering
		\includegraphics[width=0.4\textwidth]{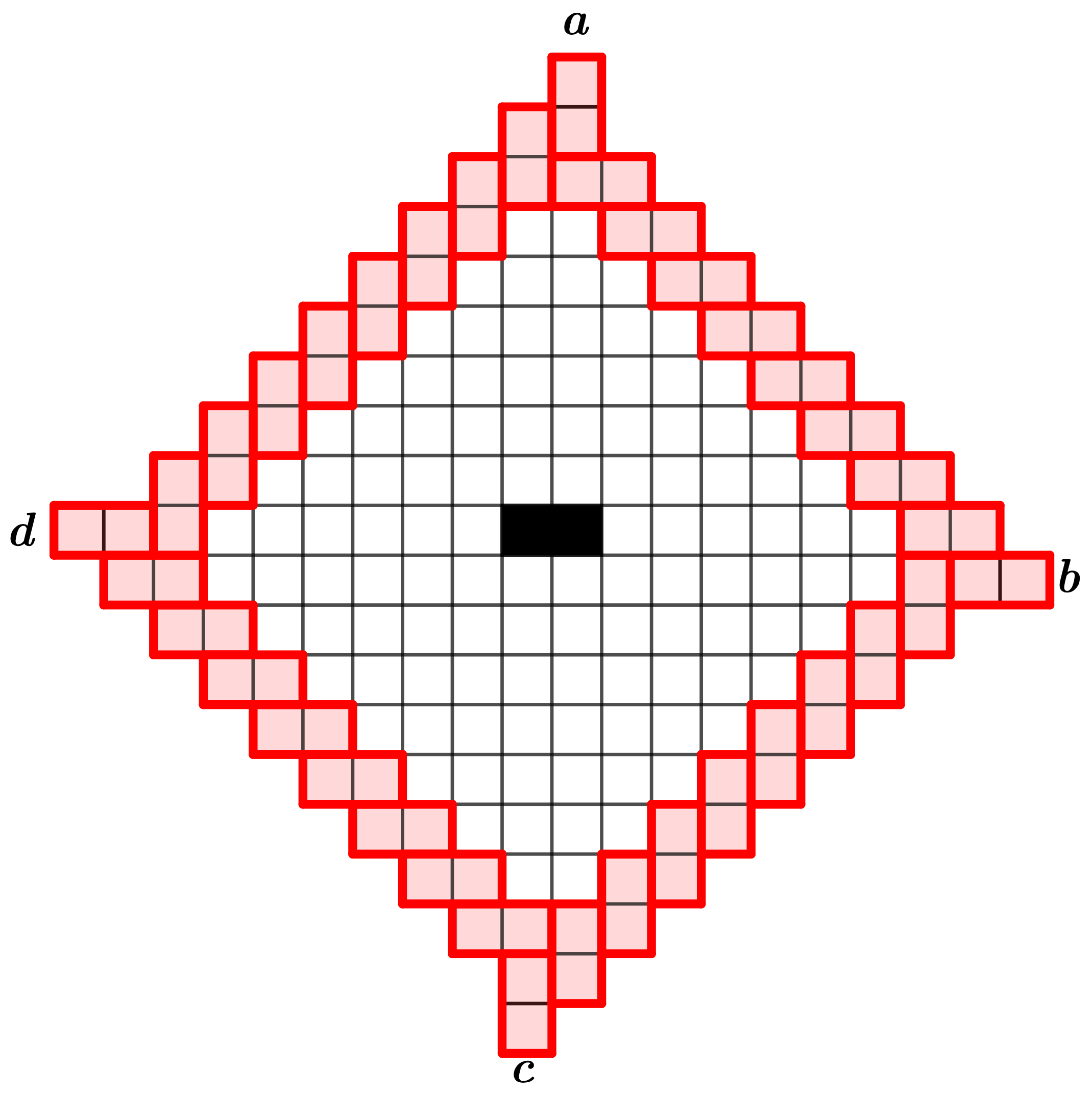}
		\caption{For the domain $\mathcal{A}^+$ we observe a lot of forced tiles. These can be removed and hence we derive that the number of domino tilings $M(\mathcal{A}^+)$ equals the number of tilings of the Aztec diamond of size $4p+1-2=4(p-1)+3$ with a horizontal domino shaped hole at position $(0,0)$. This number is given by $2^{4p\cdot(4(p-1)+3)/2}\mathbb{P}(0,0;4(p-1)+3)$. }
		\label{A+abcd}
	\end{figure}

    \begin{figure}[ht] 
		\centering
		\includegraphics[width=0.4\textwidth]{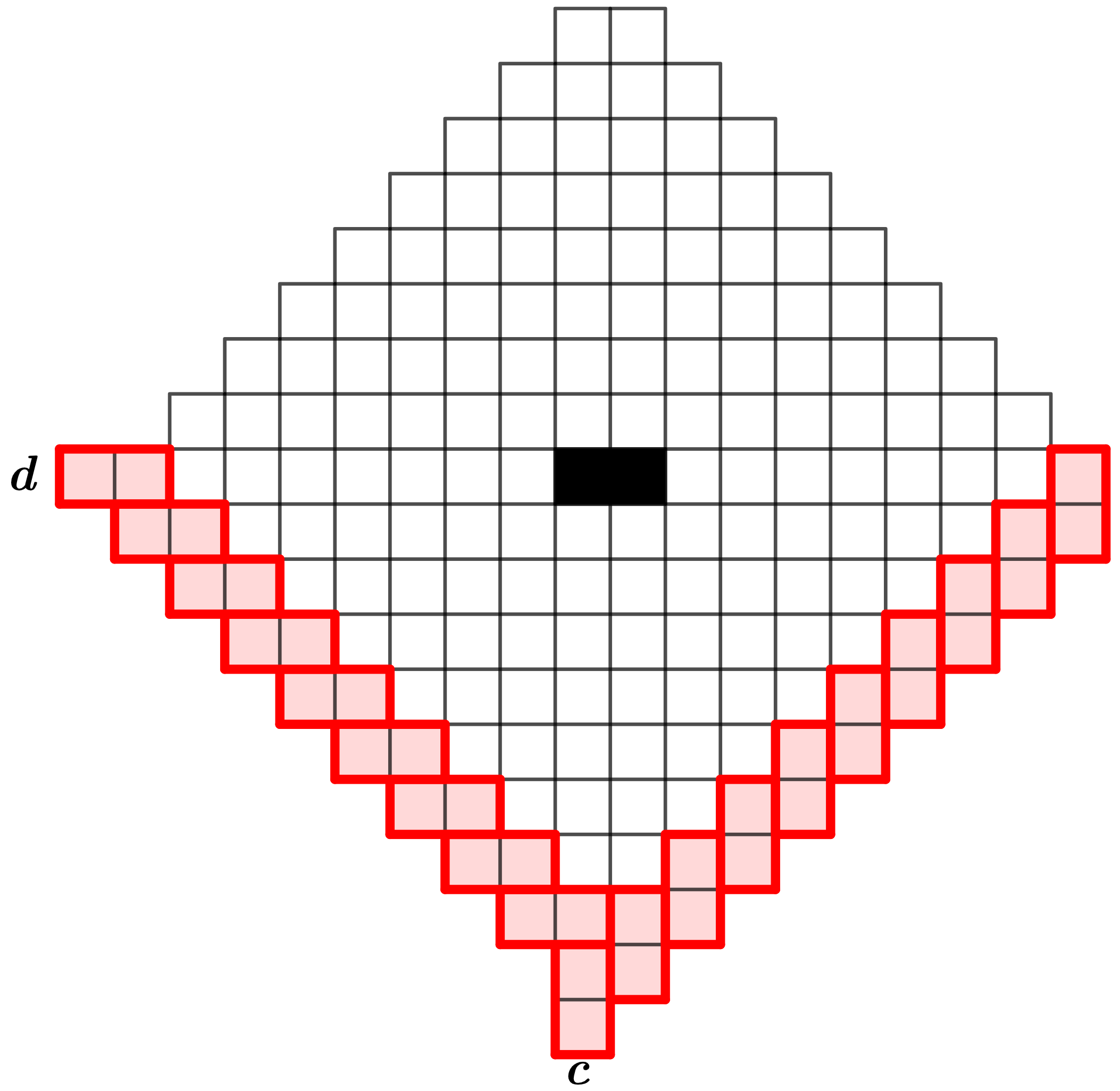}
		\caption{For the domain $\mathcal{A}^+\bs \{a,b\}$ we mark again the forced tiles. The remaining region corresponds to an Aztec diamond of size $4p$ with an horizontal domino shaped hole at position $(0,-1)$. Thus we have $M(\mathcal{A}^+\bs \{a,b\})=2^{2p(4p+1)}\mathbb{P}(0,-1;4p)$.}
		\label{A+ab}
	\end{figure}

     \begin{figure}[ht] 
		\centering
		\includegraphics[width=0.4\textwidth]{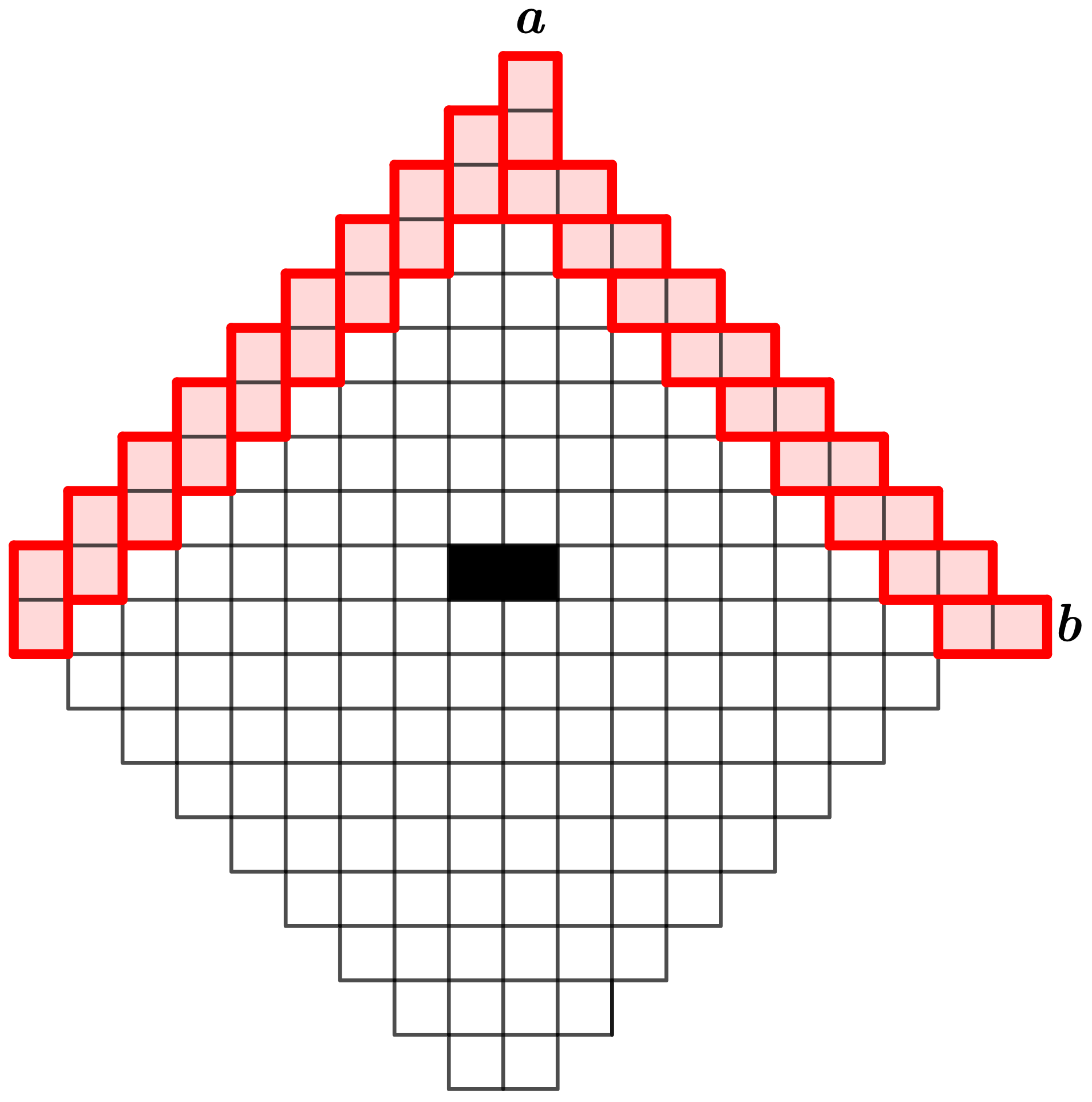}
		\caption{Also for the domain $\mathcal{A}^+\bs \{c,d\}$ we mark the forced tiles. The remaining region corresponds to an Aztec diamond of size $4p$ with an horizontal domino shaped hole at position $(0,1)$. Thus we have $M(\mathcal{A}^+\bs \{c,d\})=2^{2p(4p+1)}\mathbb{P}(0,1;4p)$.}
		\label{A+cd}
	\end{figure}

    \begin{figure}[ht] 
		\centering		\includegraphics[width=0.4\textwidth]{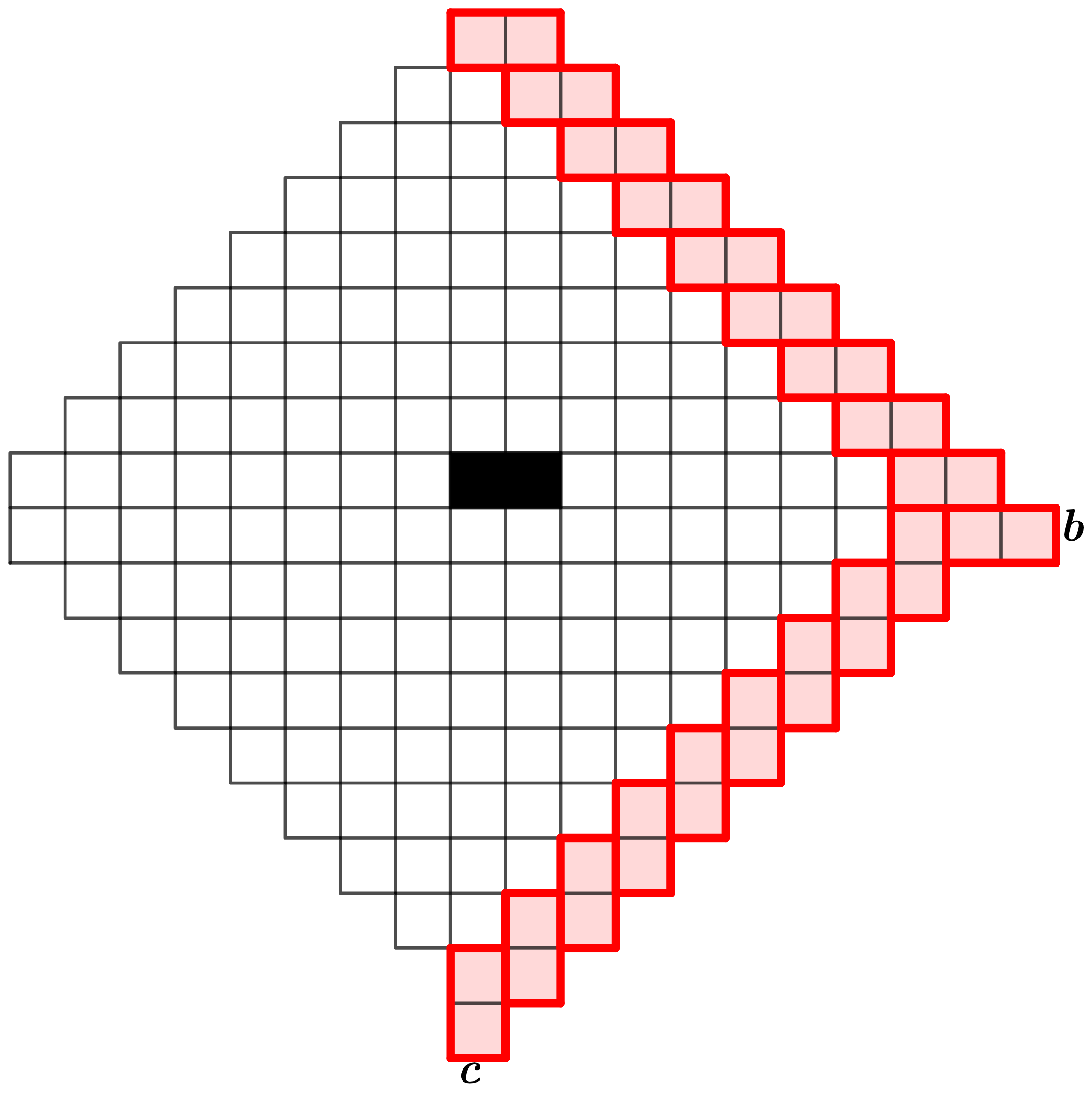}
		\caption{The domain $\mathcal{A}^+\bs \{a,d\}$ with marked forced tiles. The remaining region corresponds to an Aztec diamond of size $4p$ with an horizontal domino shaped hole at position $(1,0)$. Thus we have $M(\mathcal{A}^+\bs \{a,d\})=2^{2p(4p+1)}\mathbb{P}(1,0;4p)$.}
		\label{A+ad}
	\end{figure}

    \begin{figure}[ht] 
		\centering
		\includegraphics[width=0.5\textwidth]{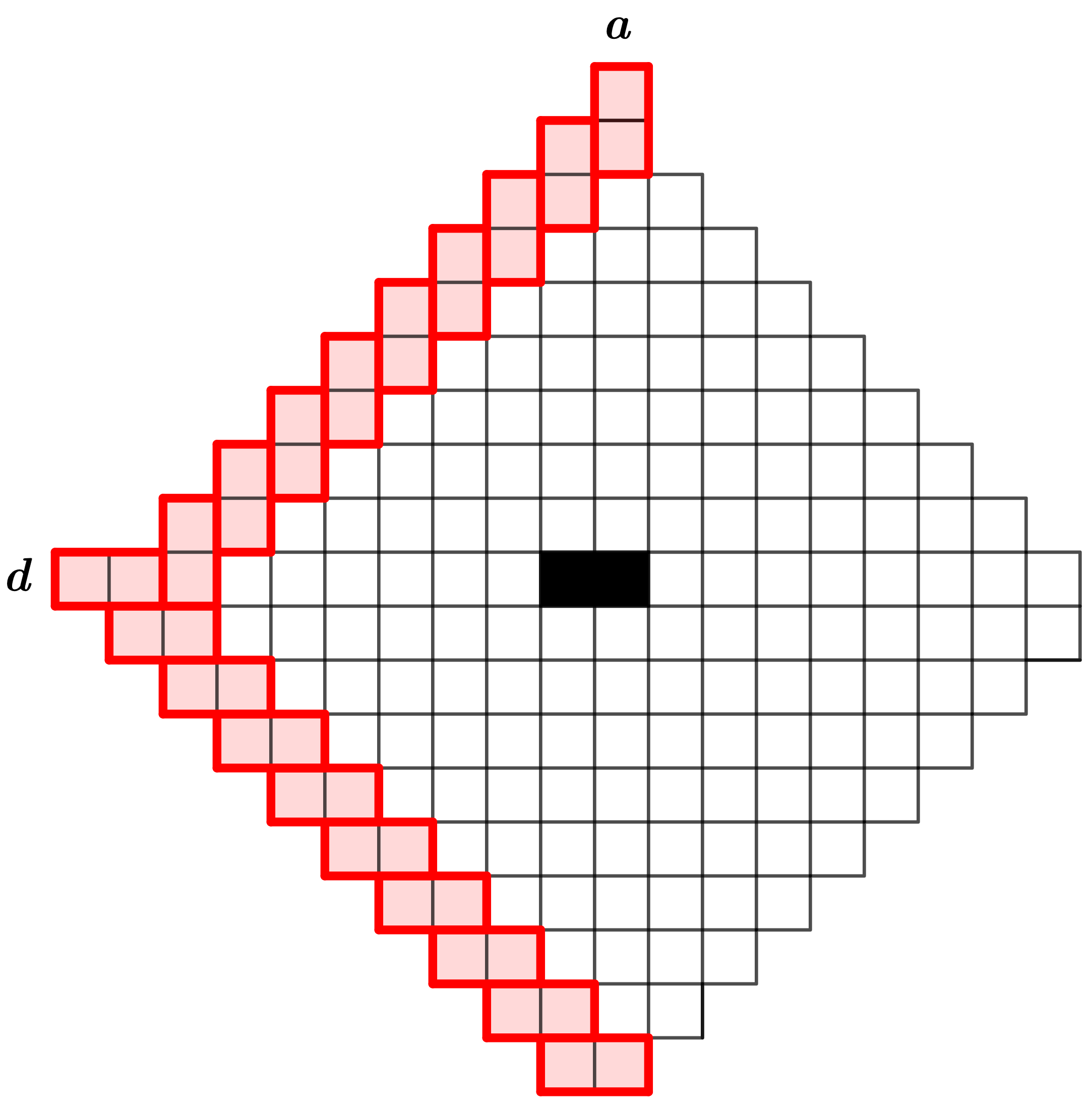}
		\caption{The domain $\mathcal{A}^+\bs \{b,c\}$ with marked forced tiles. The remaining region corresponds to an Aztec diamond of size $4p$ with an horizontal domino shaped hole at position $(-1,0)$. Thus we have $M(\mathcal{A}^+\bs \{b,c\})=2^{2p(4p+1)}\mathbb{P}(-1,0;4p)$.}
		\label{A+bc}
	\end{figure}

    If we now insert our results from Equations (\ref{Erst})-(\ref{Letzt}) into our equation obtained from the Kuo condensation we get
     \begin{align}
       0=&- 2^{8p^2-2p-2}\cdot2^{8p^2+6p+1}X &(I)\notag\\
       &+\Bigg(2^{8p^2+2p}X-2^{8p^2-2p+1}\binom{2p-1}{p}^2\Bigg)\Bigg(2^{8p^2+2p-2}+2^{8p^2-2p}\binom{2p-1}{p}^2\Bigg) &(II) \notag \\
       &+\Bigg(2^{8p^2+2p-1}-2^{8p^2+2p}X+2^{8p^2-2p+1}\binom{2p-1}{p}^2\Bigg)^2,\notag &(III) 
   \end{align}
    which is a quadratic equation in $X$. The proof is finished once we show that $\frac{1}{4}+2^{-4p}\binom{2p-1}{p}^2$ is a solution to this equation (and the second solution does not yield correct probabilities). Hence, in a next step we substitute $X= \frac{1}{4}+2^{-4p}\binom{2p-1}{p}^2$, and we do this for each summand $(I),(II),(III)$ separately.
    \begin{align}
        (I)&=-2^{16p^2+4p-1}\Bigg(\frac{1}{4}+2^{-4p}\binom{2p-1}{p}^2 \Bigg) \notag \\
        &=-2^{16p^2+4p-3}-2^{16p^2-1}\binom{2p-1}{p}^2. \notag
    \end{align}
   
    \begin{align}
        (II)&= \Bigg(2^{8p^2+2p}\Big(\frac{1}{4}+2^{-4p}\binom{2p-1}{p}^2\Big)-2^{8p^2-2p+1}\binom{2p-1}{p}^2\Bigg)\cdot\notag\\ 
        &\hspace{2cm}\cdot\Bigg(2^{8p^2+2p-2}+2^{8p^2-2p}\binom{2p-1}{p}^2\Bigg) \notag \\
        &=\Bigg(2^{8p^2+2p-2}+2^{8p^2-2p}\binom{2p-1}{p}^2-2^{8p^2-2p+1}\binom{2p-1}{p}^2\Bigg)\cdot \notag \\
        &\hspace{2cm}\cdot\Bigg(2^{8p^2+2p-2}+2^{8p^2-2p}\binom{2p-1}{p}^2\Bigg) \notag \\
        &=\Bigg(2^{8p^2+2p-2}-2^{8p^2-2p}\binom{2p-1}{p}^2\Bigg)\Bigg(2^{8p^2+2p-2}+2^{8p^2-2p}\binom{2p-1}{p}^2\Bigg) \notag \\
        &=2^{16p^2+4p-4}-2^{16p^2-4p}\binom{2p-1}{p}^4. \notag
    \end{align}
    
    \begin{align}
        (III)&=\Bigg(2^{8p^2+2p-1}-2^{8p^2+2p}\Bigg(\frac{1}{4}+2^{-4p}\binom{2p-1}{p}^2\Bigg)+2^{8p^2-2p+1}\binom{2p-1}{p}^2\Bigg)^2 \notag \\
        &=\Bigg(2^{8p^2+2p-1}-2^{8p^2+2p-2}-2^{8p^2-2p}\binom{2p-1}{p}^2+2^{8p^2-2p+1}\binom{2p-1}{p}^2\Bigg)^2 \notag \\
        &=\Bigg(2^{8p^2+2p-2}+2^{8p^2-2p}\binom{2p-1}{p}^2\Bigg)^2 \notag \\
        &=2^{16p^2+4p-4}+2^{16p^2-1}\binom{2p-1}{p}^2+2^{16p^2-4p}\binom{2p-1}{p}^4. \notag
    \end{align}

    Thus, for the right-hand-side of our original equation we have
    \begin{align}
        (I)+(II)+(III) &= - 2^{16p^2+4p-3}-2^{16p^2-1}\binom{2p-1}{p}^2 \notag \\
                        &\hphantom{={}} {}+2^{16p^2+4p-4}\hspace{3cm}-2^{16p^2-4p}\binom{2p-1}{p}^4 \notag \\
                        &\hphantom{={}} {}+2^{16p^2+4p-4}+2^{16p^2-1}\binom{2p-1}{p}^2+2^{16p^2-4p}\binom{2p-1}{p}^4 \notag\\
                        &= 0. \notag
    \end{align}
    In conclusion, $\frac{1}{4}+2^{-4p}\binom{2p-1}{p}^2$ is a solution to the equation dictated by Kuo condensation and therefore a possible candidate for $\mathbb{P}(0,0;4p+1)$. The other candidate is the second solution of the quadratic equation. By applying Vieta's Theorem it is easily checked, that the second solution does not yield the correct numbers for the probabilities. (In particular the second solution is always greater than $1/2$ while $\mathbb{P}(0,0;4p+1)<1/2$ due to rotational symmetry.)
    However, this shows that
    \[\mathbb{P}(0,0;1)=\frac{1}{4}+2^{-4p}\binom{2p-1}{p}^2\]
    and this finishes both the proofs of Proposition \ref{Case1} and Theorem \ref{1/4}.
\end{proof}

\section{Remark on asymptotic behaviour}\label{RemA}

For the sake of completeness we should also discuss the asymptotic behaviour of the placement probabilities $\mathbb{P}(l,m;4p+\alpha)$ as $p$ tends to infinity. The result is \textit{just} a remark, since it has already been studied by Cohn, Elkies and Propp in \cite{localstats}. They have proven the following theorem.

\begin{thm}[{\cite[Thm.~1]{localstats}}] \label{statrel}
    Let $U$ be an open set containing the points $(\pm \frac{1}{2},\frac{1}{2})$. If $(x,y)$ is the normalized location of a horizontal domino space, whose left square is coloured black, and $(x,y)\notin U$ then as $n\to \infty$ the placement probability at $(x,y)$ is within $o(1)$ of $\mathcal{P}(x,y)$, where
    \[\mathcal{P}(x,y)=\begin{cases}
        0, &\text{if } \ x^2+y^2\geq \frac{1}{2} \text{ and } y<\frac{1}{2}, \\
        1, &\text{if } \ x^2+y^2\geq \frac{1}{2} \text{ and } y>\frac{1}{2},\\
        \frac{1}{2}+\frac{1}{\pi}\tan^{-1}\Big(\frac{2y-1}{\sqrt{1-2x^2-2y^2}}\Big), &\text{if } \ x^2+y^2< \frac{1}{2} .
    \end{cases}\]
\end{thm}

Here, the term \textit{domino space} refers to the location in the Aztec diamond before a horizontal tile is observed in a random tiling. The theorem yields the asymptotic behaviour of the placement probabilities for the \textit{relative positions} $(x,y)$ in an Aztec diamond of size $n\in \N$ which is normalized, i.e., an original position $(l,m)$ is now located at $(x,y)=(\frac{l}{n},\frac{m}{n})$. However, when $n\to \infty$ for every fixed $l,m\in \Z$ we observe
\[(x,y)=\Big(\frac{l}{n},\frac{m}{n}\Big)\to(0,0).\]
Thus we obtain the following simple corollary.

\begin{cor}\label{placeasymp}
    Let $l,m\in \Z$ be arbitrary and $\alpha\in\{0,1,2,3\}$ such that $l+m\equiv \alpha +1 $~\em{mod}~$2$. Then we have
    \[\mathbb{P}(l,m;4p+\alpha)\underset{p\to \infty}{\longrightarrow} \frac{1}{4}.\]
\end{cor}

\begin{proof}
    Although, it might not be difficult to argue by using the formula in our main result that 
    \[2^{-4p-\alpha}\binom{2p-1}{p}^2f_{l,m,\alpha}(p)\underset{p\to \infty}{\longrightarrow}0\]
    for all the rational functions $f_{l,m,\alpha}(p)$, it is even simpler to just use Theorem \ref{statrel} and evaluate $\mathcal{P}(0,0)$:
    \begin{align}
        \mathcal{P}(0,0)=\frac{1}{2}+\frac{1}{\pi}\tan^{-1}(-1)=\frac{1}{2}-\frac{1}{4}=\frac{1}{4}. \notag
    \end{align}
\end{proof}

\section{Tiling enumeration of Aztec diamonds with $2\times2$-square holes.}\label{dimdef}

In this last section we apply our main result Theorem \ref{1/4} to count domino tilings in an Aztec diamond which has a $2\times2$-square hole whose centre is located at $(l,m)\in \Z^2$.

\begin{cor} \label{squarehole}
    The number of domino tilings of an Aztec diamond of size $4p+\alpha$,\break$\alpha\in\{0,1,2,3\}$  with a $2\times2$-square hole centred at position $(l,m)\in \Z^2$ is equal to
    \[2^{(4p+\alpha+1)(4p+\alpha)/2}\Bigg(\frac{1}{8}+2^{-4p-\alpha-2}\binom{2p-1}{p}^2g_{l,m,\alpha}(p)+2^{-8p-2\alpha}\binom{2p-1}{p}^4 h_{l,m,\alpha}(p)\Bigg),\]
    where $g_{l,m,\alpha}(p)$ and $h_{l,m,\alpha}(p)$ are two rational functions in $p$.
\end{cor}

\begin{proof}
    We use again Kuo condensation, Theorem \ref{Kuo}, to deduce this shape for the formula. Every $2\times2$-square hole consists of four neighbouring squares labelled in clockwise order $a,b,c,d$. Without loss of generality assume that the squares $a$ and $b$ give a horizontal domino whose left square is black and labelled by $a$. Denote by $\mathcal{A}$ the Aztec diamond of size $4p-\alpha$. Kuo condensation tells us
    \begin{align}
        M(\mathcal{A})M(\mathcal{A}\bs\{a,b,c,d\})= M(\mathcal{A}\bs\{a,b\})M(\mathcal{A}\bs\{c,d\})+M(\mathcal{A}\bs\{a,d\})M(\mathcal{A}\bs\{b,c\}). \notag
    \end{align}
    Now, again by rotating the Aztec diamond and viewing it from the right direction we have in our standard notation:
    \begin{align}
        M(\mathcal{A})&=2^{(4p+\alpha+1)(4p+\alpha)/2}, \notag \\
        M(\mathcal{A}\bs\{a,b\})&=2^{(4p+\alpha+1)(4p+\alpha)/2}\mathbb{P}(l,m;4p+\alpha), \notag \\
         M(\mathcal{A}\bs\{c,d\})&=2^{(4p+\alpha+1)(4p+\alpha)/2}\mathbb{P}(-l,-m;4p+\alpha), \notag \\
          M(\mathcal{A}\bs\{a,d\})&=2^{(4p+\alpha+1)(4p+\alpha)/2}\mathbb{P}(-m,l;4p+\alpha), \notag \\
           M(\mathcal{A}\bs\{b,c\})&=2^{(4p+\alpha+1)(4p+\alpha)/2}\mathbb{P}(m,-l;4p+\alpha). \notag
    \end{align}
    Now, applying Theorem \ref{1/4} to the probabilities and inserting this into our equation for Kuo condensation we obtain
    \begin{align}
        M(\mathcal{A}\bs\{a,b,c,d\})&=2^{(4p+\alpha+1)(4p+\alpha)/2}\cdot\notag \\\cdot\Bigg(\Bigg(\frac{1}{4}+2^{-4p-\alpha}&\binom{2p-1}{p}^2f_{l,m,\alpha}(p)\Bigg)\Bigg(\frac{1}{4}+2^{-4p-\alpha}\binom{2p-1}{p}^2f_{-l,-m,\alpha}(p)\Bigg)+\notag \\
        \ \Bigg(\frac{1}{4}+2^{-4p-\alpha}&\binom{2p-1}{p}^2f_{-m,l,\alpha}(p)\Bigg)\Bigg(\frac{1}{4}+2^{-4p-\alpha}\binom{2p-1}{p}^2f_{m,-l,\alpha}(p)\Bigg)\Bigg)= \notag\\
        2^{(4p+\alpha+1)(4p+\alpha)/2}&\Bigg(\frac{1}{8}+2^{-4p-\alpha-2}\binom{2p-1}{p}^2g_{l,m,\alpha}(p)+2^{-8p-2\alpha}\binom{2p-1}{p}^4 h_{l,m,\alpha}(p)\Bigg), \notag
    \end{align}
    where 
    \begin{align}
        g_{l,m,\alpha}(p)&=f_{l,m,\alpha}(p)+f_{-l,-m,\alpha}(p)+f_{-m,l,\alpha}(p)+f_{m,-l,\alpha}(p) \notag \\
        h_{l,m,\alpha}(p)&=f_{l,m,\alpha}(p)f_{-l,-m,\alpha}(p)+f_{-m,l,\alpha}(p)f_{m,-l,\alpha}(p) \notag
    \end{align}
    are rational functions in $p$.
\end{proof}

In particular, we obtain the following complement to Theorem \ref{reflect} by Ciucu.

\begin{cor}
    The Aztec diamond with a central $2\times 2$-square hole has 
    \[2^{2p(4p+1)}\Bigg(\frac{1}{8}+2^{-4p}\binom{2p-1}{p}^2+2^{-8p+1}\binom{2p-1}{p}^4\Bigg),\]
    tilings if the diamond is of size $4p$ and
    \[2^{(2p+1)(4p+1)}\Bigg(\frac{1}{8}+2^{-4p}\binom{2p-1}{p}^2+2^{-8p+1}\binom{2p-1}{p}^4\Bigg),\]
    tilings if the diamond is of size $4p+1$.
\end{cor}

\begin{proof}
    By all the knowledge we have gathered, it is easy for us to compute
    \[f_{0,0,0}(p)=1.\]
    Thus we have
    \[g_{0,0,0}(p)=4\]
    and 
    \[h_{0,0,0}(p)=2.\]
    Inserting this into the formula in Corollary \ref{squarehole} with $\alpha=0$ one gets the first expression.
    On the other hand, we also already computed 
    \[f_{0,0,1}(p)=2\]
    and thus
    \[g_{0,0,1}(p)=h_{0,0,1}(p)=8.\]
    Therefore, Corollary \ref{squarehole} and $\alpha=1$ yield the second formula.
    
\end{proof}

\appendix

\section{Degree bounds for rational parts}
In this appendix we examine again our main objects, Kravchuk polynomials, creation rates and placement probabilities, to find bounds for the degrees of the polynomials in the rational parts. As we will see, it is possible to provide an explicit formula for the denominators of these objects. However, the numerators still remain mysterious even though their degrees are linearly bounded. We start off with an notation for members of a certain family of polynomials which will turn out to be central in the upcoming analysis.
\begin{nota}
    Throughout this section we denote by
    \[N(a,b,A,B;p):=\prod_{i=1}^a(i+p)^{\alpha_i}\prod_{j=1}^b(-2j+1+2p)^{\beta_j}\]
    the polynomial in $p$ where $\alpha_i=\begin{cases}
        2, &\text{if }\ i\leq A, \\ 1, &\text{otherwise,} 
    \end{cases}$ and $\beta_j=\begin{cases}
        2, &\text{if }\ j\leq B, \\ 1, &\text{otherwise.} 
    \end{cases}$
\end{nota}

\begin{rem}
    The degrees of the polynomials $N(a,b,A,B,p)$ are especially simple to determine. Namely, as long as $A\leq a$ and $B\leq b$, we instantly observe 
    \[\deg(N(a,b,A,B,p))=a+b+A+B.\]
\end{rem}

Moreover, we would like to introduce the following notation for the non-negative division by 2 without remainder.
\begin{nota}
    For an integer $z\in\Z$ set
    \[z\divi 2:= \max \bigg(\bigg\lfloor{\frac{z}{2}}\bigg\rfloor,0\bigg).\]
\end{nota}

Our main goal is to prove the following improvement of Theorem~\ref{1/4}.
\begin{thm}[\sc{Advanced 1/4-Phenomenon}] \label{adv1/4}
    Let $l\geq 0$ and $m\in \Z$ be integers and $\alpha\in\{0,1,2,3\}$ such that $l+m\equiv \alpha +1$~\em{mod}~$2$\em. If $m>0$ define 
    \[(l,m_1,\alpha_1),(l,m_2,\alpha_2),\dots,(l,m_k,\alpha_k)\in\{(l,i,\alpha-m+i\ \text{\em mod\em} \ 4)\mid i=1,2,\dots,m\}\]
    to be all those tuples $(l,m',\alpha')$ which fulfil either
    \begin{itemize}
        \item $\alpha'\neq 0$,  $m'\leq l+\alpha'-1$ and $l-m'-1\equiv \alpha$~\em{mod}~$4$ \em or
        \item $\alpha'= 0$,  $m'\leq l+3$ and $l-m'-1\equiv 0$~\em{mod}~$4$. 
    \end{itemize}
    Furthermore, denote by $S_1,S_2,\dots, S_k$ the integers
    \[S_i:= \bigg(\frac{l-m_i+\alpha_i+1}{2}-\alpha_i\bigg)\divi 2.\]
    Then we can write the according placement probability as
    \[\mathbb{P}(l,m;4p+\alpha)=\frac{1}{4}+2^{-4p-\alpha}\binom{2p-1}{p}^2\frac{h_{l,m}(p)}{N(r,s,R,S;p)\cdot\prod_{i=1}^k(-2S_i-1+2p)}\]
    where $N(r,s,R,S;p)=C(l,m,\alpha;p)$ is the explicit denominator of the rational part of $\Cr(l,m;4p+\alpha)$ as presented in Corollary~\ref{advcreation} below and $h_{l,m}(p)$ is a polynomial in $p$ with \[\deg(h_{l,m})\leq\max(l,m)+\bigg\lfloor\frac{\min(l+3,m)}{2}\bigg\rfloor+1.\]

    On the other hand, if $m<0$ let 
    \[(l,m_1,\alpha_1),(l,m_2,\alpha_2),\dots,(l,m_k,\alpha_k)\in\{(l,i,\alpha-m+i\ \text{\em mod\em}\ 4)\mid i=1,2,\dots,m\}\]
    be all those tuples $(l',m',\alpha')$ which fulfil either
    \begin{itemize}
        \item $\alpha'\neq 0$,  $|m'|< l'+\alpha'-1$ and $l+m'+\alpha\equiv 3$~\em{mod}~$4$ \em or
        \item $\alpha'= 0$,  $|m'|<l$ and $l+m'\equiv 0$~\em{mod}~$4$. 
    \end{itemize}
    Furthermore, denote by $R_1,R_2,\dots, R_k$ the integers
    \[R_i:= \frac{l-m_i+\alpha_i-1}{2}\divi 2.\]
    Then we can write the according placement probability as
    \[\mathbb{P}(l,m;4p+\alpha)=\frac{1}{4}+2^{-4p-\alpha}\binom{2p-1}{p}^2\frac{h_{l,m}(p)}{N(r,s,R,S;p)\cdot\prod_{i=1}^k(R_i+1+p)}\]
    where $N(r,s,R,S;p)=C(l,m,\alpha;p)$ is the explicit denominator of the rational part of $\Cr(l,m;4p+\alpha)$ as presented in Corollary~\ref{advcreation} below and $h_{l,m}(p)$ is a polynomial in $p$ with \[\deg(h_{l,m})\leq\max(l,|m|)+\bigg\lfloor\frac{\min(l+3,|m|)}{2}\bigg\rfloor+1.\]
    \end{thm}

This result naturally implies Theorem~\ref{shortadv} from the introduction. \newline

Just as our simpler version of the $1/4$-phenomenon also the proof for Theorem~\ref{adv1/4} works by first showing an advanced version of our Growth Theorem for Kravchuk polynomials, then of creation rates and finally deduce Theorem~\ref{adv1/4}. On the way we use again different recursive properties of our objects to prove our claims. However, the reasoning at each step now becomes much more delicate since it is not enough any more to just deduce the existence of some rational function satisfying our formulas, but we need to verify that these rational parts appear in certain shapes. Mathematically, this is done by analysing algebraic operations much more carefully and working ourselves through a lot of case distinctions. In the end, we will be able to show that the minimal denominator in the rational part of the formula of the $1/4$-phenomenon divides the denominator polynomial given in Theorem~\ref{adv1/4}. Nevertheless, the mere technical nature of the proofs below is the main reason to cover them in the appendix. We start off with the proof of an advanced version of Theorem~\ref{growing1}. 

\subsection{The Advanced Growth Theorem for Kravchuk polynomials}

It turns out that all the denominators of the occurring rational functions are of the shape $N(a,b,A,B,p)$. This already starts with the Kravchuk polynomials. Below, we present an advanced version of our Growth Theorem for Kravchuk polynomials. The proof turns out to be much more technical than the one for Theorem \ref{growing1} but on the other hand it also provides much more detail on the shape of the rational functions.

\begin{thm}[\sc{The Advanced Growth Theorem for Kravchuk polynomials}]\label{growing2}
    Let $a,b\in \Z$ and $\alpha\in \{0,1,2,3\}$. Then we have
    \[\Krav(a+2p,b+2p;4p+\alpha-1)=(-1)^p\binom{2p-1}{p}\frac{G_{a,b,\alpha}(p)\cdot q_{a,b,\alpha}(p)}{N(r(a),s(b),0,0;p)}\]
    where 
    \begin{itemize}
        \item $r(a):=\begin{cases}
             a\divi 2, &\textit{if }\ a>0,\\
              (|a|+\alpha-1)\divi 2, &\textit{if }\ a\leq 0,
        \end{cases}$
        \item $s(b):=\begin{cases}
             (b-\alpha+1)\divi 2, &\textit{if }\ b\geq0,\\
             |b|\divi 2, &\textit{if }\ b< 0,
        \end{cases}$
        \item $G_{a,b,\alpha}(p)$ is a polynomial in $p$ whose explicit formula can be read off Table \ref{advTable} and
        \item $q_{a,b,\alpha}(p)$ is a polynomial in $p$ with $\deg(q_{a,b,\alpha})\leq \min(|a|,|b|)+1$.
    \end{itemize}
\end{thm}

 \begin{table}[]
        \centering
        \begin{tabular}{|m{1.5cm}|m{9cm}|}
        \hline
            $a\geq 0$,\newline $b\geq0$ &$
                \displaystyle{\prod_{i=((b-\alpha)\divi 2)+1}^{ (a-\alpha)\divi 2} (-i+p)\cdot\prod_{j\in\{ a<2j-1\leq b\}}(2j-1+2p)}$\\
            \hline
             $a\leq 0$,\newline $b\geq0$ &$\displaystyle{\prod_{i= ((b-\alpha)\divi 2)+1}^{(a-1)\divi 2} (-i+p)\cdot \prod_{j\in\{ |a|-1+\alpha<2j-1\leq b\}}(2j-1+2p)}$\\
            \hline
             $a\geq 0$,\newline $b\leq0$ &$\displaystyle{\prod_{i=((b-1)\divi 2)+1}^{(a-\alpha)\divi 2} (-i+p)\cdot\prod_{j\in\{ a<2j-1\leq |b|-1+\alpha\}}(2j-1+2p)}$\\ 
            \hline
            $a\leq 0$,\newline $b\leq0$ &$\displaystyle{\prod_{i=((b-1)\divi 2)+1}^{ (a-1)\divi 2} (-i+p)\cdot\prod_{j\in\{ |a|-1+\alpha<2j-1\leq |b|-1+\alpha\}}(2j-1+2p)}$   \\ 
            \hline
        \end{tabular}
        \vskip10pt
        \caption{The polynomial $G_{a,b,\alpha}(p).$ Notice that for every instance either the left or the right product is empty.}
        \label{advTable}
    \end{table}

The proof of Theorem \ref{growing2} follows similar ideas as the one for the original Growth Theorem. However, we need to consider many different cases. First, we show that the structure conjectured in Theorem \ref{growing2} is preserved under the symmetry relation of Kravchuk polynomials.

\begin{lem}\label{sym2}
    If $\Krav(a+2p,b+2p;4p+\alpha-1)$ with $a,b\in \Z$ and $\alpha\in\{0,1,2,3\}$ is of the shape stated in Theorem \ref{growing2} then so is $\Krav(b+2p,a+2p;4p+\alpha-1)$. 
\end{lem}

\begin{proof}
    We will only go through the subcase $a>0, b<0$ with $|a|\geq |b|+\alpha-1$. All other cases follow the same line of computations. The starting point is always the application of Lemma~\ref{kravsum} which tells us
    \begin{align}
        \Krav(b+2p,a+2p;4p+\alpha-1)&= \notag \\ \frac{(a+2p)!(2p-a+\alpha-1)!}{(b+2p)!(2p-b+\alpha-1)!}&\Krav(a+2p,b+2p;4p+\alpha-1). \notag
    \end{align}
    After inserting our assumed structure for $\Krav(a+2p,b+2p;4p+\alpha-1)$ this becomes
    \begin{align}
        &\dfrac{(a+2p)!(2p-a+\alpha-1)!}{(b+2p)!(2p-b+\alpha-1)!}\cdot \notag \\
        &\hphantom{hhhhhhh}{}\textcolor{purple}{(-1)^p\binom{2p-1}{p}}\cdot \dfrac{q(p)\cdot \displaystyle{\prod_{i=((|b|-1)\divi 2)+1}^{(a-\alpha)\divi 2 }}(-i+p) }{\displaystyle{\prod_{i=1}^{a\divi 2}(i+p)\prod_{j=1}^{|b|\divi 2}(-2j+1+2p)}}.\label{ZR1}
    \end{align}
    If we now rewrite the factorials as products and ignore the non-rational growth factor $\textcolor{purple}{(-1)^p\binom{2p-1}{p}}$ we obtain for the rest
    \[\dfrac{q(p)\displaystyle{\prod_{i=b+1}^{a}(i+2p)}\prod_{i=((|b|-1)\divi 2)+1}^{(a-\alpha)\divi 2 }(-i+p)\textcolor{blue}{\prod_{j=(|b|\divi 2)+1}^{(a-\alpha+1)\divi 2}(-2j+1+2p)}}{\displaystyle{\prod_{i=b-\alpha+1}^{a-\alpha}(-i+2p)\underbrace{\textcolor{brown}{\prod_{i=1}^{|b|+\alpha-1}(i+p)}\textcolor{blue}{\prod_{j=1}^{(a-\alpha+1)\divi 2}(-2j+1+2p)}}_{N(r(b),s(a),0,0;p)}\textcolor{brown}{\prod_{i=(|b|+\alpha-1\divi 2)+1}^{a\divi 2}(i+p)}}}\]
    where the coloured products are obtained by splitting and extending the products in the denominator of \eqref{ZR1}, respectively.
    Then, by reordering the single product terms, we arrive at the expression
     \begin{align}
         &\frac{q(p)}{N(r(b),s(a),0,0;p)}\cdot\notag \\
         &{ \frac{\displaystyle{\prod_{i=b+1}^{a}(i+2p)\prod_{i=((|b|-1)\divi 2)+1}^{(a-\alpha)\divi 2 }(-i+p)\prod_{j=(|b|\divi 2)+1}^{(a-\alpha+1)\divi 2}(-2j+1+2p)}}{\displaystyle\prod_{i=b-\alpha+1}^{a-\alpha}(-i+2p)\prod_{i=(|b|+\alpha-1\divi 2)+1}^{a\divi 2}(i+p)}}\notag\\
         &=:\frac{q(p)}{N(r(b),s(a),0,0;p)}\cdot R(p)\notag
     \end{align}
     where we collected all the big products in a factor $R(p)$. It remains to show, that $R(p)$ is equal to $G_{b,a,\alpha}(p)$. Thus, we compute by splitting into positive values for the running variables
     \begin{align}
         &R(p)= \notag\\
         &=\frac{\displaystyle{\textcolor{purple}{\prod_{i=1}^{a}(i+2p)\prod_{i=0}^{|b|-1}(-i+2p)}\prod_{i=((|b|-1)\divi 2)+1}^{(a-\alpha)\divi 2 }(-i+p)\prod_{j=(|b|\divi 2)+1}^{(a-\alpha+1)\divi 2}(-2j+1+2p)}}{\displaystyle\textcolor{blue}{\prod_{i=1}^{a-\alpha}(-i+2p)\prod_{i=0}^{|b|+\alpha-1}(i+2p)}\prod_{i=(|b|+\alpha-1\divi 2)+1}^{a\divi 2}(i+p)}\notag
         \end{align}
    where the marked products correspond to the leftmost products in the numerator and denominator of $R(p)$ respectively. We continue our computation by doing cancellations and obtain
         \begin{align}
         &\frac{\displaystyle{{\prod_{i=|b|+\alpha}^{a}(i+2p)}\prod_{i=((|b|-1)\divi 2)+1}^{(a-\alpha)\divi 2 }(-i+p)\prod_{j=(|b|\divi 2)+1}^{(a-\alpha+1)\divi 2}(-2j+1+2p)}}{\displaystyle{\prod_{i=|b|}^{a-\alpha}(-i+2p)}\prod_{i=(|b|+\alpha-1\divi 2)+1}^{a\divi 2}(i+p)}\notag\\
          &=\frac{\displaystyle{{\prod_{i=|b|+\alpha}^{a}(i+2p)}\prod_{i=((|b|-1)\divi 2)+1}^{(a-\alpha)\divi 2 }(-\textcolor{purple}{2}i+\textcolor{purple}{2}p)\prod_{j=(|b|\divi 2)+1}^{(a-\alpha+1)\divi2}(-2j+1+2p)}}{\displaystyle{\prod_{i=|b|}^{a-\alpha}(-i+2p)}\prod_{i=(|b|+\alpha-1\divi 2)+1}^{a\divi 2}(\textcolor{purple}{2}i+\textcolor{purple}{2}p)}\notag\\
          &=\frac{\displaystyle{\prod_{i=|b|+\alpha}^{a}(i+2p)}}{\displaystyle\prod_{i=(|b|+\alpha-1\divi 2)+1}^{a\divi 2}(\textcolor{purple}{2}i+\textcolor{purple}{2}p)}\cdot\dfrac{\displaystyle{}\prod_{i=((|b|-1)\divi 2)+1}^{(a-\alpha)\divi 2 }(-\textcolor{purple}{2}i+\textcolor{purple}{2}p)\prod_{j=(|b|\divi 2)+1}^{(a-\alpha+1)\divi 2}(-2j+1+2p)}{\displaystyle\prod_{i=|b|}^{a-\alpha}(-i+2p)}\notag
     \end{align}
     which is equal to the product for $G_{b,a,\alpha}(p)$ suggested by Table~\ref{advTable} since in the first fraction all the even factors $(2i+2p)$ cancel within the numerator and the second factor vanishes since its numerator is just the decomposition of the denominator into its even- and odd-indexed factors. What remains is only $G_{b,a,\alpha}(p)$.
\end{proof}
\begin{rem}
    The non-explicit polynomial $q(p)=q_{a,b,\alpha}(p)$ is still dependent on the parameters $a,b\in \Z$ and $\alpha\in\{0,1,2,3\}$. In the representative case illustrating the proof of Lemma~\ref{sym2} we have $q_{a,b,\alpha}(p)=q_{b,a,\alpha}(p)$. This is not always the case. For some occasions we only have $q_{a,b,\alpha}(p)=c\cdot q_{b,a,\alpha}(p)$ for some constant $c\in \Q$. However, it is always true that
    \[\deg(q_{a,b,\alpha})=\deg(q_{b,a,\alpha}).\]
\end{rem}
Now to proceed in proving the Advanced Growth Theorem for Kravchuk polynomials, Theorem~\ref{growing2}, we need to show two results covering the special cases of $b=0$ or $1$. These will serve as initial values for the usual three term recurrence.

\begin{lem} \label{lemB0}
    Let $a\in \Z$. Then $\Krav(a+2p,2p,4p+\alpha-1)$ is of the shape stated in Theorem~\ref{growing2} and the formula of the polynomial $q_{a,0,\alpha}(p)$ is  given explicitly by Table \ref{tabB0}.
     
\end{lem}
    \begin{table}[]
        \centering
        \begin{tabular}{|m{1.5cm}|m{5.5cm}|m{5.5cm}|}
        \hline
           $q_{a,0,\alpha}(p)$& $a> 0$& $a<0$\\
        \hline
         $\vphantom{\displaystyle{}\prod}\alpha=0$ & $(-1)^{(a+1)\divi 2}$ & $(-1)^{|a|\divi 2}$\\
         \hline
         $\alpha=1$& $\begin{cases}
             (-1)^{a/2}\cdot 2p, &\text{ if }\ a \text{ even,}\\
             0, & \text{ if }\ a \text{ odd.}
         \end{cases}$ & $\begin{cases}
             (-1)^{a/2}\cdot 2p, &\text{ if }\ a \text{ even,}\\
             0, & \text{ if }\ a \text{ odd.} \end{cases}$\\
         \hline
         $\vphantom{\displaystyle{}\prod}\alpha=2$\newline $a\geq 2$  & $(-1)^{a\divi 2}\cdot 2p$ & $(-1)^{(|a|+1)\divi 2}\cdot 2p$\\
         \hline
         $\alpha=3$ & $\begin{cases}
             (-1)^{1+a/2}\cdot (-2+2a)p, & \text{ if }\ a\text{ even},\\
             (-1)^{(a-1)/ 2}\cdot 4p, &\text{ if }\ a\text{ odd.}
         \end{cases}$&$\begin{cases}
             (-1)^{|a|/ 2}\cdot (2+2|a|)p, & \text{ if }\ a\text{ even},\\
             (-1)^{(|a|+1)/2}\cdot 4p, &\text{ if }\ a\text{ odd.}
         \end{cases}$\\
         \hline
        \end{tabular}
        \vskip10pt
        \caption{The polynomial $q_{a,0,\alpha}(p).$ Missing instances may be found in Table \ref{tab:1}.}
        \label{tabB0}
    \end{table}

A similar statement is also true for the case when $b=1$.
\begin{lem} \label{lemB1}   
    Let $a\in \Z$. Then $\Krav(a+2p,1+2p,4p+\alpha-1)$ is of the shape as stated in Theorem~\ref{growing2} and the polynomial $q_{a,1,\alpha}(p)$ has an explicit formula provided by Table~\ref{tabB1}.
\end{lem}

\begin{table}[]
        \centering
        \begin{tabular}{|m{1.5cm}|m{5.5cm}|m{5.5cm}|}
        \hline
           $q_{a,1,\alpha}(p)$& $a> 0$& $a<0$\\
        \hline
         $\vphantom{\displaystyle{}\prod}\alpha=0$ & $(-1)^{(a\divi 2)+1}((-1)^a\cdot(2a+1)+2p)$ & $(-1)^{(|a|-1)\divi 2}\cdot((-1)^a(2a+1)+2p)$\\
         \hline
         $\alpha=1$& $(-1)^{(a\divi 2)+1}\cdot\begin{cases}
             a, &\text{ if }\ a \text{ even,}\\
             2, & \text{ if }\ a \text{ odd.}
         \end{cases}$ & $(-1)^{|a|\divi 2}\begin{cases}
             |a|, &\text{ if }\ a \text{ even,}\\
             2, & \text{ if }\ a \text{ odd.} \end{cases}$\\
         \hline
         $\vphantom{\displaystyle{}\prod}\alpha=2$\newline $a\geq 2$  & $(-1)^{(a+1)\divi 2}\cdot 2p$ & $(-1)^{|a|\divi 2}\cdot 2p$\\
         \hline
         $\alpha=3$ & $\begin{cases}
             (-1)^{a/2}\cdot (2+4p), & \text{ if }\ a\text{ even},\\
             0, &\text{ if }\ a\text{ odd.}
         \end{cases}$&$\begin{cases}
             (-1)^{|a|/ 2}\cdot (2p+4p^2), & \text{ if }\ a\text{ even},\\
             0, &\text{ if }\ a\text{ odd.}
         \end{cases}$\\
         \hline
        \end{tabular}
        \vskip10pt
        \caption{The polynomial $q_{a,1,\alpha}(p).$ Missing instances may be found in Table \ref{tab:1}.}
        \label{tabB1}
    \end{table}

\begin{proof}[Proof of Lemma~\ref{lemB0} and Lemma~\ref{lemB1}.] To deduce Lemma~\ref{lemB0} or \ref{lemB1} one simply recalls the explicit expressions of the Kravchuk polynomials for $a,b\in \{0,1\}$ and $\alpha\in\{0,1,2,3\}$ which we deduced to prove Theorem~\ref{growing1} and which are summarised in Table~\ref{tab:1}. These serve as initial values for the three term recurrence of Kravchuk polynomials as described in Proposition~\ref{Kravrec}. Then the 16 cases of Lemma~\ref{lemB0} and Lemma~\ref{lemB1} follow each via a simple induction argument. We illustrate this line of arguments exemplarily by proving the statement for the intriguing looking case where $a>0$, $b=0$ and $\alpha=3$.

The three term recurrence in this case yields
\begin{align}
    \Krav(a+1+2p,&2p,4p+2)=\notag\\
    \frac{2}{a+1+2p}&\Krav(a+2p,2p,4p+2)-\frac{2p-a+3}{a+1+2p}\Krav(a-1+2p,2p,4p+2).\notag
\end{align}
Inserting our induction hypothesis this becomes
\begin{align}
    (-1)^p\binom{2p-1}{p}\Bigg(\frac{2q_{a,0,3}(p)}{a+1+2p}&\frac{\displaystyle{\prod_{i=1}^{(a-3)\divi 2}(-i+p)}}{N(a\divi 2,0,0,0;p)}\notag \\ &-\frac{(2p-a+3)q_{a-1,0,3}(p)}{a+1+2p}\frac{\displaystyle{\prod_{i=1}^{(a-4)\divi 2}(-i+p)}}{N((a-1)\divi 2,0,0,0;p)}\Bigg). \label{eqB01}
\end{align}
We again just ignore the universal factor $(-1)^p\binom{2p-1}{p}$ and focus on the remaining expression. To proceed we need to do a case distinction on the parity of $a$.\\
\textsc{Case 1:} $a$ is even.

This means $(a-4)\divi 2=(a-3)\divi 2$ and $a\divi 2=(a+1)\divi 2$. Via inserting correction terms the expression in the big parenthesis transforms as follows.
\begin{align}
    \frac{2\textcolor{purple}{q_{a,0,3}(p)}}{a+1+2p}&\frac{(-\frac{a-2}{2}+p)\displaystyle{\prod_{i=1}^{(a-3)\divi 2}(-i+p)}}{(-\frac{a-2}{2}+p)\cdot N(a+1\divi 2,0,0,0;p)}\notag \\ &\hphantom{aaaa}{}-\frac{(2p-a+3)\textcolor{purple}{q_{a-1,0,3}(p)}}{a+1+2p}\frac{(\frac{a}{2}+p)(-\frac{a-2}{2}+p)\displaystyle{\prod_{i=1}^{(a-3)\divi 2}(-i+p)}}{(\frac{a}{2}+p)(-\frac{a-2}{2}+p)\cdot N((a-1)\divi 2,0,0,0;p)}\notag\\
    &=\textcolor{blue}{\frac{\displaystyle{\prod_{i=1}^{(a-2)\divi 2}(-i+p)}}{N((a+1)\divi 2,0,0,0;p)}}\Bigg(\frac{2\textcolor{purple}{q_{a,0,3}(p)}-(2p-a+3)(\frac{a}{2}+p)\textcolor{purple}{q_{a-1,0,3}(p)}}{(a+1+2p)(-\frac{a-2}{2}+p)}\Bigg).\notag
\end{align}
Now, the big factor marked in blue is the correct term for $a+1>0$, $b=0$ and $\alpha=3$. It remains to check whether the second factor reduces to $q_{a+1,0,3}(p)$. We insert the formulas for $q_{a,0,3}(p)$ and $q_{a-1,0,3}(p)$ given by induction and obtain for the expression in the parenthesis:
\begin{align}
    &\frac{2\textcolor{purple}{(-2+2a)(-1)^{\frac{a}{2}+1}p}-(2p-a+3)(\frac{a}{2}+p)\textcolor{purple}{(-1)^{\frac{a-2}{2}}\cdot4p}}{(a+1+2p)(-\frac{a-2}{2}+p)}\notag\\
    &\hskip2cm=\frac{(-2+2a)(-1)^{\frac{a}{2}+1}4p-(2p-a+3)(a+2p)(-1)^{\frac{a-2}{2}}\cdot4p}{(a+1+2p)(-a+1+2p)}\notag\\
    &\hskip2cm=\frac{(-2+2a)(-1)^{\frac{a}{2}+1}4p+(2p-a+3)(a+2p)(-1)^{\frac{a}{2}}\cdot4p}{(a+1+2p)(-a+1+2p)}\notag\\
    &\hskip2cm=(-1)^{\frac{a}{2}}4p\cdot\frac{(2-2a)+(2p-a+3)(a+2p)}{(a+1+2p)(-a+2+2p)}\notag\\
    &\hskip2cm=q_{a+1,0,3}(p)\cdot \frac{2 + a - a^2 + 6 p + 4 p^2}{2 + a - a^2 + 6 p + 4 p^2}\notag\\&\hskip2cm=q_{a+1,0,3}(p).\notag
\end{align}
\textsc{Case 2:} $a$ is odd.

This means $(a-3)\divi 2=(a-2)\divi 2$ and $(a-1)\divi 2=a\divi 2$. Hence, when going back to Term~\eqref{eqB01} the expression in the big parenthesis becomes
\begin{align}
    &\frac{\textcolor{purple}{q_{a,0,3}(p)}\displaystyle{\prod_{i=1}^{(a-2)\divi 2}}(-i+p)}{(\frac{a+1}{2}+p)N(a\divi 2,0,0,0;p)}-\dfrac{\textcolor{purple}{q_{a-1,0,3}(p)}(-\frac{a-3}{2}+p)\displaystyle{\prod_{i=1}^{(a-4)\divi 2}(-i+p)}}{(\frac{a+1}{2}+p)N(a\divi 2,0,0,0;p)}\notag\\
    &\hphantom{aaaaaaaa}{}=\frac{\textcolor{purple}{q_{a,0,3}(p)}\displaystyle{\prod_{i=1}^{(a-2)\divi 2}}(-i+p)-\textcolor{purple}{q_{a-1,0,3}(p)}\prod_{i=1}^{(a-2)\divi 2}(-i+p)}{N((a+1)\divi 2,0,0,0;p)}\notag\\
    &\hphantom{aaaaaaaa}=\frac{\displaystyle{\prod_{i=1}^{(a-2)\divi 2}(-i+p)}}{N((a+1)\divi 2,0,0,0;p)}\cdot (\textcolor{purple}{q_{a,0,3}(p)}-\textcolor{purple}{q_{a-1,0,3}(p)})\notag\\
     &\hphantom{aaaaaaaa}=\frac{\displaystyle{\prod_{i=1}^{(a-2)\divi 2}(-i+p)}}{N((a+1)\divi 2,0,0,0;p)}\cdot (\textcolor{purple}{(-1)^{\frac{a-1}{2}}4p}-\textcolor{purple}{(-2+2(a-1))(-1)^{\frac{a-1}{2}+1}p})\notag\\
     &\hphantom{aaaaaaaa}=\frac{\displaystyle{\prod_{i=1}^{(a-2)\divi 2}(-i+p)}}{N((a+1)\divi 2,0,0,0;p)}\cdot (-1)^{\frac{a-1}{2}}(4p+(-4+2a)p)\notag\\
     &\hphantom{aaaaaaaa}=\frac{\displaystyle{\prod_{i=1}^{(a-2)\divi 2}(-i+p)}}{N((a+1)\divi 2,0,0,0;p)}\cdot (-1)^{\frac{a+1}{2}+1}(-2+2(a+1))p\notag\\
     &\hphantom{aaaaaaaa}=\frac{\displaystyle{\prod_{i=1}^{(a-2)\divi 2}(-i+p)}}{N((a+1)\divi 2,0,0,0;p)}\cdot q_{a+1,0,3}(p).\notag
\end{align}
This completes the proof for the identity proposed by Lemma~\ref{lemB0} for the case $a>0$, $b=0$ and $\alpha=3$. We would like to stress again that all other cases are deduced analogously by a similar line of computations. Since they would not add any further insights to the problem at this point, we declare the proof of Lemma~\ref{lemB0} and Lemma~\ref{lemB1} as completed.
\end{proof}

Finally, we are able to prove the Advanced Growth Theorem for Kravchuk polynomials (Theorem~\ref{growing2}). Again, we will not write down all cases covered by Table~\ref{advTable}. However, it is easy to convince oneself that every other case follows analogously.

\begin{proof}[Proof of Theorem~\ref{growing2}]
    We start by observing that it is actually enough to prove the statement for the $|b|\geq |a|$ case, since by Lemma~\ref{sym2} the assertion is true for $(a,b)\in\Z^2$ with $|a|>|b|$ if and only if it is true for the ordered pair $(b,a)$. Moreover, by Lemma~\ref{lemB0} and Lemma~\ref{lemB1} the theorem holds for all $(a,0)$ and all $(a,1)$ with $a\in\Z$. As a consequence, it is enough to show that the three term recursion of Kravchuk polynomials preserves the shape demanded in the theorem. Since then, given $a,b\in \Z$ with $|a|\leq |b|$, we know that the theorem holds for $(b,0)$ and $(b,1)$, therefore also for $(0,b)$ and $(1,b)$ and thus via the three term recursion we conclude validity of the theorem at $(a,b)$ too.

    Before we begin, let us introduce a shorthand notation
    \[K(a,b,\alpha):= \frac{\Krav(a+2p,b+2p,4p+\alpha-1)}{(-1)^p\binom{2p-1}{p}}\]
    which let us ignore the universal growth factor whose existence was already studied in Theorem~\ref{growing1}. Notice that $K(a,b,\alpha)$ is a rational function in $p$ which satisfies the same linear recursions as $\Krav(a+2p,b+2p,4p+\alpha-1)$. In particular we have
    \[K(a+1,b,\alpha)=\frac{-2b+\alpha-1}{a+1+2p}K(a,b,\alpha)-\frac{2p-(a-\alpha)}{a+1+2p}K(a-1,b,\alpha)\]
    and
    \[K(a-1,b,\alpha)=\frac{-2b+\alpha-1}{2p-a+\alpha}K(a,b,\alpha)-\frac{a+1+2p}{2p-a+\alpha}K(a+1,b,\alpha).\]
    First we check the case $b>a>0$.

    Inserting the formulas of Table~\ref{advTable} into the first recursion yields
    \begin{align}
        K(a+1,b,\alpha)&=\notag\\
        &\frac{(-2b+\alpha-1)q_{a,b,\alpha}(p) }{(a+1+2p)N(a\divi 2,s(b),0,0;p)}\prod_{\{j|a<2j-1\leq b\}}(2j-1+2p)\notag\\
        &-\frac{(2p-(a-\alpha))q_{a-1,b,\alpha}(p)}{(a+1+2p)N((a-1)\divi 2,s(b),0,0;p)}\prod_{\{j|a-1<2j-1\leq b\}}(2j-1+2p).\label{cas1}
    \end{align}
    \textsc{Subcase:} $a$ is even.\\
    This implies for parity reasons that firstly $a\divi 2= (a+1)\divi 2$ and secondly we have\break$\{j\mid a<2j-1\leq b\}=\{j\mid a-1<2j-1\leq b\}$. Thus, Expression \eqref{cas1} above turns into
    \begin{align}
        &\frac{(-2b+\alpha-1)q_{a,b,\alpha}(p)}{N((a+1)\divi 2,s(b),0,0;p)}\cdot \frac{\displaystyle{\prod_{\{j\mid a-1<2j-1\leq b\}}(2j-1+2p)}}{(a+1+2p)}\notag\\
        &\hphantom{aaaaaaaa}{}-\frac{(2p-a+\alpha)(\frac{a}{2}+p)q_{a-1,b,\alpha}(p)}{N((a+1)\divi2,s(b),0,0;p) }\cdot\frac{\displaystyle{\prod_{\{j|a-1<2j-1\leq b\}}(2j-1+2p)}}{(a+1+2p)}\notag\\
        &\hphantom{}{}=\frac{{\displaystyle{\prod_{\{j|a+1<2j-1\leq b\}}(2j-1+2p)}}}{N((a+1)\divi2,s(b),0,0;p)}\cdot\notag\\ &\hphantom{aaaaaaaaa}{}\underbrace{\Big((-2b+\alpha-1)q_{a,b,\alpha}(p)-(2p-a+\alpha)(\frac{a}{2}+p)q_{a-1,b,\alpha}(p)\Big)}_{=: q_{a+1,b,\alpha}(p)}\notag
    \end{align}
    with 
    \begin{align}
        \deg(q_{a+1,b,\alpha})&\leq \max(\deg(q_{a,b,\alpha}),\deg(q_{a-1,b,\alpha})+2)\notag \\
        &\leq \max (\min(a,b)+1,\min(a-1,b)+3) \notag\\
        &=\max(a+1,a+2)=a+2=\min(a+1,b)+1.\notag
    \end{align}
    \textsc{Subcase:} $a$ is odd.\\
    This means $(a-1)\divi 2=a\divi 2$ and $\{j\mid a<2j-1\leq b\}\leq \{j\mid a+1<2j-1\leq b\}$. Hence, Expression \eqref{cas1} becomes
    \begin{align}
        &\frac{(-2b+\alpha-1)q_{a,b,\alpha}(p)}{2\big(\frac{a+1}{2}+p\big) N(a\divi 2,s(b),0,0;p)}\prod_{\{j\mid a+1<2j-1\leq b\}}(2j-1+2p)\notag\\
        &\hphantom{aaaaa}{}-\frac{(2p-a+\alpha)q_{a-1,b,\alpha}(p)}{2\big(\frac{a+1}{2}+p\big)N(a\divi 2,s(b),0,0;p)}(a+2p)\prod_{\{j\mid a+1<2j-1\leq b\}}(2j-1+2p)\notag \\
        &=\frac{\displaystyle\prod_{\{j\mid a+1< 2j-1\leq b\}}(2j-1+2p)}{N((a+1)\divi 2, s(b),0,0;p)}\cdot\notag\\
        &\hphantom{aaaaaaaaa}{}\cdot\underbrace{\Big(\frac{1}{2}(-2b+\alpha-1)q_{a,b,\alpha}(p)-\frac{1}{2}(2p-a+\alpha)(a+2p)q_{a-1,b,\alpha}(p)\Big)}_{=:q_{a+1,b,\alpha}(p)}\notag
    \end{align}
    where again $\deg(q_{a+1,b,\alpha})\leq \min(a+1,b)+1$.
    
    Next we check whether the three term relation preserves the demanded shape also for negative $a$. As an illustration for such a situation we analyse the case $0>a>b$. Inserting the proposed terms for $K(a,b,\alpha)$ and $K(a+1,b,\alpha)$ into the second recursion we obtain
    \begin{align}
        &K(a-1,b,\alpha)=\notag\\
        &\frac{(-2b+\alpha-1)q_{a,b,\alpha}(p)}{(-a+\alpha+2p)N((|a|-1+\alpha)\divi 2,s(b),0,0;p)}\prod_{\{j: |a|-1+\alpha<2j-1\leq |b|-1+\alpha\}}(2j-1+2p)\notag\\
        &-\frac{(a+1+2p)q_{a+1,b,\alpha}(p)}{(-a+\alpha+2p)N((|a|-2+\alpha)\divi 2,s(b),0,0;p)}\prod_{\{j: |a|-2+\alpha<2j-1\leq |b|-1+\alpha\}}(2j-1+2p).\notag
    \end{align}
    \textsc{Subcase:} $|a|+\alpha$ is even.\\
    Therefore, $(|a|-1+\alpha)\divi 2= (|a|-2+\alpha)\divi 2$ and \\
    $\{j:|a|+\alpha<2j-1\leq |b|-1+\alpha\}=\{j:|a|+\alpha-1<2j-1\leq |b|-1+\alpha\}$. Thus, the computation above continues to
    \begin{align}
        &\frac{(-2b+\alpha-1)q_{a,b,\alpha}(p)}{2\big(-\frac{a-\alpha}{2}+p\big)N((|a|-2+\alpha)\divi 2,s(b),0,0;p)}\prod_{\{j: |a|+\alpha<2j-1\leq |b|-1+\alpha\}}(2j-1+2p)\notag\\
        &-\frac{(a+1+2p)q_{a+1,b,\alpha}(p)(|a|-1+\alpha+2p)}{2\big(-\frac{a-\alpha}{2}+2p\big)N((|a|-2+\alpha)\divi 2,s(b),0,0;p)}\prod_{\{j: |a|+\alpha<2j-1\leq |b|-1+\alpha\}}(2j-1+2p)\notag\\
        &=\frac{(-2b+\alpha-1)q_{a,b,\alpha}(p)}{2\big(\frac{|a|+\alpha}{2}+p\big)N((|a|-2+\alpha)\divi 2,s(b),0,0;p)}\prod_{\{j: |a|+\alpha<2j-1\leq |b|-1+\alpha\}}(2j-1+2p)\notag\\
        &-\frac{(a+1+2p)q_{a+1,b,\alpha}(p)(|a|-1+\alpha+2p)}{2\big(\frac{|a|+\alpha}{2}+2p\big)N((|a|-2+\alpha)\divi 2,s(b),0,0;p)}\prod_{\{j: |a|+\alpha<2j-1\leq |b|-1+\alpha\}}(2j-1+2p)\notag\\
        &=\frac{\displaystyle{\prod_{\{j: |a|+\alpha<2j-1\leq |b|-1+\alpha\}}(2j-1+2p)}}{N((|a|+\alpha)\divi 2,s(b),0,0;p)}\notag\\
        &\hphantom{aaaaa}{}\cdot\underbrace{\frac{1}{2}\Big(((-2b+\alpha-1)q_{a,b,\alpha}(p)-(a+1+2p)(|a|-1+\alpha+2p)q_{a+1,b,\alpha}(p)\Big)}_{=:q_{a-1,b,\alpha}(p)},\notag
    \end{align}
    where $\deg(q_{a-1,b\alpha})\leq \min(|a-1|,|b|)+1$.\\
    \newline
    \textsc{Subcase:} $|a|+\alpha$ is odd.\\
    This means $(|a|-1+\alpha)\divi 2=(|a|+\alpha)\divi 2$ and\\
    $\{j:|a|-1+\alpha<2j-1\leq |b|-1+\alpha\}=\{j:|a|-2+\alpha<2j-1\leq |b|-1+\alpha\}$. Hence, in this case the recursion formula for $K(a-1,b,\alpha)$ becomes
\begin{align}
    &\frac{(-2b+\alpha-1)(|a|+\alpha+2p)q_{a,b,\alpha}(p)}{(-a+\alpha+2p)N((|a|+\alpha)\divi 2,s(b),0,0;p)}\prod_{\{j: |a|+\alpha<2j-1\leq |b|-1+\alpha\}}(2j-1+2p)\notag\\
        &-\frac{(a+1+2p)(|a|+\alpha+2p)(\frac{-a+1+\alpha}{2}+p)q_{a+1,b,\alpha}(p)}{(-a+\alpha+2p)N((|a|+\alpha)\divi 2,s(b),0,0;p)}\prod_{\{j: |a|+\alpha<2j-1\leq |b|-1+\alpha\}}(2j-1+2p)\notag\\
    &=\frac{\displaystyle{\prod_{\{j: |a|+\alpha<2j-1\leq |b|-1+\alpha\}}(2j-1+2p)}}{N((|a|+\alpha)\divi 2,s(b),0,0;p)}\cdot\notag\\
    &\hphantom{aaaaa}\cdot\underbrace{\Big((-2b+\alpha-1)q_{a,b,\alpha}(p)-(a+1+2p)(\frac{-a+1+\alpha}{2}+p)q_{a+1,b,\alpha}(p)\Big)}_{=:q_{a-1,b,\alpha}(p)},\notag
\end{align}
    with $\deg(q_{a-1,b,\alpha})\leq \min(|a|+1,|b|)$. (Notice, that in the computations above we used $|a|=-a$ several times.)

    All other cases are deduced analogously. This finishes the proof for the Advanced Growth Theorem for Kravchuk polynomials.
\end{proof}

\subsection{Advanced growth for creation rates and divisibility rules}

In a next step, we expand our analysis on degree bounds to both creation rates and placement probabilities. For creation rates, Lemma \ref{Creaform} tells us how to obtain the creation rates from the Kravchuk polynomials. Hence, as a corollary from the Advanced Growth Theorem \ref{growing2} we have the following immediate result.
\begin{cor}[\sc{Advanced growth of creation rates}] \label{advcreation}
    Let $n=4p+\alpha>0$. Suppose $l$ and $m$ are integers with $l+m\equiv n-1$~\em{mod}~$2$ \em and\/ $|l|+|m|\leq n-1$. Set $a:= (l+m+\alpha-1)/2$ and\/ $b:= (l-m+\alpha-1)/2$, then
    \begin{align}
        &\Cr(l,m;4p+\alpha)=\notag\\
&2^{-4p-\alpha+1}\binom{2p-1}{p}^2\frac{G_{a,b,\alpha}(p)\cdot G_{b,a,\alpha}(p)\cdot q_{a,b,\alpha}(p)\cdot q_{b,a,\alpha}(p)}{N(\max[r(a),r(b)],\max[s(a),s(b)],\min[r(a),r(b)],\min[s(a),s(b)];p)}.\notag
    \end{align}
    For all other integers $l$ and $m$ we have $\Cr(l,m;n+1)=0$.
\end{cor}

While we were able to determine a huge part of the numerator of the rational part in the Kravchuk polynomials and the creation rates, the numerator for the placement probabilities remains completely mysterious in the sense that none of its roots are given by an explicit formula for general positions $(l,m)\in \Z$. However, one might determine the numerator at a given location via interpolation methods. This only works if suitable degree bounds for the numerator and denominator polynomial are known. Christian Krattenthaler and Michael Drmota hinted to the author that it is actually enough to find a general bound for the denominator degree. Their observation is summarised in the lemma below.

\begin{lem} \label{degbound}
    Let $l,m\in \Z$ and $\alpha\in\{0,1,2,3\}$ be integers such that $l+m\equiv \alpha-1$~\em{mod}~$2$. Moreover let $Z(p)$ and $D(p)$ be polynomials in $p$ such that $Z(p)/D(p)$ is the rational part of the placement probability $\mathbb{P}(l,m;4p+\alpha)$ as it appears in formula given by Theorem~\ref{1/4}. Then $\deg(Z)\leq \deg(D).$
\end{lem}
    
    \begin{proof}
        First, we remember the asymptotic behaviour of the central binomial coefficient. Namely we have
        \[\binom{2n}{n}\sim \frac{4^n}{\sqrt{\pi n}}\]
        for $n\to \infty$. (See e.g. \cite[p. 35, Eq. (20)]{Assymp} for a reference.) Therefore, we have
        \begin{align}
            2^{-4p-\alpha}\binom{2p-1}{p}^2&=2^{-4p-\alpha}\bigg(\frac{(2p-1)!}{p!(p-1)!}\bigg)^2\notag \\
            &=2^{-4p-\alpha}\frac{p^2}{4p^2}\binom{2p}{p}^2\notag \\
            &\sim 2^{-4p-\alpha-2}\bigg(\dfrac{4^p}{\sqrt{\pi p}}\bigg)^2\notag \\
            &=\dfrac{2^{-\alpha-2}}{\pi p}=\frac{\text{constant}}{p}.\notag
        \end{align}
        Now Theorem \ref{1/4} tells us
        \[\mathbb{P}(l,m;4p+\alpha)=\frac{1}{4}+2^{-4p-\alpha}\binom{2p-1}{p}^2\frac{Z(p)}{D(p)}\]
        for some polynomials $Z(p)=Z_{l,m,\alpha}(p)$ and $D(p)=D_{m,l,\alpha}(p)$. Moreover, Corollary \ref{placeasymp} states
        \[\mathbb{P}(l,m;4p+\alpha)\overset{p\to\infty}{\longrightarrow} \frac{1}{4},\]
        which implies
        \[2^{-4p-\alpha}\binom{2p-1}{p}^2\frac{Z(p)}{D(p)}\longrightarrow 0.\]
        On the other hand, we have already found out that
        \begin{align}
            2^{-4p-\alpha}\binom{2p-1}{p}^2\frac{Z(p)}{D(p)}&\sim \frac{C}{p}\cdot \frac{Z(p)}{D(p)}\notag
        \end{align}
        for some constant $C\in \R$. Now, the expression on the right-hand side tends to $0$ if and only if $\deg(Z)\leq \deg(D)$. This proves the assertion.
    \end{proof}

    From this we can immediately deduce a linear degree bound for the numerator of the rational part in the formula of creation rates.

    \begin{cor}\label{creadegree}
         Let $n=4p+\alpha>0$. Suppose $l$ and $m$ are integers with $l+m\equiv n-1$~\em{mod}~$2$ \em and\/ $|l|+|m|\leq n-1$. Denote by $Z_{l,m,\alpha}(p)$ the numerator of the rational part in in the formula for $\Cr(l,m;4p+\alpha)$ presented in Theorem~\ref{advcreation}. Then we have
         \[\deg(Z_{l,m,\alpha})\leq \begin{cases}
             |l|,&\text{ if }\ |m|\leq |l|+\alpha-1,\\
             |m|,&\text{ if }\ |m|> |l|+\alpha-1.
         \end{cases}\]
    \end{cor}
    \begin{proof}
        The symmetry of the Aztec diamond with respect to the vertical axis translates to an invariance of the placement probabilities with respect to the sign of the parameter $\pm l$. By the definition of $\Cr(l,m;4p+\alpha)$ this extends to creation rates.
        As always set\break$a:= (l+m+\alpha-1)/2$ and\/ $b:= (l-m+\alpha-1)/2$. Now, by Lemma~\ref{Creaform} the creation rates are also symmetric in the parameter $m$. In conclusion, we have
        \[\Cr(l,m,4p+\alpha)=\Cr(-l,m,4p+\alpha)=\Cr(l,-m,4p+\alpha)=\Cr(-l,-m,4p+\alpha).\]
        Thus, we can assume $l,m\geq 0$. This implies $a\geq 0$ and
        \[b\geq 0 \Longleftrightarrow m\leq l+\alpha-1.\]
        Now, since the creation rates are defined as an subtraction between placement probabilities, we observe that the statement of Lemma~\ref{degbound} from before also extends to the rational part of creation rates. Using Theorem~\ref{advcreation} this means
        \begin{align}\deg(Z_{l,m,\alpha})&\leq\deg(N(\max[r(a),r(b)],\max[s(a),s(b)],\min[r(a),r(b)],\min[s(a),s(b)];p))\notag\\
        &=r(a)+r(b)+s(a)+s(b).\notag
        \end{align}
        \textsc{First Subcase:} $m\leq l+\alpha-1$.\\
        In this case we compute
        \begin{align}
            r(a)&+r(b)+s(a)+s(b)\notag\\
            &=(a\divi 2)+(b\divi 2)+(a-\alpha +1\divi 2)+(b-\alpha+1\divi 2)\notag\\
            &=\bigg(\frac{l+m+\alpha-1}{2}\divi 2\bigg)+\bigg(\frac{l-m+\alpha-1}{2}\divi 2\bigg)\notag\\
            &\hphantom{+++}{}+\bigg(\frac{l+m+\alpha-1}{2}-\alpha+1\divi 2\bigg)+\bigg(\frac{l-m+\alpha-1}{2}-\alpha+1\divi 2\bigg)    \notag\\
            &=\bigg(\frac{l+m+\alpha-1}{2}\divi 2\bigg)+\bigg(\frac{l+m-\alpha+1}{2}\divi 2\bigg)\notag\\
            &\hphantom{+++}{}+\bigg(\frac{l-m+\alpha-1}{2}\divi 2\bigg)+\bigg(\frac{l-m-\alpha+1}{2}\divi 2\bigg)    \notag\\
            &\leq (l+m\divi 2)+(l-m\divi 2) \label{Eq. bound}\\
            &\leq 2l\divi 2 = l, \notag
        \end{align}
        where from Expression \eqref{Eq. bound} onwards we used the inequality 
        \[x\divi 2 + y\divi 2\leq (x+y) \divi 2.\]
        \textsc{Second Subcase:} $m> l+\alpha-1$.\\
        Here, we obtain
        {\allowdisplaybreaks
        \begin{align}
             r(a)&+r(b)+s(a)+s(b)\notag\\
             &= \bigg(\frac{l+m+\alpha-1}{2}\divi 2\bigg)+\bigg(\bigg(\frac{|l-m+\alpha-1|}{2}+\alpha-1\bigg)\divi 2\bigg)   \notag\\
             &\hphantom{+++}{}+\bigg(\bigg(\frac{l+m+\alpha-1}{2}-\alpha+1\bigg)\divi 2\bigg)+\bigg(\bigg(\frac{|l-m+\alpha-1|}{2}-1+1\bigg)\divi 2\bigg)\notag\\
             &= \bigg(\frac{l+m+\alpha-1}{2}\divi 2\bigg)+\bigg(\bigg(\frac{-l+m-\alpha+1}{2}+\alpha-1\bigg)\divi 2\bigg)   \notag\\
             &\hphantom{+++}{}+\bigg(\bigg(\frac{l+m+\alpha-1}{2}-\alpha+1\bigg)\divi 2\bigg)+\bigg(\frac{-l+m-\alpha+1}{2}\divi 2\bigg)\notag \\
             &= \bigg(\frac{l+m+\alpha-1}{2}\divi 2\bigg)+\bigg(\frac{l+m-\alpha+1}{2}\divi 2\bigg)   \notag\\
             &\hphantom{+++}{}+\bigg(\frac{-l+m+\alpha-1}{2}\divi 2\bigg)+\bigg(\frac{-l+m-\alpha+1}{2}\divi 2\bigg)\notag\\
             &\leq (l+m\divi 2)+(-l+m\divi 2)\notag\\
             &\leq 2m\divi 2 = m.\notag
        \end{align}}%
        This concludes the proof.
        \end{proof}

    Before we are able to clarify degree bounds for the numerator of the rational part of the placement probabilities we need to further analyse the behaviour of the denominator of the rational part of $\Cr(l,m;4p+\alpha)$ when varying the parameters. We will establish four divisibility rules which will be essential for the proof of the degree bounds in the placement probabilities. But first, we introduce another shorthand notation.
    \begin{nota}
        We denote the denominator of the rational part of $\Cr(l,m;4p+\alpha)$, presented in Theorem~\ref{advcreation}, by $C(l,m,\alpha)$. I.e., we have
        \begin{align}
        C(l,m,\alpha):=N(\max[r(a),r(b)],\max[s(a),s(b)],\min[r(a),r(b)],\min[s(a),s(b)];p).\notag
        \end{align}
    \end{nota}
    \begin{rem}
        For two polynomials of the $N(\dots)$-shape we observe: $N(r,s,R,S;p)$ divides $N(r',s',R',S';p)$ if and only if $r\leq r'$, $s\leq s'$, $R\leq R'$ and $S\leq S'$.
    \end{rem}
    This will be the key-idea to prove the following divisibility-lemmas.
    \begin{lem}\label{div1+}
    Let $l,m\geq 0$ be integers and let $0\neq \alpha\in \{1,2,3\}$. Then we have \break$C(l,m-1,\alpha-1)$ divides $C(l,m,\alpha)$
    unless $m<l+\alpha+1$ and $l-m-1\equiv\alpha $~\em{mod}~$4$.
    \end{lem}

    \begin{proof}
        We use the idea presented in the remark above. 
        Writing $C(l,m,\alpha)$ as $N(r,s,R,S)$ and $C(l,m-1,\alpha-1)$ as $N(r',s',R',S')$ we just need to extract the single arguments and compare them.
        \begin{align}
            r&=\frac{l+m+\alpha-1}{2}\divi 2 \geq \frac{l+m+\alpha-3}{2}\divi 2 = r',\notag \\
            s&=\bigg(\frac{l+m+\alpha-1}{2}-\alpha+1\bigg)\divi 2=\frac{l+m-\alpha+1}{2}\divi 2\notag\\ &\hphantom{++}{}\geq \frac{l+m-\alpha-1}{2}\divi 2=\bigg(\frac{l+m+\alpha-3}{2}-\alpha+1\bigg)\divi 2=s', \notag\\
            R&=\begin{cases}
                \frac{l-m+\alpha-1}{2}\divi 2, &\text{ if }\ m<l+\alpha-1,\\
                \big(|\frac{l-m+\alpha-1}{2}|+\alpha-1\big)\divi 2, &\text{ if }\ m\geq l+\alpha-1, 
            \end{cases} \notag\\
            R'&=\begin{cases}
                \frac{l-m+\alpha-1}{2}\divi 2, &\text{ if }\ m<l+\alpha-1,\\
                \big(|\frac{l-m+\alpha-1}{2}|+\alpha-2\big)\divi 2, &\text{ if }\ m\geq l+\alpha-1, 
            \end{cases} \notag\\
            S&=\begin{cases}
                \big( \frac{l-m+\alpha-1}{2}-\alpha +1\big)\divi 2, &\text{ if }\ m<l+\alpha-1,\\
                \big(|\frac{l-m+\alpha-1}{2}|+1\big)\divi 2, &\text{ if }\ m\geq l+\alpha-1,
            \end{cases}\notag\\
            S'&=\begin{cases}
                 \big(\frac{l-m+\alpha-1}{2}-\alpha +2\big)\divi 2, &\text{ if }\ m<l+\alpha-1,\\
               \big( |\frac{l-m+\alpha-1}{2}|+1\big)\divi 2, &\text{ if }\ m\geq l+\alpha-1.
            \end{cases}\notag
        \end{align}
        Hence, we have $r\geq r'$, $s\geq s'$ and $R\geq R'$ in every case and $S\geq S'$ for $m\geq l+\alpha-1$. If on the other hand $m<l+\alpha-1$ then we have
        \begin{align}
            S<S' &\Longleftrightarrow \bigg(\frac{l-m-\alpha-1}{2}+1\bigg)\divi 2 <\bigg(\frac{l-m-\alpha-1}{2}+2\bigg)\divi 2\notag\\
            &\Longleftrightarrow \frac{l-m-\alpha-1}{2} \text{ is an even integer}\notag\\
            &\Longleftrightarrow l-m\equiv\alpha+1\mod 4.\notag
        \end{align}
    \end{proof}

    \begin{lem}\label{div2+}
        Let $l\geq 0, m\geq 1$ be integers. Then we have for the polynomial\break $\frac{(2p-1)^2}{p^2}C(l,m-1,3;p-1)$ (i.e., the denominator evaluated at $p-1$) that it divides $C(l,m,0;p)$ unless $m\leq l+3$ and $l-m-1\equiv 0$~\em{mod}~$4$.
    \end{lem}
    \begin{proof}
        The first step is to bring all the polynomials in the statement into $N(\dots)$-shape. If $m\leq l+3$ we have
        \begin{align}
            C(l,m-1,3;p)&=\prod_{i=1}^{\frac{l+m+1}{2}\divi 2}(i+p)^{\alpha_i}\prod_{j=1}^{\big(\frac{l+m+1}{2}-2\big)\divi 2}(-2j+1+2p)^{\beta_j},\notag
        \end{align}
        where $\alpha_i=\begin{cases}
            2,&i\leq\frac{l-m+3}{2}\divi 2,\\ 1,&\text{otherwise,}
        \end{cases}$ and  $\beta_j=\begin{cases}
            2,&j\leq\frac{l-m+3}{2}-2\divi 2,\\ 1,&\text{otherwise.}
        \end{cases}$\\Therefore, 
        \[C(l.m-1,3;p-1)=p^2\prod_{i=1}^{\frac{l+m+1}{2}\divi 2 -1} (i+p)^{\alpha_{i+1}}\prod_{j=2}^{\frac{l+m+1}{2}\divi 2}(-2j+1+2p)^{\beta_{j-1}}.\]
        Thus, if $C(l,m-1,3;p)=N(r',s',R',S';p)$ then we have
        \[\frac{(2p-1)^2}{p^2}C(l,m-1,3;p-1)=N(r'-1,s'+1,R'-1,S'+1;p).\] 
        Now, for $C(l,m,0;p)=N(r,s,R,S;p)$ we determine
        \begin{align}
            r&= \frac{l+m-1}{2}\divi 2=\bigg(\frac{l+m+1}{2}-1\bigg)\divi 2\geq \frac{l+m+1}{2}\divi 2-1=r'-1,\notag\\
            s&=\bigg(\frac{l+m-1}{2}+1\bigg)\divi 2=s'+1,\notag\\
            R&=\frac{l-m-1}{2}\divi 2=\bigg(\frac{l-m+3}{2}-2\bigg)\divi 2=\frac{l-m+3}{2}\divi 2-1=R'-1,\notag\\
            S&=\bigg(\frac{l-m-1}{2}+1\bigg)\divi 2=\frac{l-m+1}{2}\divi 2,\notag\\
            S'+1&=\bigg(\frac{l-m+3}{2}-2\bigg)\divi 2+1=\bigg(\frac{l-m-1}{2}+2\bigg)\divi 2=\frac{l-m+3}{2}\divi 2.\notag
        \end{align}
        Hence, $\frac{(2p-1)^2}{p^2}C(l,m-1,3;p-1)$ divides $C(l,m,0;p)$ unless $S'+1>S$. This happens if and only if
        \begin{align}
            \frac{l-m+3}{2}\divi 2&=\bigg(\frac{l-m+1}{2}+1\bigg)\divi 2>\frac{l-m+1}{2}\divi 2\notag
        \end{align}
        which is the case if and only if $\frac{l-m+1}{2}$ is an odd number, i.e., $l-m+1\equiv 2$~\em{mod}~$4$ \em or (equivalently) $l-m-1\equiv 0$~\em{mod}~$4$.\em

        If $m>l+3$ then the only things changing are the values for $R,R',S$ and $S'$. We observe
        \begin{align}
            R&=\bigg(\Bigg|\frac{l-m-1}{2}\Bigg|-1\bigg)\divi 2=\frac{m-l-1}{2}\divi 2,\notag\\
            R'-1&=\bigg(\Bigg|\frac{l-m+3}{2}\Bigg|+2\bigg)\divi 2-1=\frac{m-l-3}{2}\divi 2,\notag\\
            S&=\Bigg|\frac{l-m-1}{2}\Bigg|\divi 2=\frac{m-l+1}{2}\divi 2,\notag\\
            S'+1&=\Bigg|\frac{l-m+3}{2}\Bigg|\divi 2+1=\bigg(\frac{m-l-3}{2}+2\bigg)\divi 2=\frac{m-l+1}{2}\divi 2.\notag
        \end{align}
        In particular, $R'-1\leq R$ and $S'+1=S$, meaning in this case divisibility is achieved.
    \end{proof}
    Lemma~\ref{div1+} and Lemma~\ref{div2+} explain divisibility relations for the numerators of the rational parts of the creation rates when $m\geq 0$. The following two results are their counterparts for negative values of $m$. 
    \begin{lem}\label{div1-}
    Let $l\geq 0$ and $m\leq -1$, and $3\neq \alpha\in\{0,1,2\}$. Then $C(l,m+1,\alpha+1)$ divides $C(l,m,\alpha)$ unless $|m|< l+\alpha-1$ and $l+m+\alpha\equiv 3$~\em{mod}~$4$.
    \end{lem}
    \begin{proof}
        Like before, we simple compare the arguments of these polynomials in their $N(\dots)$-shape. We write $C(l,m,\alpha)=N(r,s,R,S;p)$ and $C(l,m+1,\alpha+1)=N(r',s',R',S';p)$ and obtain
        \begin{align}
            r&=\frac{l-m+\alpha-1}{2}\divi 2=r',\notag\\
            s&=\bigg(\frac{l-m+\alpha-1}{2}-\alpha+1\bigg)\divi 2=\frac{l-m-\alpha+1}{2}\divi 2\notag\\
            &\hphantom{+++}{}\geq \frac{l-m-\alpha-1}{2}\divi 2=\bigg(\frac{l-m+\alpha-1}{2}-\alpha\bigg)\divi 2=s',\notag\\
             R&=\begin{cases}
                \frac{l+m+\alpha-1}{2}\divi 2, & \text{ if }\ |m|<l+\alpha-1,\\
                \big(\big|\frac{l+m+\alpha-1}{2}\big|+\alpha-1\big)\divi 2, &\text{ if }\ |m|\geq l+\alpha-1,
            \end{cases}\notag\\
            S&=\begin{cases}
                \big(\frac{l+m+\alpha-1}{2}-\alpha+1\big)\divi 2, & \text{ if }\ |m|\leq l+\alpha-1,\\
                \big|\frac{l+m+\alpha-1}{2}\big|\divi 2, &\text{ if }\ |m|> l+\alpha-1,
            \end{cases}\notag\\
            R'&=\begin{cases}
                \frac{l+m+\alpha+1}{2}\divi 2, & \text{ if }\ |m|<l+\alpha-1,\\
                \big(\big|\frac{l+m+\alpha+1}{2}\big|+\alpha\big)\divi 2, &\text{ if }\ |m|\geq l+\alpha-1,
            \end{cases}\notag\\
            S'&=\begin{cases}
                \big(\frac{l+m+\alpha+1}{2}-\alpha\big)\divi 2, & \text{ if }\ |m|\leq l+\alpha-1,\\
                \big|\frac{l+m+\alpha+1}{2}\big|\divi 2, &\text{ if }\ |m|> l+\alpha-1.
            \end{cases}\notag
        \end{align}
        Hence, if $|m|>l+\alpha-1$ then (since $m<0$)
        \begin{align}
            R&=\bigg(\bigg|\frac{l+m+\alpha-1}{2}\bigg|+\alpha-1\bigg)\divi 2= \bigg(\bigg|\frac{l+m+\alpha-1+2}{2}\bigg|+\alpha\bigg)\divi 2=R',\notag \\
            S&= \bigg|\frac{l+m+\alpha-1}{2}\bigg|\divi 2= \bigg(\bigg|\frac{l+m+\alpha+1}{2}\bigg|+1\bigg)\divi 2\geq S'.\notag
        \end{align}
        On the other hand, if $|m|< l+\alpha -1$ then still $S=S'$ but 
        \[R=\frac{l+m+\alpha-1}{2}\divi 2=\bigg( \frac{l+m+\alpha+1}{2}-1\bigg)\divi 2< \frac{l+m+\alpha+1}{2}\divi 2=R'\]
        if $l+m+\alpha+1\equiv 0$~\em{mod}~$4$.\em Otherwise we have equality. This proves the lemma.
    \end{proof}
    Finally, we have one last divisibility rule.
    \begin{lem}\label{div2-}
        Let $l\geq 0$ and $m\leq -1$ be integer parameter. Then the polynomial\break$\frac{(p+1)^2}{(2p+1)^2}C(l,m+1,0;p+1)$ divides the polynomial $C(l,m,3;p)$ unless $|m|< l$ and \\$l+m\equiv0$~\em{mod}~$4$.
    \end{lem}
    \begin{proof}
        Again we first set $C(l,m+1,0;p)=N(r',s',R',S';p)$ with
        \begin{align}
            r'&=\frac{l-m-1+0-1}{2}\divi 2=\bigg(\frac{l-m}{2}-1\bigg)\divi 2,\notag\\
            s'&=\bigg(\frac{l-m-2}{2}+1\bigg)\divi 2=\frac{l-m}{2}\divi 2,\notag\\
            R'&=\begin{cases}
                \frac{l+m}{2}\divi 2,&\text{ if }\ |m|< l,\\
                \big(\frac{|l+m|}{2}-1\big)\divi 2,&\text{ if }\ |m|\geq l,
            \end{cases}\notag\\
            S'&=\begin{cases}
                \big(\frac{l+m}{2}+1\big)\divi 2, &\text{ if }\ |m|\leq l,\\ \big|\frac{l+m}{2}\big| \divi 2, &\text{ if }\ |m|>l.
            \end{cases}\notag
        \end{align}
        Thus, 
        \[C(l,m+1,0;p)=\prod_{i=1}^{r'}(i+p)^{\alpha_i}\prod_{j=1}^{s'}(-2j+1+2p)^{\beta_j}\]
        with $\alpha_i=2$ if $i\leq R'$ and $\alpha_i=1$ otherwise and $\beta_j=2$ as long as $j\leq S'$ and $\beta_j=1$ otherwise. Then
        \begin{align}
            C(l,m+1,0;p+1)&=\prod_{i=1}^{r'}(i+1+p)^{\alpha_i}\prod_{j=1}^{s'}(-2j+1+2+2p)^{\beta_j}\notag\\
            &=\prod_{i=2}^{r'+1}(i+p)^{\alpha_{i-1}}\prod_{j=1}^{s'-1}(-2j+1+2p)^{\beta_{j+1}}\cdot (1+2p)^2.\notag
        \end{align}
        Therefore, we have
        \[\frac{(p+1)^2}{(2p+1)^2}C(l,m+1,0;p+1)=N(r'+1,s'-1,R'+1,S'-1;p).\]
        We need to compare this polynomial with $C(l,m,3;p)=N(r,s,R,S;p)$ where 
        \begin{align}
            r&=\frac{l-m+3-1}{2}\divi 2 =\bigg(\frac{l-m}{2}+1\bigg)\divi 2\geq r'+1,\notag \\
            s&=\bigg(\frac{l-m+2}{2}-2\bigg)\divi 2=\bigg(\frac{l-m}{2}-1\bigg)\divi 2\geq s'-1,\notag\\
            R&=\begin{cases}
                \big(\frac{l+m}{2}+1\big)\divi 2, &\text{ if }\ |m|<l+2,\\
                \big(\frac{|l+m+2|}{2}+2\big)\divi 2,&\text{ if }\ |m|\geq l+2,
            \end{cases}\notag \\
            S&=\begin{cases}
                \big(\frac{l+m+2}{2}-2\big)\divi 2, &\text{ if }\ |m|\leq l+2.\\
                \frac{|l+m+2|}{2}\divi 2, &\text{ if }\ |m|>l+2.
                \end{cases}\notag
        \end{align}
        Now, observe that in any case we have $S\geq S'-1$. Also, we have $R\geq R'$ if $|m|\geq l+2$. However, if $|m|<l$ we have 
        \[R=\bigg(\frac{l+m}{2}+1\bigg)\divi 2 < \bigg(\frac{l+m}{2}+2\bigg)\divi 2=\frac{l+m}{2}\divi 2+1=R'+1\]
        if and only if $l+m\equiv 0$~\em{mod}~$4$. \em
        It remains to check the two boundary cases. If $|m|=l$ (i.e., $m=-l$) then
        \[R=1\divi 2=0=R'+1\]
        and the case $|m|=l+1$ is prohibited since $l+m\equiv 0$~\em{mod}~$2$.\em This finishes the proof.
    \end{proof}

\subsection{Degree bounds for placement probabilities}

    We are finally able to talk about degree bounds for the rational part in the formula for placement probabilities. In this section we will conclude the proof of Theorem~\ref{adv1/4}. To advance we introduce the following notation.

    \begin{nota}
        Let $l,m\in\Z$ and $\alpha\in\{0,1,2,3\}$ such that $l+m\equiv \alpha+1$~\em{mod}~$2$.\em Using Theorem~\ref{1/4} we write
        \[\mathbb{P}(l,m;4p+\alpha)=\frac{1}{4}+2^{-4p-\alpha}\binom{2p-1}{p}^2\frac{Z(l,m,\alpha;p)}{D(l,m,\alpha;p)}\]
        where $Z(l,m,\alpha;p)$ and $D(l,m,\alpha;p)$ denote the minimal polynomials in $p$ (with respect to their degree) which satisfy the equation. 
    \end{nota}

    The main idea to prove Theorem~\ref{adv1/4} is to show that $D(l,m,\alpha;p)$ divides $C(l,m,\alpha;p)$ more or less. \textit{"More or less"} means that for some problematic choices of $l$ and $m$ we might need some additional factors. Nevertheless, those additional factors will be given explicitly. Since we know an explicit expression for $C(l,m,\alpha;p)$ this will give rise to a formula for $\mathbb{P}(l,m;4p+\alpha)$ with an explicit denominator in the rational part. Moreover, knowing the denominator and applying Lemma~\ref{degbound} we furthermore establish degree bounds for the unknown numerator polynomial in the formula. 
    The following proposition tells us how to obtain denominators $D(l,m,\alpha;p)$ and rather nice degree bounds for a fixed $l$-coordinate via checking a finite amount of initial instances. 

    \begin{prop}\label{PVA}
        Let $l\geq 0$ and $|m|>l+\alpha-1$. If $D(l,l+\alpha-1,\alpha;p)$ divides\break$C(l,l+\alpha-1,\alpha;p)$ for all valid values of $\alpha$ then we have that $D(l,m,\alpha;p)$ divides $C(l,m,\alpha;p)$ for all $m>l+\alpha-1$ and all valid values of $\alpha$. Also in negative direction: if\break$D(l,-|l+\alpha-1|,\alpha;p)$ divides $C(l,-|l+\alpha-1|,\alpha;p)$ then also $D(l,m,\alpha;p)$ divides $C(l,m,\alpha;p)$ for all $m<-|l+\alpha-1|$ and all valid values of $\alpha$. 
    \end{prop}

    \begin{rem}
        Notice, that whenever $D(l,m,\alpha;p)$ divides $C(l,m,\alpha;p)$ one has
        \[\deg(Z(l,m,\alpha;p))\leq \deg(D(l,m,\alpha;p))\leq\begin{cases}
            l, &\text{ if }\ |m|\leq l+\alpha-1,\\
            |m|, &\text{ if }\ |m|>l+\alpha-1.
        \end{cases}\]
        Here, for the first inequality we applied Lemma~\ref{degbound} and for the second one we used the proof of Corollary~\ref{creadegree}.
    \end{rem}

    \begin{proof}
        First, assume $\alpha\neq 0$. Then 
        \begin{align}
            \mathbb{P}(l,m;4p+\alpha)&=\mathbb{P}(l,m-1;4p+\alpha-1)+\frac{1}{2}\Cr(l,m;4p+\alpha)\notag\\
            &=\frac{1}{4}+2^{-4p-\alpha}\binom{2p-1}{p}^2\Bigg(\frac{2Z(l,m-1,\alpha-1;p)}{D(l,m-1,\alpha-1;p)}+\frac{c_{l,m,\alpha}(p)}{C(l,m,\alpha;p)}\Bigg),\notag
        \end{align}
        where $c_{l,m,\alpha}(p)$ just for now denotes the polynomial for the numerator of the rational part of the creation rates. Now, by induction we assume that $D(l,m-1,\alpha-1;p)$ divides $C(l,m-1,\alpha-1;p)$. Moreover, Lemma~\ref{div1+} tells us that $C(l,m-1,\alpha-1;p)$ divides $C(l,m,\alpha;p)$. Thus, by transitivity $D(l,m-1,\alpha-1;p)$ divides $C(l,m,\alpha;p)$ and the placement probability can be expressed as
        \[\mathbb{P}(l,m;4p+\alpha)=\frac{1}{4}+2^{-4p-\alpha}\binom{2p-1}{p}^2\frac{h_{l,m,\alpha}(p)}{C(l,m,\alpha;p)}\]
        for some polynomial $h_{l,m,\alpha}(p)$. Thus, just by definition of $D(l,m,\alpha;p)$ as the minimal denominator in such an expression, we must already have $D(l,m,\alpha;p)$ dividing $C(l,m,\alpha;p)$.

        Now, if $\alpha=0$ the initial recursion becomes
         \begin{align}
            \mathbb{P}(l,m;4p)&=\mathbb{P}(l,m-1;4(p-1)+3)+\frac{1}{2}\Cr(l,m;4p)\notag\\
            =\frac{1}{4}+2&^{-4(p-1)-3}\binom{2(p-1)-1}{(p-1)}^2\frac{Z(l,m-1,3;p-1)}{D(l,m-1,3;p-1)}+2^{-4p}\binom{2p-1}{p}^2\frac{c_{l,m,0}(p)}{C(l,m,0;p)},\notag\\
            =\frac{1}{4}+2&^{-4p}\binom{2p-1}{p}^2\Bigg(\frac{\frac{1}{2}Z(l,m-1,3;p-1)}{\frac{(2p-1)^2}{p^2}D(l,m-1,3;p-1)}+\frac{c_{l,m,0}(p)}{C(l,m,0;p)}\Bigg).\notag
        \end{align}
        By induction $D(l,m-1,3;p-1)$ divides $C(l,m-1,3;p-1)$. Furthermore, by Lemma~\ref{div2+} we have that $\frac{(2p-1)^2}{p^2}C(l,m-1,3;p-1)$ divides $C(l,m,0;p)$. Hence, the placement probability can be written as
        \[\mathbb{P}(l,m;4p)=\frac{1}{4}+2^{-4p}\binom{2p-1}{p}^2\frac{h_{l,m,0}(p)}{C(l,m,0;p)}\]
        which implies that $D(l,m,0;p)$ actually divides $C(l,m,0;\alpha)$.

        The statement for negative $m$ is proven analogously. However, for the final arguments one needs to use Lemma~\ref{div1-} and Lemma~\ref{div2-} instead of Lemma~\ref{div1+} and Lemma~\ref{div2+}, respectively.
    \end{proof}

    We will use the proposition above to prove degree bounds for positions $(l,m)\in \Z$ along the vertical and horizontal axis. 

    \begin{cor}\label{divVA}
        The polynomial $D(0,m,\alpha;p)$ divides $C(0,m,\alpha;p)$ and $D(1,m,\alpha;p)$ divides the polynomial $C(1,m,\alpha;p)$ for all $m\in \Z$ and all valid values of $\alpha\in\{0,1,2,3\}$.
    \end{cor}
    \begin{proof}
        By Proposition~\ref{PVA} it is enough to compute $D(0,m,\alpha;p)$, $D(1,m,\alpha;p)$,\break$C(0,m,\alpha;p)$ and $C(1,m,\alpha;p)$ for all $m$ with $-\alpha-1\leq m\leq \alpha$ and check whether the according polynomials divide each other. Once one has implemented the recursive steps for constructing $\mathbb{P}(l,m;4p+\alpha)$ that led to the proof of Theorem~\ref{1/4}, this is verified automatically by an computer. 
    \end{proof}

    In a next step, we also check such divisibility properties along positions on the horizontal axis.

    \begin{cor}\label{advm=0}
        The polynomial $D(l,0,\alpha;p)$ divides the polynomial $C(l,0,\alpha;p)$ for all $l\in\Z$ and all valid values of $\alpha\in\{0,1,2,3\}$.
    \end{cor}

    \begin{proof}
        Recall the ideas of the proof of Proposition~\ref{horireduce}: applying counter probability we have
        \begin{align}
            \mathbb{P}(l,0;4p+\alpha)&=1-\mathbb{P}(l-1,-1;4p+\alpha)-\mathbb{P}(1,l-1;4p+\alpha)-\mathbb{P}(0,-l;4p+\alpha).\notag
        \end{align}
        Notice, that
        \[\mathbb{P}(l-1,-1;4p+\alpha)=\mathbb{P}(l-1,0;4p+\alpha+1)-\Cr(l-1,0;4p+\alpha+1)/2.\]
        Hence, we have
        \begin{align}
            \mathbb{P}(l,0;4p+\alpha)&=1-\mathbb{P}(l-1,0;4p+\alpha+1)+\Cr(l-1,0;4p+\alpha+1)/2\notag\\
            &\hphantom{++++}{}-\mathbb{P}(1,l-1;4p+\alpha)-\mathbb{P}(0,-l;4p+\alpha).\notag
        \end{align}
        Therefore, if $\alpha\neq 3$ the rational part of $\mathbb{P}(l,0;4p+\alpha)$ is given by
        \[\frac{-Z(l-1,0,\alpha+1;p)}{D(l-1,0,\alpha+1;p)}+\frac{c_{l-1,0,\alpha+1}(p)}{C(l-1,0,\alpha+1;p)}-\frac{Z(1,l-1,\alpha;p)}{D(1,l-1,\alpha;p)}-\frac{Z(0,-l,\alpha;p)}{D(0,-l,\alpha;p)}\]
        where $c_{l-1,0,\alpha+1}(p)$ denotes again the numerator polynomial of the rational part of the creation rate. By induction we have $D(l-1,0,\alpha+1;p)$ divides $C(l-1,0,\alpha+1;p)$. Moreover, Corollary~\ref{divVA} yields that $D(1,l-1,\alpha;p)$ divides $C(1,l-1,\alpha;p)$ and $D(0,-l,\alpha;p)$ divides $C(0,-l,\alpha;p)$. Thus we are able to work with the explicitly known denominators of of the rational parts of creation rates and the proof for the $(\alpha\neq 3)$-case is finished once we have shown that all of the polynomials $C(l-1,0,\alpha+1;p)$, $C(1,l-1,\alpha;p)$ and $C(0,-l,\alpha;p)$ divide $C(l,0,\alpha;p)$.

        As always for such an analysis we write $C(l,0,\alpha;p)=N(r,s,R,S;p)$ and compute for the arguments:
        \begin{align}
            r&=R=\frac{l+\alpha-1}{2}\divi 2,\notag\\
            s&=S=\bigg(\frac{l+\alpha-1}{2}-\alpha+1\bigg)\divi 2=\frac{l-\alpha+1}{2}\divi 2.\notag
        \end{align}
        (Notice that here we are allowed to assume $l\geq 2$ since we already proved the statement for values $l=0$ and $l=1$ in the previous corollary.) 

        Now for $C(l-1,0,\alpha+1;p)=N(r_1,s_1,R_1,S_1;p)$ we have
        \begin{align}
            r_1&=R_1=\frac{l-1+\alpha+1-1}{2}\divi 2=r=R,\notag\\
            s_1&=S_1=\bigg(\frac{l-1+\alpha}{2}-\alpha\bigg)\divi 2=\frac{l-\alpha-1}{2}\divi 2\leq s=S.\notag
        \end{align}

        For $C(1,l-1,\alpha;p)=N(r_2,s_2,R_2,S_2;p)$ we deduce
        \begin{align}
            R_2\leq r_2&=\frac{l-1+\alpha}{2}\divi 2=r=R,\notag\\
            S_2\leq s_2&=\bigg(\frac{l-1+\alpha}{2}-\alpha+1\bigg)\divi 2\leq\frac{l-\alpha+1}{2}\divi 2=s=S\notag.
        \end{align}
        
        Finally, for $C(0,-l,\alpha;p)=C(0,l,\alpha;p)=N(r_3,s_3,R_3,S_3;p)$ we compute
        \begin{align}
            R_3\leq r_3&=\frac{l+\alpha-1}{2}\divi 2=r=R,\notag\\
            S_3\leq s_3&=\bigg(\frac{l+\alpha-1}{2}-\alpha+1\bigg)\divi 2= s=S.\notag
        \end{align}
        Hence, in this situation dividibility is indeed given.

        However, if $\alpha=3$ then we have
        \begin{align}
            \mathbb{P}(l,0;4p+3)&=1-\mathbb{P}(l-1,0;4(p+1))+\Cr(l-1,0;4(p+1))/2\notag\\
            &\hphantom{++++}{}-\mathbb{P}(1,l-1;4p+3)-\mathbb{P}(0,-l;4p+3)\notag
        \end{align}
        and thus its rational part is given by
        \begin{align}
        \frac{(2p+1)^2}{4(p+1)^2}\bigg(\frac{-Z(l-1,0,0;p+1)}{D(l-1,0,0;p+1)}&+\frac{c_{l-1,0,0}(p+1)}{C(l-1,0,0;p+1)}\bigg)\notag\\
        &-\frac{Z(1,l-1,\alpha;p)}{D(1,l-1,\alpha;p)}-\frac{Z(0,-l,\alpha;p)}{D(0,-l,\alpha;p)}.\notag
        \end{align}
        Again by induction $D(l-1,0,0;p+1)$ divides $C(l-1,0,0;p+1)$. Also, for the third\break and fourth summand not much changes. Therefore, it only remains to check whether\break $\frac{(p+1)^2}{(2p+1)^2}C(l-1,0,0;p+1)$ is a polynomial which divides $C(l,0,3;p)$.

        Writing again $C(l,0,3;p)=N(r,s,R,S;p)$ we now have
        \begin{align}
            r=R&=\frac{l+2}{2}\divi 2,\notag\\
            s=S&=\frac{l-2}{2}\divi 2\notag
        \end{align}
        just by inserting $\alpha =3$ into the formulas from before.

        On the other hand, if we write $C(l-1,0,0,p)=N(r',s',R',S';p)$ we have
        \begin{align}
            R'= r'&=\frac{l-2}{2},\notag\\
            S'= s'&=\frac{l}{2}.\notag
        \end{align}
        Thus, as we already have computed before, we have
        \[\frac{(p+1)^2}{(2p+1)^2}C(l-1,0,0;p+1)=N(r'+1,s'-1,R'+1,S'-1)\]
        which then indeed divides $C(l,0,3;p)$. This finishes the proof.
    \end{proof}

    To summarise, we now know that $D(l,m,\alpha;p)$ divides $C(l,m,\alpha;p)$ for positions $(l,m)\in \Z^2$ along the coordinate axes. In fact, we conjecture that $D(l,m,\alpha;p)$ divides $C(l,m,\alpha;p)$ for all positions $(l,m)\in\Z$ and all valid values of $\alpha$. Experiments strongly support this. However, the problematic cases in Lemmas~\ref{div1+} to \ref{div2-} make it hard to deduce this result. There seem to happen hidden cancellations which are very difficult to track. Nevertheless, one thing we can do is to count those problematic cases and analyse their contribution to a possibly larger denominator of the rational part of $\mathbb{P}(l,m;4p+\alpha)$. This observation finally leads to the proof of Theorem~\ref{adv1/4}, the advanced version of the $1/4$-phenomenon.

\begin{proof}[Proof of Theorem \ref{adv1/4}]
    Corollary~\ref{advm=0} tells us, that 
    \[\mathbb{P}(l,0;4p+\alpha)=\frac{1}{4}+2^{-4p-\alpha}\binom{2p-1}{p}^2\frac{h_{l,0}(p)}{C(l,0,\alpha;p)} \]
    for some polynomial $h_{l,0}$ with $\deg(h_{l,0}(p))\leq \deg(C(l,0,\alpha;p))\leq \max(l,0)=l$. From there on, we would like to increase the value for the $m$-parameter and analyse how the denominator $D(l,m,\alpha;p)$ might change. Lemma~\ref{div1+} and Lemma~\ref{div2+} tell us that\break$C(l,m'-1,\alpha'-1;p)$ divides $C(l,m',\alpha';p)$ and $\frac{(2p-1)^2}{p^2}C(l,m'-1,3;p-1)$ divides\break$C(l,m',0;p)$ unless $(l,m',\alpha')$ equals one of the tuples $(l,m_1,\alpha_1),\dots ,(l,m_k,\alpha_k)$ listed in the theorem. If one of those problematic cases occurs then we read from the proofs of those lemmas that we need to carry along additional factors. In particular, we see that in these cases the previous denominator divides $C(l,m',\alpha';p)\cdot (-2S_i-1+2p)$ in both situations whether $\alpha =0$ or $\alpha\neq 0$. (In the proofs of the lemmas this appeared as the inequality $S'>S$ when $\alpha\neq 0$ or $S'+1>S$ in the $\alpha'=0$ case. However, the difference was marginal, meaning that then $S'=S+1$ or $S'+1=S+1$, respectively. Hence, to obtain divisibility we need to add the $S+1$-factor given by the $S_i$.) 
    Therefore if we argue inductively on the parameter $m$, we have again by the definition of the creation rates
    \begin{align}
        \mathbb{P}&(l,m;4p+\alpha)=\mathbb{P}(l,m-1;4p+\alpha-1)+\frac{1}{2}\Cr(l,m;4p+\alpha)\notag\\
        &=\frac{1}{4}+2^{-4p-\alpha}\binom{2p-1}{p}^2\bigg(\frac{h(p)}{C(l,m-1,\alpha-1;p)\prod_{i=1}^{k'}(-2S_i-1+2p)}+\frac{g(p)}{C(l,m,\alpha;p)}\bigg) \notag
    \end{align}
    for some polynomials $h(p)$ and $g(p)$ and with $k'$ being either equal to $k$ or $k'=k-1$. (Note also that this is shape of formulas is obtained for the $(\alpha\neq0)$-case. However the other one works completely analogously.) Now Lemma~\ref{div1+} (or Lemma~\ref{div2+}) tells us that either $C(l,m-1,\alpha-1;p)$ divides $C(l,m,\alpha;p)$ (and hence $k'=k$) or it divides \\$C(l,m,\alpha;p)\cdot(-2S_k-1+2p)$ (whereas $k'=k-1$). Extending both fractions to the same denominator and adding them yields the formula of the shape proposed in the theorem.

    The formula for negative values of $m$ is shown analogously using Lemma~\ref{div1-} and Lemma~\ref{div2-}. In this situation the problematic tuples $(l,m_1,\alpha_1),\dots,(l,m_k,\alpha_k)$ lead to additional factors corresponding to the $R$-parameter in the polynomials $N(r,s,R,S;p)$. Hence, we would need to carry them along. 

    In conclusion, it only remains to check the degree bound for the numerator polynomial $h_{l,m}(p)$. By Lemma~\ref{degbound} we have
    \begin{align}
        \deg(h_{l,m})&\leq\begin{cases}
             \deg\bigg(N(r,s,R,S;p)\cdot\prod_{i=1}^k(-2S_i-1+2p)\bigg), &\text{ if }\ m>0,\\
              \deg\bigg(N(r,s,R,S;p)\cdot\prod_{i=1}^k(R_i+1+p)\bigg),&\text{ if }\ m<0
        \end{cases} \notag\\
        &=\deg(N(r,s,R,S;p))+k.\notag
    \end{align}
    Now, Corollary~\ref{creadegree} tells us $\deg(N(r,s,R,S;p))\leq \max(l,|m|)$. And to estimate $k$ we observe that, if 
    \[l'-m'-1\equiv\alpha' \mod 4\]
    then 
    \[l'-(m'+1)-1\equiv l-m'-2\equiv \alpha'-1\not\equiv\alpha+1\mod 4,\]
    meaning that if $(l,m',\alpha')$ is a problematic tuple, the next one is not. Since by the first condition we look at at most $l+3$ such tuples, we have that $k\leq \big\lfloor\frac{\min(m,l+3)}{2}\big\rfloor+1$. The same reasoning holds, when altering $m$ to $m-1$ for $m<0$. Thus, this proves the degree bound. 
\end{proof}

This finalises our analysis of the placement probabilities in the Aztec diamond. We suppose that the $1/4$-phenomenon suggests the existence of a much more general symmetry law of placement probabilities of dimers in regular lattices. This however will be the content of future research.
\nocite{Kuo1}

\newpage

\bibliographystyle{plain}
\bibliography{bib}

\end{document}